\documentclass[a4,12pt]{amsart}%
\usepackage{amsfonts,textcomp,amssymb,amsmath,amsthm}
\usepackage{color,ulem}
\usepackage{fullpage}
\usepackage{graphicx}%
\usepackage[active]{srcltx}
\usepackage{mathabx}
\usepackage{dsfont,mathrsfs,stmaryrd,wasysym,amsbsy}

\def \W{\mathscr W}
\def \R{\mathbb R}
\def \E{\mathbb E}
\def \N{\mathbb N}
\def \P{\mathbb P}

\providecommand{\U}[1]{\protect\rule{.1in}{.1in}}

\newtheorem{theorem}{Theorem}

\newtheorem{definition}[theorem]{Definition}

\newtheorem{lemma}[theorem]{Lemma}

\newtheorem{proposition}[theorem]{Proposition}
\newtheorem{remark}[theorem]{Remark}

\newcommand{\ud}{\operatorname{d}\! }

\newcommand{\C}{\mathcal{C}}
\newcommand{\Cd}{\mathcal{C}_2}

\newcommand{\M}{\mathcal{M}}
\renewcommand{\Mc}{\widehat{\mathcal{M}}}

\renewcommand{\L}{\mathcal{L}}
\newcommand{\rp}[1][\alpha]{\mathcal{R}^{#1}}

\newcommand{\cR}{\mathscr R}

\newcommand{\nt}[1][v]{\Theta_T^{\vartheta,\lambda}(#1)}

\begin{document}
\title{Rough paths and 1d SDE with a time dependent distributional drift. Application to polymers.}
\author{Fran\c{c}ois Delarue$^1$ and Roland Diel$^2$}
\date{}
\maketitle

\begin{center}
{\footnotesize Laboratoire J.-A. Dieudonn\'e,  

Universit\'e de Nice Sophia-Antipolis and UMR CNRS 7351, 

Parc Valrose, 06108 Nice Cedex 02, France.}
\end{center}
\footnotetext[1]{\texttt{delarue@unice.fr}}
\footnotetext[2]{\texttt{diel@unice.fr}}

\begin{abstract}
Motivated by the recent advances in the theory of stochastic partial differential equations involving nonlinear functions of distributions, like the Kardar-Parisi-Zhang (KPZ) equation, we reconsider the unique solvability of one-dimensional stochastic differential equations, the drift of which is a distribution, by means of rough paths theory. Existence and uniqueness are established in the weak sense when the drift reads as the derivative of a $\alpha$-H\"older continuous function, $\alpha >1/3$. Regularity of the drift part is investigated carefully and a related stochastic calculus is also proposed, which makes the structure of the solutions more explicit than within the earlier framework of Dirichlet processes. 
\end{abstract}

\section{Introduction}
\label{sec:intro}
Given a family of continuous paths $(\R \ni x \mapsto Y_t(x))_{t \geq 0}$ with values in $\R$, we are interested in the solvability of the stochastic differential equation 
\begin{equation}
\label{eq:18:1:1}
\ud X_{t} = \partial_{x} Y_t(X_{t}) \ud t  + \ud B_{t}, \quad t \geq 0,
\end{equation}
with a given initial condition, where $\partial_{x} Y_t$ is understood as the derivative of $Y_t$ in the sense of distribution and $(B_{t})_{t \geq 0}$ is a standard one-dimensional Wiener process. 

 {When $\partial_{x} Y_t$ makes sense as a measurable function, 
with suitable integrability conditions, 
pathwise existence and uniqueness are known to hold:
See the earlier papers by Zvonkin \cite{zvo:74} and Veretennikov \cite{ver:80} 
when the derivative exists as a bounded function, in which case existence and uniqueness hold globally, 
together with the
more recent result by Krylov and R\"ockner 
\cite{kry:roc:05} when 
$\partial_{x} Y_t$ is in $L^p_{\rm loc}((0,+\infty) \times \R^d)$
for some $p>d+2$
--the equation being set over $\R^d$ instead of $\R$--, 
in which case existence and uniqueness just hold locally; 
see also the Saint-Flour Lecture Notes by Flandoli \cite{fla:10}
for a complete account.} 
In the case when $\partial_{x} Y_t$ only exists as a distribution, existence and uniqueness have been 
{mostly} discussed within the restricted time homogeneous framework. When the field $Y$ is independent of time, $X$ indeed reads as a diffusion process with $(1/2)\exp(-2Y(x)) \partial_{x} (\exp(2Y(x)) \partial_{x})$ as generator. Then, solutions to \eqref{eq:18:1:1} can be proved to be the sum of a Brownian motion and of a process of zero quadratic variation and are thus referred to as \textit{Dirichlet} processes. In this setting, unique solvability can be proved to hold in the weak or strong sense according to the regularity of $Y$, see for example the papers by Flandoli, Russo and Wolf \cite{fla:rus:wol:03,fla:rus:wol:04} on the one hand and the paper by Bass and Chen \cite{bas:che:01} on the other hand. 
{We also refer to the more recent work by Catellier and Gubinelli 
\cite{cat:gub:12}
for the case when $(B_{t})_{t \geq 0}$ is replaced by a general 
rough signal, like the trajectory of a fractional Brownian motion with an arbitrary Hurst parameter.}

In the current paper, we allow $Y$ to depend upon time, making impossible any factorization of the generator of $X$ under a divergence form and thus requiring a more systematic treatment of the singularity of the drift. 
In order to limit the technicality of the paper, the analysis is restricted to the case when the diffusion coefficient in \eqref{eq:18:1:1} 
is $1$, which is already, as explained right below, a really interesting case for practical purposes and which is, anyway, 
somewhat universal because of the time change property of the Brownian motion.
As suggested in the aforementioned paper by Bass and Chen \cite{bas:che:01}, pathwise existence and uniqueness are then no more expected to hold whenever the path $Y_t$ has oscillations of H\"older type with a H\"older exponent strictly less than $1/2$. For that reason, we will investigate the unique solvability of \eqref{eq:18:1:1} in the so-called weak sense by tackling a corresponding formulation of the martingale problem. Indeed, we will consider the case when $Y_t$ is H\"older continuous, the H\"older exponent, denoted by $\alpha$, being strictly greater than $1/3$, hence possibly strictly less than $1/2$, thus yielding solutions to \eqref{eq:18:1:1} of weak type only, that is solutions that are not adapted to the underlying noise $(B_{t})_{t \geq 0}$. At this stage of the introduction, it {must be stressed} that the threshold $1/3$ for the H\"older exponent of the path is exactly of the same nature as the one that occurs in the theory of rough paths.  
{
It is also worth mentioning that a variant of our set-up has just 
been considered by Flandoli, Issoglio and Russo \cite{fla:iss:rus:14}, which
handle the same equation, the dimension of the state space being possibly larger than 
1 but the H\"older exponent of $Y_{t}$ being (strictly) greater than $1/2$.}

Actually, the theory of rough paths will play a major role in our analysis. The strategy for solving \eqref{eq:18:1:1} is indeed mainly inspired by the papers \cite{zvo:74,ver:80,kry:roc:05} we mentioned right above and consists in finding harmonic functions associated with the (formal) generator
\begin{equation}
\label{eq:gene:rough}
\partial_{t} + {\mathcal L}_{t} : = \partial_{t} + \frac{1}{2} \partial_{xx}^2 + \partial_{x} Y_t(x) \partial_{x}. 
\end{equation}
Solving Partial Differential Equations (PDEs) driven by $\partial_{t} + {\mathcal L}_{t}$, say in the standard mild formulation, then requires to integrate with respect to $\partial_{x} Y_t(x)$ (in $x$), which is a non-classical thing. This is precisely the place where the rough paths theory initiated by Lyons (see \cite{lyo:qia:02,lyo:car:lev:07}) comes in: As recently exposed by Hairer in his seminal paper \cite{hai:13} on the KPZ equation
{and in the precursor paper \cite{hai:11} on rough stochastic PDEs}, mild solutions to PDEs driven by $\partial_{t} + {\mathcal L}_{t}$  may be expanded as rough integrals involving the standard heat kernel on the one hand and the `rough' increments $\partial_{x} Y_t$ on the other hand. 
In our case, we are interested in the solutions of the PDE
\begin{equation}
\label{eq:gene:pde}
\partial_{t} u_t(x)+ {\mathcal L}_{t} u_t(x)= f_t(x),
\end{equation}
when set on a cylinder $[0,T] \times \R$, with a terminal boundary condition at time $T>0$, and when driven by a smooth function $f$. Solutions obtained by letting the source term $f$ vary generates a large enough `core' in order to apply the standard martingale problem approach by Stroock and Varadhan 
\cite{str:var:79} and thus to characterize the laws of the solutions to \eqref{eq:18:1:1}. 

Unfortunately, although such a strategy seems quite clear, some precaution is in fact needed. When $\alpha$ is between $1/3$ and $1/2$, which is the typical range of application of Lyons' theory, the expansion of mild solutions as rough integrals involving the heat kernel and the increments of $\partial_{x} Y_{t}$ is not so straightforward. It is indeed not enough to assume that the path
$\R \ni x \mapsto Y_{t}(x)$ has a rough path structure for any given time $t \geq 0$. As explained in detail in Section \ref{se:2:1}, the rough path structures, when taken 
at different times, also interact, asking for the existence, at any time $t \geq 0$, of a `lifted' 2-dimensional rough path with $Y_{t}$ as 
first coordinate. We refrain from detailing the shape of such a lifting right here as it is longly discussed in the sequel. We just mention that, in Hairer \cite{hai:13}, the family $(Y_{t}(x)))_{t \geq 0,x \in \R}$ has a Gaussian structure, which permits to construct the lifting by means of generic results on rough paths for Gaussian processes, see Friz and Victoir \cite{fri:vic:10}. 
Existence of the lifting under more general assumptions is thus a challenging question, which is (partially) addressed in Section \ref{sec:yz}: The lifting is proved to exist in other cases, including that when $\alpha >1/2$
and when
 $(Y_{t}(x))_{t \geq 0,x \in \R}$ is smooth
enough in time (and in particular when it is time homogeneous). Another difficulty is that, contrary to 
Hairer \cite{hai:11,hai:13} in which problems are set on the torus, the PDE is here set on a non-compact domain. This requires an additional analysis of the growth of the solutions in terms of the behavior of $(Y_{t}(x))_{t \geq 0,x \in \R}$ for large values of $\vert x \vert$, such an analysis being essential to discuss the non-explosion of the solutions to \eqref{eq:18:1:1}.

Besides existence and uniqueness, it is also of great interest to understand the specific dynamics of the solutions to \eqref{eq:18:1:1}. Part of the paper is thus dedicated to a careful analysis of the infinitesimal variation of $X$, that is of the asymptotic behavior of $X_{t+h} - X_{t}$ as $h$ tends to $0$. In this perspective, we prove that the increments of $X$ may be split into two pieces: a Brownian increment as suggested by the initial writing of Eq. \eqref{eq:18:1:1} and a sort of drift term, the magnitude of which is of order $h^{(1+\beta)/2}$, for some $\beta >0$ that is nearly equal to $\alpha$. Such a decomposition is much stronger than the standard decomposition of a Dirichlet process into the sum of a martingale and of a zero quadratic variation process. Somehow it generalizes the one obtained by Bass and Chen \cite{bas:che:01} in the time homogeneous framework when $\alpha \geq 1/2$. As a typical example, $(1+\beta)/2$ is nearly equal to $3/4$ when $Y_t$ is almost $1/2$-H\"older continuous, which fits for instance the framework investigated by Hairer \cite{hai:13}. In particular, except trivial cases when the distribution is a true function, integration with respect to the drift term in \eqref{eq:18:1:1} cannot be performed as a classical integration with respect to a function of bounded variation. In fact, since the value of $(1+\beta)/2$ is strictly larger than $1/2$, it makes sense to understand the integration with respect to the drift term as a kind of Young integral, 
in the spirit of the earlier paper \cite{you:36}. We here say `a kind of Young integral' and not `a Young integral' directly since, as we will see in the analysis, it sounds useful to develop a stochastic version of Young's integration, that is a Young-like integration that takes into account the probabilistic notion of adaptedness as it is the case in It\^o's calculus.  

In the end, we prove that, under appropriate assumptions on the regularity of the field $(Y_t(x))_{t \geq 0,x\in \R}$, Eq. \eqref{eq:18:1:1} is uniquely solvable in the weak sense (for a given initial condition) and that the solution reads as 
\begin{equation}
\label{eq:our:stoc:calculus}
\ud X_{t} = b(t,X_{t},\ud t) + \ud B_{t},
\end{equation}
where $b$ maps $[0,+\infty) \times \R \times [0,+\infty)$ to $\R$ and the integral with respect to $b(t,X_{t},\ud t)$ makes sense as a stochastic Young integral, the magnitude of $b(t,X_{t},\ud t)$ being of order $\ud t^{(1+\beta)/2}$.  

The examples we have in mind are twofold. The first one is the so-called `Brownian motion in a time-dependent random environment' or `Brownian motion in a time-dependent random potential'. Indeed, much has been said about the long time behavior of the Brownian motion in a time-independent random potential such as the Brownian motion in a Brownian potential, see for example \cite{and:die:11, bro:86, die:11, hu:shi:98, hu:shi:98b, rus:tru:07, tan:94}. We expect our paper to be a first step forward toward a more general analysis of one-dimensional diffusions in a time-dependent random potential, even if, in the current paper, nothing is said about the long run behavior of the solutions to \eqref{eq:18:1:1}, this question being left to further investigations. As already announced, the second example we have in mind is the so-called Kardar-Parisi-Zhang (KPZ) equation
(see \cite{kar:par:zha:86}), to which much attention has been paid recently, see among others  Bertini and Giacomin 
\cite{ber:gia:97}, Hairer \cite{hai:13}
and Friz and Hairer \cite[Chap. 15]{fri:hai:14}
 about the well-posedness and Amir, Corwin and Quastel 
\cite{ami:cor:qua:11} about the long time behavior. In this framework, $Y$ must be thought as a realization of 
the time-reversed solution of the KPZ equation, that is $Y_t(x)=u(\omega,T-t,x)$, $T$ being positive and $u(\omega,\cdot,\cdot)$ denoting the random 
solution to the KPZ equation and being defined either as in Bertini and Giacomin by means of the Cole-Hopf transform or as in Hairer by means of rough paths theory. Then, Eq. \eqref{eq:18:1:1} reads as the equation for describing the dynamics of the canonical path $(w_{t})_{0 \leq t \leq T}$ on the canonical space ${\mathcal C}([0,T],\R)$ under the polymer measure
\begin{equation*}
\exp \biggl( \int_{0}^T \dot{\zeta}(t,w_{t}) \ud t \biggr) \ud {\mathbb P}(w),
\end{equation*}
where $\dot{\zeta}$ is a space-time white noise and $\P$ is the Wiener measure, the white noise being independent of the realizations of the Wiener process under $\P$. In this perspective, our result provides a \textit{quenched} description of the infinitesimal dynamics of the polymer. 

The paper is organized as follows. We remind the reader of the rough paths theory in Section \ref{se:2:1}. Main results about the solvability of \eqref{eq:18:1:1} are also exposed in Section \ref{se:2:1}. Section \ref{sec:pde} is devoted to the analysis of PDEs driven by the operator \eqref{eq:gene:rough}. In Section \ref{sec:young:sto}, we propose a stochastic variant of Young's integral in order to give a rigorous meaning to 
\eqref{eq:our:stoc:calculus}. We discuss in Section \ref{sec:yz} the construction of the `rough' iterated
integral that makes the whole construction work. Finally, in Section \ref{sec:KPZ}, we explain the connection with the KPZ equation.

\section{General Strategy and Main Results}\label{se:2:1}


Our basic strategy to define a solution to the SDE \eqref{eq:18:1:1} relies on a suitable adaptation of Zvonkin's method for solving SDEs driven by a bounded and measurable drift (see \cite{zvo:74}) and of Stroock and Varadhan's martingale problem
(see \cite{str:var:79}). The main point is to transform the original 
equation into a martingale. For sure such a strategy requires a suitable version of It\^o's formula and henceforth a right notion of harmonic functions for the generator of the diffusion process \eqref{eq:18:1:1}. This is precisely the point where the rough paths theory comes in, on the same model as it does in Hairer's paper for solving the KPZ equation. 

This section is thus devoted to a sketchy presentation of rough paths theory and then to an appropriate reformulation of Zvonkin's method. 

\subsection{Rough paths on a segment}
\label{subse:segment} We start with reminders about rough paths, following Gubinelli's approach in \cite{gub:04}. Given $\alpha\in(0,1]$, $n\in\N \setminus \{0\}$ and a segment ${\mathbb I} \subset \R$, we denote by $\C^{\alpha}({\mathbb I},\R^n)$ the set of $\alpha$-H\"older continuous functions
$f : {\mathbb I} \rightarrow \R^n$ and we define the seminorm 
$$\|f\|_{\alpha}^{\mathbb I} := \sup_{x,y \in {\mathbb I},x\neq y} 
\frac{\vert f(y) - f(x)\vert }{ \vert y-x \vert^{\alpha}} \text{ and the norm }\ldbrack f\rdbrack_{\alpha}^{\mathbb I} := \|f\|_{\infty}^{\mathbb I} +(1\vee\max_{x\in{\mathbb I}}|x|)^{-\frac{\alpha}{2}}\|f\|_{\alpha}^{\mathbb I},$$
with 
$\|f\|_{\infty}^{\mathbb I}
:=
\sup_{x \in {\mathbb I}} |f(x)|$ and $a \vee b= \max(a,b)$. Note that the factor 
$(1\vee\max_{x\in{\mathbb I}}|x|)^{-\alpha/2}$ is somewhat useless and could be replaced by $1$ at this stage of the paper. 
Actually it will really matter in the sequel,  
when considering paths over the whole line. 
Similarly, we denote by $\Cd^{\alpha}({\mathbb I},\R^{n})$ the set of functions $\cR$ from ${\mathbb I}^2$ to $\R^n$ such that $\cR(x,x)=0$ for every $x$ and with finite norm $\|\cR\|_{\alpha}^{\mathbb I}:=\sup_{x,y \in {\mathbb I},x\neq y}\{
|\cR(x,y)|/|y-x|^\alpha\}$. (Functionals defined on the product space $\R^2$ will be denoted by calligraphic letters). 

For $\alpha\in(1/3,1]$, we call $\alpha$-rough path (on ${\mathbb I}$) a pair $(W,{\mathscr W})$ where $W\in\C^\alpha({\mathbb I},\R^n)$ and $\W\in\Cd^{2\alpha}({\mathbb I},\R^{n^2})$ such that, for any indices $i,j\in\{1,\dots,n\}$,
 the following relation holds:
  \begin{align}\label{intcond}
  	\W^{i,j}(x,z)-\W^{ij}(x,y)-\W^{ij}(y,z)=(W^i(y)-W^i(x))(W^j(z)-W^j(y)), \quad x \leq y \leq z. 
  \end{align}
We then denote by ${\mathcal R}^{\alpha}({\mathbb I},\R^n)$ the set of $\alpha$-rough paths; we will often only write {${\boldsymbol W}$} for the rough path $(W,\W)$.  The quantity $\W^{i,j}(x,y)$ must be understood as a value for the iterated integral (or cross integral) \hbox{``$\int_x^y(W^i(z)-W^i(x))\ud W^j(z)$''} of $W$ with respect to itself
(we will also use the tensorial product ``$\int_x^y(W(z)-W(x)) \otimes \ud W(z)$'' to denote the product between coordinates). When $\alpha = 1$, such 
an integral exists in a standard sense. When $\alpha > 1/2$, it exists as well, but in the so-called Young sense (see \cite{you:36,lyo:car:lev:07} and Lemma \ref{lem:young} below). When $\alpha \in (1/3,1/2]$, which is the typical range of values in rough paths theory, there is no more a canonical way to define the cross integral and it must be given a priori in order to define a proper integration theory with respect to $\ud W$. In that framework, condition \eqref{intcond} imposes some consistency in the behavior of $\W$ when intervals of integration are concatenated. Of course, $\W$ plays a role in the range $\alpha \in (1/3,1/2]$ only, but in order to avoid any distinction between the cases $\alpha \in (1/3,1/2]$ and $\alpha \in (1/2,1]$, we will refer to the pair $(W,\W)$ in both cases, even when $\alpha>1/2$, in which case $\W$ will be just given by the iterated integral of $W$. 

Given ${\boldsymbol W} \in\rp({\mathbb I},\R^n)$ as above, the point is then to define the integral ``$\int_{x}^y v(z) \ud W(z)$'' of some function 
$v$ (from ${\mathbb I}$ into itself) with respect to the coordinates of $\ud W$ for some $[x,y] \subset {\mathbb I}$. When $v$ belongs to ${\mathcal C}^{\beta}({\mathbb I},\R)$, for 
$\beta > 1-\alpha$, Young's theory applies, without any further reference to the second-order structure $\W$ of 
{${\boldsymbol W}$}. When $\beta \leq 1- \alpha$, Young's theory fails, but, in the typical example when $v$ is $W-W(x)$ itself (or one coordinate of $W-W(x)$), the integral is well-defined as it is precisely given by $\W$. In order to benefit from the second-order structure of $\W$ for integrating a more general $v$, the increments 
of $v$ must actually be structured in a similar fashion to that of $W$. This motivates the following notion (which holds whatever the sign of $\alpha+\beta -1$ is): For 
$\beta\in(1/3,1]$, a path $v$ is said to be $\beta$-controlled by $W$ if $v\in\C^\beta({\mathbb I},\R)$ and there is a function $\partial_Wv\in\C^{\beta}({\mathbb I},\R^n)$ such that the remainder term 
\begin{equation}
 \label{eq:28:12:14:1}
{\mathscr R}^v(x,y) :=
v(y) - v(x) - 
\partial_{W} v(x) \bigl(W(y) - W(x)\bigr), \quad x,y \in {\mathbb I},
\end{equation}
is in $\Cd^{2 \beta'}({\mathbb I},\R)$, with $\beta' := \beta \wedge 1/2$.
In the above right-hand side, $\partial_{W} v(x)$ reads as a row vector -as it is often the case for gradients- and
  $(W(y)-W(x))$ as a column vector.
   Although $\partial_{W} v$ may not be uniquely defined, 
we will sometimes write $v$ for $(v,\partial_{W} v)$
when no confusion is possible on the value of $\partial_{W} v$.
  For instance, any function $v \in \C^{2\beta'}({\mathbb I},\R)$ is 
  $\beta$-controlled by $W$, a possible (but not necessarily unique) 
   choice for the `derivative' $\partial_{W} v$ being 
   $\partial_{W} v \equiv 0$. 

We are then able to define the integral of a function $v$ controlled by $W$ (see {\cite{gub:04,hai:11,hai:13}}): 
\begin{theorem}\label{defintrp}
	Given $\alpha,\beta\in(1/3,1]$, let ${\boldsymbol W} \in\rp({\mathbb I},\R^n)$ be a rough path and $v \in \C^\beta({\mathbb I},\R)$
		 be a path $\beta$-controlled by $W$. For two reals $x<y$ in ${\mathbb I}$, consider the compensated (vectorial) Riemann sum:
	$$
	S(\Delta) :=\sum_{i=0}^{N-1} \Bigl\{ 
	v(x_{i}) \bigl(W(x_{i+1})-W(x_{i}) \bigr)+ \partial_Wv(x_i) \W(x_{i},x_{i+1})
	\Bigr\}
	$$
	where $\Delta=(x=x_0<\dots<x_N=y)$ is a partition of $[x,y]$ (above $\partial_{W} v(x_{i})$ is a row vector and 
	$\W(x_{i},x_{i+1})$ a matrix). Then, as the step size $\pi(\Delta)$ of the partition 
	tends to 0, $S(\Delta)$ 
	converges to a limit, denoted by $\int_x^y v(z) \ud W(z)$,  independent of the choice of the approximating partitions.  Moreover,	there is a constant $C=C(n,\alpha,\beta)$ such that,  
	\begin{equation}
	\label{eq:thm:1}
	\begin{split}
	&\left|\int_x^y v(z) \ud W(z) - v(x) \bigl(W(y)-W(x)\bigr) - \partial_Wv(x) \W(x,y)\right|
	\\
	&\hspace{15pt}\leq C\Bigl(\|\W\|_{2\alpha}^{[x,y]}
	\|\partial_Wv\|_{\beta}^{[x,y]} \vert y-x \vert^{2\alpha+\beta}
	+\|W\|_{\alpha}^{[x,y]}
	\|\cR^{v}\|_{{2\beta'}}^{[x,y]}
	\vert y -x \vert^{\alpha+2\beta'} 
	 \Bigr). 
	\end{split}
	\end{equation}
\end{theorem}
Observe that, with our prescribed range of values for $\alpha$ and $\beta$, the exponents $2\alpha+\beta$ and
$\alpha +2 \beta'$ are (strictly) greater than 1. 
 This observation is crucial to prove the convergence of $S(\Delta)$ as the step size tends to $0$. 
 When $v$ is any arbitrary function in
${\mathcal C}^{2\beta'}({\mathbb I},\R)$, 
 Definition \ref{defintrp}
applies and the integral of $\int_{x}^y v(z) dW(z)$
coincides with the Young integral.
Notice also that, most of the time, we shall work with $\beta < \alpha$.   

We now address the problem of stability of the integral with respect to $W$.  Replacing $((v,\partial_{W}v),{\boldsymbol W})$ by a sequence of smooth approximations $((v^n,
\partial_{W^n} v^n),{\boldsymbol W}^{n})_{n \geq 1}$,
a question is to decide whether the (classical) integrals of the $(v^n)_{n \geq 1}$'s with respect to the approximated paths are indeed close to the rough integral of $v$ with respect to $W$. 
Actually, it is true if 
\begin{itemize}
\item[(i)]
the convergence of ${\boldsymbol W}^n$ to ${\boldsymbol W}$ holds in the sense of rough paths, that is 
$\ldbrack W-W^{n}\rdbrack_{\alpha}^{\mathbb I} + \| \W-\W^{n}\|_{2\alpha}^{\mathbb I}$ tends to $0$ as $n$ tends to the infinity ($\W^n$ standing for the true iterated integral of $W^n$), in which case we say that 
the rough path {${\boldsymbol W}$} (or $(W,\W)$) is  geometric; 
\item[(ii)] {the convergence of $(v^n,\partial_{W^n} v^n)$ to $(v,\partial_{W} v)$ holds
in the sense of controlled paths}, that is 
$\ldbrack v-v^n \rdbrack_{\beta}^{\mathbb I}+ 
\ldbrack \partial_{W}v-\partial_{W^n}v^n \rdbrack_{\beta}^{\mathbb I} + 
\| \cR^v - \cR^{v^n}  \|_{2\beta'}^{\mathbb I}$ tends to $0$ as $n$ tends to the infinity.
\end{itemize} 
\subsection{Time indexed families of rough paths}
It is well-guessed that, in order to handle \eqref{eq:18:1:1}, we have in mind to choose $W(x)=Y_{t}(x)$, $x \in \R$, 
and to apply rough paths theory at any fixed time $t \geq 0$ (thus requiring 
to choose ${\mathbb I} = {\mathbb R}$ and subsequently  to extend the notion of rough paths to the whole $\R$, which will be done in the next paragraph). Anyhow a difficult aspect for handling \eqref{eq:18:1:1} is precisely that $(Y_{t}(x))_{t \geq 0,x \in \R}$ is time dependent. If it were time homogeneous, part of the analysis we provide here would be useless: we refer for instance to
\cite{fla:rus:wol:03,fla:rus:wol:04,bas:che:01}. From the technical point of view, the reason is that, in the homogeneous framework, the analysis of the generator of the process $X$ reduces to the analysis of a standard one-dimensional ordinary differential equation. Whenever coefficients depend on time, the connection with ODEs boils down, thus asking for non-trivial refinements. From the intuitive point of view, time-inhomogeneity makes things much more challenging as the underlying differential structure in space varies at any time: In order to integrate with respect to $\partial_{x} Y_{t}(x)$ in the rough paths sense, the second-order structure of the rough paths must be defined first and it is well-understood that it is then time-dependent as well. This says that the problem consists of a time-indexed family of rough paths, but, a priori (and unfortunately), it is not clear whether defining the rough paths time by time 
is enough to handle the problem. Actually, as we explain below, it may not be enough as the rough paths structures interact with one another, thus requiring additional assumptions on $(Y_{t}(x))_{t \geq 0,x \in \R}$. 

As above, we first limit our exposition of time-dependent rough paths to the case when $x$ lives in a segment ${\mathbb I}$. 
For some time horizon $T>0$, 
and for $\alpha,\gamma>0$, we define the following (semi-)norms for continuous functions $f:[0,T)\times {\mathbb I}\to\R^n$ and 
${\mathscr M}:  [0,T)  \times {\mathbb I}^{2}\to\R^{n}$:
  \begin{align*}
  	\|f\|_{\gamma,\alpha}^{[0,T) \times {\mathbb I}} :=\sup_{x,y\in{\mathbb I}, x\neq y, \atop0\leq s<t< T}\frac{|f_t(y)-f_s(x)|}{|t-s|^{\gamma}+|y-x|^{\alpha}}
  	\quad \text{and}\quad  	\|{\mathscr M}\|_{0,\alpha}^{[0,T) \times {\mathbb I}} :=\sup_{x,y\in {\mathbb I}, x\neq y\atop 0\leq t< T}\frac{\left|{\mathscr M}(t,x,y)\right|}{|y-x|^{\alpha}},
  \end{align*}
  with the convention that 
  $\|f\|_{0,\alpha}^{[0,T) \times {\mathbb I}}
  = \sup_{0 \leq t < T} \| f \|_{\alpha}^{\mathbb I}$,  
  together with 
  $$\ldbrack f\rdbrack_{\gamma,\alpha}^{[0,T) \times {\mathbb I}}:=\|f\|_{\infty}^{[0,T) \times {\mathbb I}}+(1\vee\max_{x\in{\mathbb I}}|x|)^{-\frac{\alpha}{2}}\|f\|_{\gamma,\alpha}^{[0,T) \times {\mathbb I}}.$$ We then define the spaces 
 $\C^{\gamma,\alpha}([0,T)\times {\mathbb I},\R^n)$ and $\Cd^{\gamma,\alpha}([0,T) \times {\mathbb I},\R^{n})$ accordingly. 

For $\alpha\in(1/3,1]$, we call time dependent $\alpha$-rough path a family of rough paths $
({\boldsymbol W}_{t})_{0 \leq t <T}=
(W_{t},\W_{t})_{0 \leq t < T}$  
where $W\in\C([0,T)\times {\mathbb I},\R^n)$ and $\W\in\C([0,T)\times {\mathbb I}^2,\R^{n^2})$ such that, for any $t\in[0,T)$, the pair $(W_t,\W_t)$ is an $\alpha$-rough path and
\begin{equation}
\label{eq:09:10:1}
\|(W,\W)\|^{[0,T) \times {\mathbb I}}_{0,\alpha}:=\sup_{t\in[0,T)} \bigl\{ \|W_t\|_{\alpha}^{\mathbb I}+\|\W_t\|_{2\alpha}^{\mathbb I} \bigr\} <\infty.
\end{equation}

We denote by ${\mathcal R}^{\alpha}([0,T) \times {\mathbb I},\R^n)$ the set of time-dependent $\alpha$-rough paths endowed with the seminorm $\| \cdot \|^{[0,T) \times {\mathbb I}}_{0,\alpha}.$  For $\beta\in(1/3,1]$, we then say that 
$v\in\C([0,T)\times {\mathbb I},\R)$ is $\beta$-controlled by the paths $(W_{t})_{0 \leq t < T}$ if $v\in\C^{\beta/2,\beta}([0,T) \times {\mathbb I},\R)$ and there exists a function $\partial_W v\in\C^{\beta/2,\beta}([0,T) \times {\mathbb I},\R^n)$ such that, for any $t \in [0,T)$, the remainder below is in ${\mathcal C}_{2}^{2\beta'}({\mathbb I},\R^n)$:
  \begin{equation}
  \label{eq:31:10:7}
  	{\mathscr R}^{v_t}(x,y) := v_t(y)-v_t(x)-\partial_Wv_{t}(x)
	\bigl(W_t(y)-W_t(x) \bigr), \quad x,y \in {\mathbb I}.
  \end{equation}

%

\subsection{Rough paths on the whole line}
\label{subse:whole:line}
So far, we have only defined rough paths (or time dependent rough paths) on segments. As Eq. \eqref{eq:18:1:1} is set on the whole space, we must extend the definition to $\R$, 
the point being to specify the behavior at infinity of the underlying (rough) paths and of the corresponding controlled functions. 

When the family $(Y_{t}(x))_{t \geq 0,x \in \R}$ is differentiable in $x$, a sufficient 
condition
to prevent a blow-up in \eqref{eq:18:1:1} is to require $(\partial_{x} Y_{t}(x))_{t \geq 0,x \in \R}$ to be at most of linear growth in $x$. In our setting, $(Y_{t}(x))_{t \geq 0, x \in \R}$ is singular and it makes no sense to discuss the growth of its derivative. The point is thus to control the growth of the local H\"older norm of $(Y_{t}(x))_{t \geq 0,x \in \R}$ together with (as shown later) the growth of the local H\"older norm of the associated iterated integral.  

This motivates the following definition. For $\alpha\in(1/3,1]$ and $\chi >0$, we call 
$\alpha$-rough path (on ${\mathbb \R}$) with rate $\chi$ a pair 
${\boldsymbol W}=(W,{\mathscr W})$ such that, for any $a\geq 1$, the restriction of $(W,\W)$ to $[-a,a]$ is 
in ${\mathcal R}^{\alpha}([-a,a])$, and 
\begin{equation}
\label{eq:kappa:chi}
\kappa_{\alpha,\chi}\bigl(W,\W) := 
\sup_{a\geq1} \frac{\|W\|_{\alpha}^{[-a,a]}}{a^\chi} + \frac{\|\W\|_{2\alpha}^{[-a,a]}}{a^{2\chi}} < \infty.
\end{equation}
We denote by ${\mathcal R}^{\alpha,\chi}(\R,\R^n)$ the set of all such $(W,\W)$. 

This definition extends to time-dependent families of rough paths. Given $T>0$, we say that
$(W_{t},\W_{t})_{0 \leq t <T}$ belongs to  
${\mathcal R}^{\alpha,\chi}([0,T) \times \R,\R^n)$ if 
\begin{equation}
\label{eq:30:10:2}
\kappa_{\alpha,\chi}\bigl((W_{t},\W_{t})_{0 \leq t <T}\bigr) := 
\sup_{t\in[0,T)}\kappa_{\alpha,\chi}\bigl(W_t,\W_t) < \infty.
\end{equation}

In a similar way, we must specify the admissible growth of the functions that are controlled by rough paths on the whole $\R$. A comfortable framework is to require exponential bounds. 
Given $(W,\W) \in 
{\mathcal R}^{\alpha,\chi}(\R,\R^n)$ and $\vartheta \geq 1$, we say that a function $v: \R \rightarrow \R$ is in ${\mathcal B}^{\beta,\vartheta}(\R,W)$
for some $\beta \in (1/3,1]$ if, for any segment ${\mathbb I} \subset \R$, 
the restriction of $v$ to ${\mathbb I}$ is 
$\beta$-controlled by $W$
and
\begin{equation}
\label{eq:31:10:3}
\Theta^{\vartheta}(v) := \sup_{a \geq1} \Bigl[ e^{- \vartheta a} \Bigl( \ldbrack v \rdbrack_{\beta}^{[-a,a]}
+ \tfrac{1}{2} \ldbrack \partial_{W} v \rdbrack_{\beta}^{[-a,a]} 
 + a^{-\beta'}\| \cR^v \|_{2\beta'}^{[-a,a]}
 \Bigr) \Bigr]
< \infty.
\end{equation}
Abusively, we omit the dependence upon $\partial_{W} v$
in $\Theta^{\vartheta}(v)$.
Similarly, for $(W_{t},\W_{t})_{0 \leq t <T} \in 
{\mathcal R}^{\alpha,\chi}([0,T) \times \R,\R^n)$, 
a function $v : [0,T) \times \R \rightarrow \R$ is in ${\mathcal B}^{\beta,\vartheta}([0,T) \times \R,W)$
if, for any $a \geq 1$, its restriction to $[0,T) \times [-a,a]$ is 
$\beta$-controlled by $(W_{t})_{0 \leq t < T}$ and,
for some $\lambda > 0$,
\begin{equation*}
\begin{split}
&\Theta^{\vartheta,\lambda}_T(v) 
:=
\sup_{a\geq1 \atop t \in [0,T)} \Bigl[ \frac{
 \ldbrack v \rdbrack_{\beta/2,\beta}^{[t,T) \times [-a,a]}
+ \tfrac12 \ldbrack \partial_{W} v \rdbrack_{\beta/2,\beta}^{[t,T) \times [-a,a]}
+ 
 \lambda^{\frac{\beta-\alpha}8}
 ( a^{-\beta'} \wedge (T-t)^{\beta'/2})
 \| \cR^{v_{t}}\|_{2\beta'}^{[-a,a]}}
{E_{T}^{\vartheta,\lambda}(t,a)}  
 \Bigr],
\end{split}
\end{equation*}
is finite, 
with $
E_T^{\vartheta,\lambda}(t,a):=\exp[\lambda(T-t)+\vartheta  a(1 + T-t)]$
(it  
 reflects
the backward nature of \eqref{eq:gene:pde}).
Note that the set ${\mathcal B}^{\beta,\vartheta}([0,T) \times \R,W)$ does not depend on $\lambda$, but 
that $\Theta^{\vartheta,\lambda}_{T}(v)$ does.

By Theorem \ref{defintrp}, we can easily obtain a control of the integral $\int v_t\ud Y_t$ by the norm $\nt$:
\begin{lemma}\label{majint}
	Assume $\beta \leq \alpha$. 
	Then, there exists a constant $C=C(n,\alpha,\beta)$, 
	such that, for any $\vartheta,\lambda,a\geq1$, any $v\in{\mathcal B}^{\beta,\vartheta}([0,T) \times \R,W)$ 
	and any $(t,x,y)\in[0,T)\times[-a,a]^2$,
\begin{equation*}
\begin{split}
&\left|\int_x^y \bigl(v_t(z)-v_t(x) \bigr)\ud W_t(z)\right|
\leq C \lambda^{\frac{\alpha-\beta}8}\kappa_{\alpha,\chi}\bigl(W_{t},\W_{t}\bigr)\Theta_{T}^{\vartheta,\lambda}(v) E_{T}^{\vartheta,\lambda}(t,a)
\times 
{\mathscr D}(t,a,y-x)
\\
&\left|\int_x^yv_t(z)\ud W_t(z)\right|\leq C
\lambda^{\frac{\alpha-\beta}8} \kappa_{\alpha,\chi}\bigl(W_{t},\W_{t}\bigr)
\Theta_{T}^{\vartheta,\lambda}(v) E_{T}^{\vartheta,\lambda}(t,a)
\times \bigl[ |y-x|^\alpha a^\chi +
{\mathscr D}(t,a,y-x) \bigr],
\end{split}
\end{equation*}
with ${\mathscr D}(t,a,z) := 
|z|^{2\alpha}a^{2\chi}+|z|^{2\alpha+\beta} a^{2\chi+\frac{\beta}{2}}
+ 
\vert z \vert^{{\alpha + 2\beta'}} a^{\chi}
(a^{{\beta'}} 
+(T-t)^{-{\frac{\beta'}{2}}})$. 
\end{lemma}

\subsection{Enlargement of the rough path structure}
We now discuss how the time dependent rough path structures of the drift $(Y_{t}(x))_{t \geq 0,x \in \R}$ interact 
with one another as time varies. 

Formally the generator associated with \eqref{eq:18:1:1} reads ${\mathcal L} = \partial_{t} + \partial_{x}( Y_{t}(x)) \partial_{x}
+ (1/2) \partial_{xx}^2$. This suggests that, on $[0,T) \times \R$, harmonic functions (that is zeros of the generator) read as 
\begin{align*}
u_t(x)=&P_{T-t} u_T(x)  +\int_t^T\int_{\R} p_{r-t}(x-z)\partial_x u_r(z)\ud Y_r(z)\ud r,
\quad x \in \R,
\end{align*}
where $p$ denotes the standard heat kernel
and $P$ the standard heat semi-group (so that $P_{t} f(x) = \int_{\R} p_{t}(x-y) f(y) \ud y$). In the case when the boundary condition of the function $v$ is given by 
$u_{T}(x)=x$, a formal expansion of $\partial_{x} u_{t}(x)$ in the neighborhood of $T$ gives 
\begin{equation*}
\begin{split}
\partial_{x} u_t(x)&\sim 1 + \int_t^T\int_{\R} \partial_{x} p_{r-t}(x-z) \ud Y_r(z)\ud r
\\
&\hspace{15pt}+ \int_{t}^T \int_{\R}\partial_{x} p_{r-t}(x-z) \biggl\{ \int_{r}^T \int_{\R}
\partial_{x}p_{s-r}(z-u) \ud Y_s(u)\ud s\biggr\} \ud Y_r(z)\ud r + \dots
\end{split}
\end{equation*}
In the first order term of the expansion, the space integral makes sense as the singularity can be transferred from $Y_{r}$ onto 
$\partial_{x} p_{r-t}(x-z)$, provided the integration by parts is licit: using the approximation argument discussed above, it is indeed licit when the rough path is geometric. In order to give a sense to this first order term, the point is to check that the resulting singularity in time is integrable, which is addressed in Section \ref{sec:pde}. Unfortunately, the story is much less simple for the second order term. Any formal integration by parts leads to a term involving a `cross' integral between the space increments of $Y$, but taken at different times: This is the place where rough structures, indexed by different times, interact. 

We refrain from detailing the computations at this stage of the paper and feel more convenient to defer their presentation to Section \ref{sec:pde} below. Basically, the point is to give, at any time $t \in [0,T)$, a sense to the integral $\int_{x}^y 
Z_{t}^T(z) \ud Y_{t}(z)$, where, for all $t \in [0,T)$ and $x \in \R$,
\begin{equation}
\label{eq:Z}
Z_t^T (x) =
\int_{t}^T \partial^2_{x} P_{r-t} Y_{r}(x) dr =
\int_t^T\int_{\R} \partial^2_x p_{r-t}(x-z)(Y_r(z)-Y_r(x))\ud z \ud r.
\end{equation}
Assuming that $\sup_{0 \leq t <T} \sup_{x,y \in \R} [ (1+ \vert x \vert^{\chi} + \vert y \vert^{\chi})^{-1} 
\| Y_{t} \|_{\alpha}^{[x,y]}]$ is finite (for some $\chi >0$), the above integral is well-defined 
(see Lemma \ref{lemmaZ} below). In order to make sure that the cross integral of $Z_{t}^T$ with respect to $Y_{t}$ exists, the point is to assume that the pair $(Y_{t},Z_{t}^T)$ can be lifted up to a rough path of dimension 2, which is to say that there exists some 
${\mathscr W}^T$ with values in $\R^4$ such that $((Y,Z^T),{\mathscr W}^T)$ is an $\alpha$-time dependent rough path, for 
some $\alpha>1/3$. We will see in Section \ref{sec:yz} conditions under which such a lifting ${\mathscr W}^T$ indeed exists.

\subsection{Generator of the diffusion and related Dirichlet problem}
We now provide some solvability results for the Dirichlet problem driven 
by 
$\partial_{t} + {\mathcal L}_{t}$
in \eqref{eq:gene:rough}:

%

\begin{definition}
\label{defsolution}
Given $Y\in\C([0,T)\times\R,\R)$, assume that there exists $\W^T$ such that $(W^T=(Y,Z^T),\W^T)$ 
belongs to ${\mathcal R}^{\alpha,\chi}([0,T) \times \R,\R^2)$ with $\alpha>1/3$
and $\chi >0$. Given
$f\in {\mathcal C}([0,T]\times \R,\R)$,
with $\sup_{a \geq 1} 
\sup_{0 \leq t \leq T}
e^{-\vartheta a} 
\| f_{t} \|_{\infty}^{[-a,a]} < \infty$ for some $\vartheta \geq 0$,  
a function $u : [0,T] \times \R \rightarrow \R$ is a mild solution on $[0,T]\times \R$ to the problem $\mathcal{P}(Y,f,T)$:
$$\L v=f, \quad \textrm{with} \quad \L v := \partial_tv+\L_t v,$$ 
if $u$ is continuously differentiable with respect to $x$,
with $\partial_{x} u \in {\mathcal B}^{\beta,\vartheta}([0,T) \times \R,W^T)$
for some $\beta \in (1/3,1]$, and satisfies
\begin{equation}
\label{eq:mildpde}
\begin{split}
u_t(x) &=  P_{T-t} u_T(x) - \int_t^T P_{s-t}f_s(x)\ud s
+  \int_t^T\int_{\R} \partial_{x} p_{r-t}(x-y) \int_{x}^y  
\partial_x u_r(z)  \ud Y_r(z) \ud y \ud r.  
\end{split}
\end{equation}
\end{definition}
Finiteness of the integrals over $\R$ will be checked in Lemma 
\ref{lem:28:10:1b} below.
We also emphasize that a notion of weak solution could be given as well, but we won't use it.  

\begin{remark}
\label{rem:ipp}
When $(W^T,\W^T)$ is {geometric}, the last term in the right-hand side coincides (by integration by parts, which is made licit by approximation by smooth paths) 
with 
$\int_t^T\int_{\R} p_{r-t}(x-y) \partial_x u_r(y)\ud Y_r(y) \ud r$,
which reads as a more `natural formulation' of a mild solution and which is, by the way, the formulation used in 
Sections 3.1 and 3.2 of Hairer \cite{hai:13} {for investigating the KPZ equation
and in Section 3.1 of Hairer \cite{hai:11} for handling rough SPDEs}. The formulation \eqref{eq:mildpde} seems a bit more tractable as it splits into two well separated parts the rough integration and the regularization effect of the heat kernel. Once again, both are equivalent in the geometric (and in particular smooth) setting. 
\end{remark}
Here is a crucial result in our analysis (the proof is postponed to Section \ref{sec:pde}):

\begin{theorem}\label{mildsolution}
Let $Y$ be as in Definition
\ref{defsolution}. Then, for any $f \in {\mathcal C}([0,T] \times \R,\R)$
and  $u^T\in \mathcal{C}^1(\R,\R)$, with 
\begin{equation}
\label{eq:29:12:14:1}
\begin{split}
m_{0} :=
\sup_{a \geq 1} \Bigl[ 
e^{-\vartheta a} 
\Bigl( 
\sup_{0 \leq t \leq T}
\bigl( \| f_{t} \|_{\infty}^{[-a,a]} +  \| f_{t} \|_{\gamma}^{[-a,a]} \bigr)
 + 
 \| (u^T)' \|_{\infty}^{[-a,a]}  + \| (u^T)' \|_{\beta}^{[-a,a]} 
\Bigr) \Bigr]<\infty,
\end{split}
\end{equation}
for some $\vartheta \geq 1$, $\gamma >0$ and $\beta \in (1/3,\alpha)$, with $\beta > 2 \chi$,
there is a unique solution,
in the space 
${\mathcal B}^{\beta,\vartheta}([0,T) \times \R,W^T)$,
of the problem $\mathcal{P}(Y,f,T)$ 
with $u_T=u^T$ as terminal condition.

Letting $m:=\max [1, T,\vartheta,m_{0},\kappa_{\alpha,\chi}(W^T,\W^T) ]$, we can find
$C=C(m,\alpha,\beta,\chi)$, such that, 
for any $(t,x) \in [0,T] \times \R$,
\begin{equation}
\label{eq:31:10:1}
\vert u_{t}(x) \vert + \vert \partial_{x} u_{t}(x) \vert \leq C \exp \bigl( C \vert x \vert \bigr), 
\end{equation}
 and for any $(s,t,x,y) \in [0,T]^2 \times \R^2$,
\begin{equation}
\label{eq:31:10:2}
\begin{split}
&\vert u_{t}(x) - u_{s}(x) \vert \leq C \exp \bigl( C \vert x \vert \bigr) \vert t-s \vert^{\frac{1+\beta}{2}},  
\\
&\vert \partial_{x}u_{t}(x) - \partial_{x} u_{s}(y) \vert \leq 
C \exp \bigl( C [ \vert x \vert \vee \vert y\vert ] \bigr) 
\bigl( \vert t-s \vert^{\frac{\beta}{2}} + \vert x- y \vert^{\beta} \bigr). 
\end{split}
\end{equation}
\end{theorem}

We now address the question of stability of mild solutions under mollification of 
$(W^T,\W^T)$. We call a mollification of $W^T$ `physical' 
if it consists in mollifying $Y$ in $x$ first -the mollification is then smooth in $x$, the derivatives being continuous in space and time- and then in replacing $Y$ by its mollified version in \eqref{eq:Z}. Denoting by 
$Y^n$ the mollified path at the $n$th step of the mollification, the resulting $Z^{n,T}$ is smooth in $x$, 
the derivatives being also continuous in space and time. This permits to define the corresponding pair $(W^{n,T},\W^{n,T})$ directly. In that specific \textit{geometric} setting, we claim (once again, the proof is deferred to Section \ref{sec:pde}):
\begin{proposition}
\label{mildsolution:approx}
In the same framework as in Theorem \ref{mildsolution}, assume that 
the rough path $(W^T,\W^T)$ is geometric in the sense that there exists a sequence of smooth paths
$(Y^n)_{n \geq 1}$ such that the corresponding sequence  
$(W^{n,T}=(Y^n,Z^{n,T}))_{n \geq 1}$
satisfies
\begin{enumerate}
\item $\|(W^T-W^{n,T},\W^T-\W^{n,T})\|_{0,\alpha}^{[0,T) \times {\mathbb I}}$ tends to $0$ as $n$ tends to $\infty$ for 
any segment ${\mathbb I} \subset {\mathbb R}$, where $\W^{n,T}_{t}(x,y) = \int_{x}^y 
(W_{{t}}^{n,T}(z) - W_{{t}}^{n,T}(x)) \otimes \ud W_{{t}}^{n,T}(z)$,
for $t \in [0,T)$ and $x,y \in \R$,
\item $\sup_{n \geq 1}\kappa_{\alpha,\chi}((W^{n,T}_{t}, \W_{t}^{n,T} )_{0 \leq t \leq T})$ is finite
(see \eqref{eq:30:10:2} for the definition of $\kappa_{\chi}$).
\end{enumerate} 
Then,
 the associated solutions $(u^n)_{n \geq 1}$ (in the sense of Definition \ref{defsolution}) and their gradients 
$(v^n = \partial_{x} u^n)_{n \geq 1}$
 converge towards 
$u$ and $v= \partial_{x} u$ uniformly on compact subsets of $[0,T] \times \R$. 
\end{proposition}
{It is worth noting that
each $u^n$ is actually a classical solution of
the PDE \eqref{eq:gene:pde}
driven by $Y^n$ instead of 
$Y$. The reason is that, in the characterization 
\eqref{eq:mildpde} of a mild solution (in the rough sense), 
the rough integral coincides with a standard Riemann integral when 
$W^n$ is smooth. We refer to \cite[Corollary 3.12]{hai:11} for another 
use of this (quite standard) observation.}
\subsection{Martingale problem}
We now define the martingale problem 
associated with \eqref{eq:18:1:1}:
\begin{definition}
\label{def:mart:pb}
Let $T_{0}>0$ and $x_{0}\in\R$.
Given $Y\in\C([0,T_{0})\times\R,\R)$, assume that, for any $0 \leq T \leq T_{0}$, there exists $\W^T$ such that $(W^T=(Y,Z^T),\W^T)$ belongs to ${\mathcal R}^{\alpha,\chi}([0,T) \times \R,\R^2)$ with $\alpha>1/3$
and $\chi < \alpha/2$, the supremum $\sup_{0 \leq T \leq T_{0}}
\kappa_{\alpha,\chi}((W_{t}^T,\W_{t}^T)_{0 \leq t <T})$ being finite.

A probability measure $\P$ on $\mathcal{C}([0,T_{0}],\R)$ (endowed with the canonical filtration 
$({\mathcal F}_{t})_{0 \leq t \leq T_{0}}$)  is said to solve the martingale problem related to $\L$ starting from $x$ if the canonical process $(X_t)_{0 \leq t \leq T_{0}}$ satisfies the following two conditions:
\vspace{5pt}

(1) $\P(X_0=x_{0})=1$,
\vspace{3pt}

(2) for any $T \in [0,T_{0}]$, 
$f \in {\mathcal C}([0,T] \times \R,\R)$
and $u^{T} \in {\mathcal C}^1(\R,\R)$
satisfying \eqref{eq:29:12:14:1} with respect to 
some 
$\vartheta \geq 1$, $\gamma >0$ and 
$\beta\in(2\chi,\alpha)$,
the process
$( u_{t}(X_{t})-\int_0^{t}f_r(X_r)\ud r)_{0\leq t\leq T}$
is a square integrable martingale under $\P$, where $u$ is the solution of $\mathcal{P}(Y,f,T)$ 
with $u_{T}=u^T$.
\vspace{5pt}

A similar definition holds by letting the canonical process start from $x_{0}$ at some time $t_{0} \neq 0$, in which case 
we say that the initial condition is $(t_{0},x_{0})$ and (1) is replaced by $\P(\forall s \in [0,t_{0}], \ X_{s}=x_{0})=1$. 
\end{definition}

Note that we require more in Definition \ref{def:mart:pb} than in Definition \ref{defsolution} as we let the terminal time 
$T$ vary within the interval $[0,T_{0}]$. In particular, in order to consider a solution to the martingale problem, it is not enough to assume that, at time $T_{0}$, $(W^{T_{0}},\W^{T_{0}})$ belongs to 
${\mathcal R}^{\alpha,\chi}([0,T_{0}) \times \R,\R^2)$. The rough path structure must exist at any $0 \leq T \leq T_{0}$, the regularity of the path $W^T$ and of its iterated integral ${\mathcal W}^T$ being uniformly controlled in 
$T \in [0,T_{0}]$. 

Our goal is then to prove existence and uniqueness of a solution:
\begin{theorem}
\label{thm:localmart}
In addition to the assumption of Definition \ref{def:mart:pb}, assume that, at any time $0 \leq T \leq T_{0}$, $(W^T,\W^T)$ is {geometric} (in the sense of Proposition
\ref{mildsolution:approx}), the paths $(Y^n)_{n \geq 1}$ used for defining 
the approximating paths
$(W^{n,T},\W^{n,T})_{n \geq 1}$ being the same
 for all the $T$'s
 and the supremum 
 $\sup_{0 \leq T \leq T_{0}} \sup_{n \geq 1}
\kappa_{\alpha,\chi}((W_{t}^{n,T},\W_{t}^{n,T})_{0 \leq t <T})$ being finite. Then, 
{
for an initial condition $(t_{0},x_{0}) \in [0,T_{0}] \times \R$, there exists a unique solution to the 
martingale problem (on $[0,T_{0}]$) with $(t_{0},x_{0})$ as initial condition. It is denoted by $\P_{t_{0},x_{0}}$. 
The mapping $[0,T_{0}] \times \R \ni (t,x) \mapsto \P_{t,x}(A)$ is measurable for any Borel subset $A$ of the canonical space ${\mathcal C}([0,T_{0}],\R)$. Moreover, it is strong Markov.}
\end{theorem}

\begin{remark}
The martingale problem is here set on the finite interval $[0,T_{0}]$. Obviously, existence and uniqueness extend to 
$[0,\infty)$. 
\end{remark}
%
%
%

{
The proof of Theorem
\ref{thm:localmart} is split into two
distinct parts: Existence of a solution is discussed in
Subsection  
\ref{subse:proof:mp:existence}
whereas uniqueness is investigated in 
Subsection 
\ref{subse:uniqueness:mart:prob}.
}

\subsection{
Solvability of the martingale problem}
\label{subse:proof:mp:existence}
We start with:
\begin{proposition}
 \label{thm:uniqueness}
Given $T_{0}>0$, 
assume that the assumption of Theorem \ref{thm:localmart} is in force. For an initial condition $(t_{0},x_{0}) \in [0,T_{0}] \times \R$, there exists a solution to the 
martingale problem (on $[0,T_{0}]$) with $(t_{0},x_{0})$ as initial condition. 
\end{proposition}

\begin{proof}[Proof of Proposition \ref{thm:uniqueness}]
\textit{First step.}
Without any loss of generality, we can assume that $t_{0}=0$. 
Considering a sequence of paths $(Y^n)_{n \geq 1}$ as in the statement of Proposition \ref{mildsolution:approx}, 
we can also assume that $Y^n$ has bounded derivatives on the whole space, 
see Lemma \ref{lem:regularisation} in the appendix.
We then notice that, for a given $x_{0} \in \R$, the
following
 SDE (set on some filtered probability space endowed with a Brownian motion 
$(B_{t})_{0 \leq t \leq T_{0}}$) admits a unique solution:
\begin{align}
	\ud X^n_t=\ud B_t+\partial_xY_t^n(X^n_t)\ud t, \quad t \in [0,T_{0}] \quad ; \quad X_{0} =x_{0}.
\end{align}

\textit{Second step.} 
Choosing $\beta \in (1/3,\alpha)$ with $\beta > 2 \chi$ and letting $u^T(x)=\exp( \vartheta  x )$ for a given $T \in [0,T_{0}]$, we denote by $(u^n_{t}(x))_{0 \leq t \leq T,x \in \R}$ the mild solution to \eqref{eq:mildpde} with $f=0$ and $Y$ replaced by $Y^n$. 
{Following the remark after Proposition \ref{mildsolution:approx},  $u^n$ is a classical solution of}
\begin{equation}
\label{eq:2:2:14:10}
\partial_{t} u^n_{t}(x) + \tfrac{1}{2} \partial_{xx}^2 u^n_{t}(x) + \partial_{x} Y^n_{t}(x) \partial_{x} u^n_{t}(x)=0,
\end{equation}
so that, by It\^o's formula, the process 
$(u^n_{t}(X^n_{t}))_{0 \leq t \leq T}$ is a true martingale (since we know, from Theorem \ref{mildsolution}, that
$u^n$ is at most of exponential growth). Then, \eqref{eq:31:10:1} yields
\begin{equation*}
\E \bigl[ \exp \bigl(  \vartheta X^n_{T} \bigr) \bigr] = \E \bigl[u^n_{T}\bigl(X^n_{T}\bigr)\bigr] = u_{0}(x_{0}) 
\leq  C  \exp( C \vert x_{0} \vert),
\end{equation*}
where $C=C(m,\alpha,\beta,\chi)$ as in Theorem \ref{mildsolution}. A crucial thing is that $m$ is uniformly bounded in 
$T \in [0,T_{0}]$ so that it can be assumed to be independent of $T$. 
Replacing $u^T(x)$ by $u^T(-x)$, we get the same result with $\vartheta$ replaced by $-\vartheta$ in the above inequality, so that 
\begin{equation*}
\E \bigl[ \exp \bigl( \vartheta \vert X^n_{T} \vert \bigr) \bigr] \leq  C \exp
\bigl( C
\vert x_{0} \vert \bigr).
\end{equation*}
Therefore, the exponential moments of $X^n_{T}$ are bounded, uniformly in $n \geq 1$. As $C$ is independent of $T \in [0,T_{0}]$, we deduce that the marginal exponential moments of $(X^n_{t})_{0 \leq t \leq T_{0}}$ are bounded, uniformly in $n \geq 1$. 
\vspace{5pt}

\textit{Third step.} Now we change the domain of definition and the terminal condition of the PDE.
We consider the PDE on $[0,t+h] \times \R$ with $u^{t+h}(x)=x$ as boundary condition, where $0 \leq t \leq t+h \leq T_{0}$. To simplify, we still denote by $(u^n_{s}(x))_{0 \leq s \leq t+h,x \in \R}$ the mild solution to \eqref{eq:mildpde} with $f=0$, 
$Y$ replaced by $Y^n$ and $u^n_{t+h}=u^{t+h}$ as terminal condition. By It\^o's formula, 
\begin{equation}
\label{eq:31:10:50}
\begin{split}
X^n_{t+h} - X^n_{t} &= u^n_{t+h}(X^n_{t+h}) - u^n_{t}(X^n_{t}) + u^n_{t}(X_{t}^n) - u^n_{t+h}(X^n_{t})
\\
&= \int_{t}^{t+h} \partial_{x} u_{s}^n(X^n_{s})\ud B_{s} + u^n_{t}(X_{t}^n) - u^n_{t+h}(X^n_{t}).
\end{split}
\end{equation}
Therefore, by \eqref{eq:31:10:1} and \eqref{eq:31:10:2}, we deduce that, for any $q \geq 2$, there exists a constant 
$C_{q}$, independent of $n$, such that
\begin{equation*}
\begin{split}
{\mathbb E} \bigl[ \vert X^n_{t+h} - X^n_{t} \vert^q \bigr]^{\frac{1}{q}} &\leq 
C_{q} \biggl\{ {\mathbb E} \biggl[ \biggl( \int_{t}^{t+h}
\vert \partial_{x} u_{s}^n(X^n_{s}) \vert^2 \ud s \biggr)^{\frac{q}2}\biggr]^{\frac{1}q} + 
{\mathbb E} \bigl[ \vert
u^n_{t}(X_{t}^n) - u^n_{t+h}(X^n_{t}) \vert^q \bigr]^{\frac{1}{q}} \biggr\}
\\
&\leq C_{q} \bigl\{ h^{\frac{1}{2} - \frac{1}{q}} \sup_{0 \leq s \leq T_{0}} 
{\mathbb E} \bigl[ \vert \partial_{x} u_{s}^n(X^n_{s})
 \vert^q \bigr]^{\frac1q} + {\mathbb E} \bigl[ \vert
u^n_{t}(X_{t}^n) - u^n_{t+h}(X^n_{t}) \vert^q \bigr]^{\frac{1}q} \bigr\}
 \\
 &\leq C_{q} \bigl\{ h^{\frac12-\frac1q} \sup_{0 \leq s \leq T_{0}} 
{\mathbb E} \bigl[ \exp (q \vert X^n_{s} \vert) \bigr]^{\frac1q} + h^{\frac{1+\beta}{2}}  \sup_{0 \leq s \leq T_{0}} 
{\mathbb E} \bigl[ \exp (q \vert X^n_{s} \vert)
 \vert \bigr]^{\frac1q} \bigr\}. 
 \end{split}
\end{equation*}
By the second step (uniform boundedness of the exponential moments) and by Kolmogorov's criterion, we deduce that 
the processes $(X^n_{t})_{0 \leq t \leq T_{0}}$ are tight.  
\vspace{5pt}

\textit{Fourth step.} It remains to prove that any weak limit $(X_{t})_{0 \leq t \leq T_{0}}$ is a solution to the martingale problem. {The basic argument is taken from 
\cite[Lemma 5.1]{eth:kur:86}. Anyhow, it requires a careful adaptation 
since the test functions $u$ in Definition \ref{def:mart:pb} may be of exponential 
growth (whereas test functions are assumed to be bounded in 
\cite[Lemma 5.1]{eth:kur:86}). 
We thus give the complete proof.}
%
For $T \in [0,T_{0}]$, we know from Proposition \ref{mildsolution:approx} that we can find a sequence $(u^n)_{n \geq 1}$ of classical solutions 
to the problems ${\mathcal P}(Y^n,f,T)$ such that the sequence $(u^n,\partial_{x} u^n)_{n \geq 1}$ converges
towards $(u,\partial_{x} u)$, uniformly on compact subsets of $[0,T] \times \R$. 
Applying It\^o's formula to each $(u^n_{t}(X^n_{t}))_{0 \leq t \leq T}$, $n \geq 1$, we deduce that 
\begin{equation*}
u^n_{t}(X^n_{t}) - u^n_{0}(X^n_{0}) - \int_{0}^t f_{s}(X^n_{s}) \ud s 
= \int_{0}^t \partial_{x} u^n_{s}(X^n_{s}) \ud B_{s}, \quad 0 \leq t \leq T. 
\end{equation*}
By \eqref{eq:31:10:1}, we know that the functions $(\partial_{x} u^n)_{n \geq 1}$ are at most of exponential growth, 
uniformly in $n \geq 1$. Moreover, we recall that the processes $((X^n_{t})_{0 \leq t \leq T})_{n \geq 1}$
have finite marginal exponential moments, uniformly in $n \geq 1$ as well. Therefore, the martingales 
$((u^n_{t}(X^n_{t}) - u^n_{0}(X^n_{0}) - \int_{0}^t f_{s}(X^n_{s}) \ud s)_{0 \leq t \leq T})_{n \geq 1}$ are bounded in $L^2$,
uniformly in $n \geq 1$. Letting $n$ tend to the infinity, this completes the proof. 
\end{proof}

\subsection{{Proof of Theorem \ref{thm:localmart}}}
\label{subse:uniqueness:mart:prob}
We now complete the proof of Theorem \ref{thm:localmart}. 
Existence has been already proved in Proposition \ref{thm:uniqueness}.
The point is thus to prove uniqueness and measurability of the
solution with respect to the initial point.  

We first establish uniqueness of the marginal laws. Assume indeed that $\P_{1}$ and $\P_{2}$ are two solutions 
of the martingale problem with the same initial condition $(t_{0},x_{0})$. Then, for any
$f \in {\mathcal C}([0,T] \times \R,\R)$ satisfying 
\eqref{eq:29:12:14:1}, 
it holds 
\begin{equation}
\label{eq:9:4:1}
{\E_{1} \int_{t_{0}}^{T_{0}} f_{s}(X_{s}) \ud s = \E_{2} \int_{t_{0}}^{T_{0}} f_{s}(X_{s})\ud s},
\end{equation}
where $\E_{1}$ and $\E_{2}$ denote the expectations under $\P_{1}$ and 
$\P_{2}$ ($(X_{t})_{0 \leq t \leq T_{0}}$ denotes the canonical process). Indeed, denoting by $u$
the solution of the PDE 
 ${\mathcal P}(Y,f,T_{0})$ with $0$ as terminal condition at time $T_{0}$, we know from the definition of the martingale problem that, both under $\P_{1}$ and $\P_{2}$, the process
$( u_{s}(X_{s}) - \int_{t_{0}}^{s} f_{r}(X_{r}) \ud r )_{t_{0} \leq s \leq T_{0}}$
is a martingale. Therefore, taking the expectation under $\E_{1}$ and $\E_{2}$ and noticing that 
$u_{T_{0}}(X_{T_{0}}) = 0$ almost surely under $\P_{1}$ and $\P_{2}$, we deduce that both sides in \eqref{eq:9:4:1} are equal to 
$-u_{t_{0}}(x_{0})$, which is enough to complete the proof of \eqref{eq:9:4:1} and thus to prove that the marginal laws of the canonical process are the same under $\P_{1}$ and $\P_{2}$.

Following Theorems 4.2 and 4.6 in \cite{eth:kur:86}, we deduce that the martingale problem has a unique solution
(note that the results in \cite{eth:kur:86} hold for time homogeneous martingale problems whereas the martingale problem we are here investigating is time inhomogeneous; adding an additional variable in the state space, the problem we are considering can be easily turned into a time-homogeneous one). 
%
Measurability and strong Markov property are proved as in \cite{eth:kur:86}. \qed

\section{Solving the PDE}
\label{sec:pde}
This section is devoted to the proof of Theorem \ref{mildsolution}. As the definition of a mild solution in Definition \ref{defsolution} consists in a convolution of a rough integral with the heat kernel, the first step is to investigate the smoothing effect of a Gaussian kernel onto a rough 
integral. Existence and uniqueness of a mild solution to \eqref{eq:mildpde} is then proved by means of a contraction argument. 

Parts of the results presented here are variations of the ones obtained in Sections 3.1 and 3.2 of Hairer \cite{hai:13} for solving the KPZ equation, but differ slightly in the very construction of a mild solution, see Remark \ref{rem:ipp}. 
{The reader may also have a look at Section 3 in Hairer \cite{hai:11} for a quite simpler framework.}
 
\subsection{Mild solutions as Picard's fixed points}
\label{subse:mild}
 In this subsection, we fix $\alpha,\beta, \chi, \vartheta, \lambda$ such that $1/3<\beta<\alpha\leq1$, $\chi<\beta/2$ and $\vartheta, \lambda\geq1$. 
 Given $Y\in\C([0,T)\times\R,\R)$ for some final time $T\leq1$, we assume that there exists $\W^T$ such that $(W_t^T=(Y_{t},Z_{t}^T),\W_{t}^T)_{0 \leq t \leq T}$
is in ${\mathcal R}^{\alpha,\chi}([0,T) \times \R,\R^2)$, $(Z_{t}^T)_{0 \leq t \leq T}$ being given by \eqref{eq:Z}.
 We will simply denote by $\kappa$ the semi norm $\kappa_{\alpha,\chi}((W_{t}^T,\W_{t}^T)_{t\in[0,T)})$ and we will omit the superscript $T$ in $Z^T$, $W^T$ and 
 $\W^T$.
We also recall the definition of $\nt$ for $v\in{\mathcal B}^{\beta,\vartheta}([0,T) \times \R,W)$:
\begin{equation*}
\begin{split}
&\Theta^{\vartheta,\lambda}_T(v) 
:=
\sup_{a\geq1 \atop t \in [0,T)} \Bigl[ \frac{
 \ldbrack v \rdbrack_{\beta/2,\beta}^{[t,T) \times [-a,a]}
+ \tfrac12 \ldbrack \partial_{W} v \rdbrack_{\beta/2,\beta}^{[t,T) \times [-a,a]}
+ 
 \lambda^{\frac{\beta-\alpha}8}
  ( a^{-\beta'} \wedge (T-t)^{\beta'/2})
 \| \cR^{v_{t}}\|_{2\beta'}^{[-a,a]}}
{E_{T}^{\vartheta,\lambda}(t,a)}  
 \Bigr],
\end{split}
\end{equation*}
with 
$E_{T}^{\vartheta,\lambda}(t,a)=\exp[\lambda(T-t)+\vartheta  a(1 + T-t)]$.
We start with the following technical lemma, which plays a crucial role in the proof of 
Theorem \ref{mildsolution}:
\begin{lemma}
\label{lem:28:10:1b}
 For any $\gamma_1\leq\gamma_2 \leq \beta/2$
 and $k\in\N^*$, there is a constant $C
 =C(\alpha,\beta,\gamma_1,\gamma_2,\chi,k)$ (independent of $\vartheta$ and $\lambda$) such that, for any $t,\tau\in[0,T)$, with $\tau\leq T-t$, and any $a \geq 1$, the following bounds hold for any $v\in{\mathcal B}^{\beta,\vartheta}([0,T) \times \R,W)$ and any $x \in [-a,a]$:
 \begin{equation*}
\int_\R \int_0^\tau 
\frac{|\partial_x^kp_{1}(y)|}{s^{1+\gamma_{1}}} \biggl|
\int_x^{x-\sqrt{s}y}v_{t+s}(z)\ud Y_{t+s}(z)\biggr|\ud s\ud y
\leq  \Psi
\lambda^{3\frac{\beta-\alpha}8}
\tau^{\gamma_2-\gamma_1}
a^{\gamma_2},
\end{equation*}
with $\Psi = 
Ce^{C T \vartheta^2}\kappa \nt E_{T}^{\vartheta,\lambda}(t,a)$. 
When $2 \gamma_{1} \leq \beta'$, we also have
\begin{equation*}
\int_\R 
\int_0^\tau \frac{|\partial_x^kp_{1}(y)|}{s^{1+2\gamma_{1}}} 
\biggl|  \int_x^{x-\sqrt{s}y}\bigl(v_{t+s}(z)-v_{t+s}(x)\bigr)\ud Y_{t+s}(z) \biggr| \ud s \ud y
\leq  
\Psi 
\lambda^{\frac{\beta-\alpha}8} \tau^{\beta'-2\gamma_1}
\Bigl(a^{\beta'} +  (T-t)^{-\frac{\beta'}{2}} \Bigr).
\end{equation*}
\end{lemma}

\begin{proof} In the whole proof, we just denote $\nt$ and $E_{T}^{\vartheta,\lambda}(t,a)$ by 
$\Theta$ and $E(t,a)$. We start with the proof of the first inequality. The point is to apply the second inequality in Lemma \ref{majint} with $y$ replaced by $x - \sqrt{s} y$ and thus 
$a$ replaced by $a+\vert y \vert$. We get 
\begin{equation*}
 \biggl|\int_x^{x-\sqrt{s}y}v_{t+s}(z)\ud Y_{t+s}(z)\biggr| 
\leq C \kappa \lambda^{\frac{\alpha-\beta}8}  \Theta 
E(t+s,a+\vert y \vert) 
\bigl[ s^{\frac{\alpha}{2}} \vert y \vert^{\alpha} \bigl( a+ \vert y \vert \bigr)^{\chi}+  
{\mathscr D}\bigl(t+s,a+\vert y \vert,\sqrt{s} y \bigr) \bigr],
\end{equation*}
where $C=C(\alpha,\beta)$. Noting that 
$E(t+s,a+\vert y \vert) \leq 
\exp[- ( \lambda + \vartheta (a + \vert y \vert)) s + \vartheta(1+T) \vert y \vert) ] E(t,a)$ and that ${\mathscr D}(t+s,a+\vert y\vert,\sqrt{s} y )
\leq C(1+ \vert y \vert^3) {\mathscr D}(t+s,a+\vert y\vert,\sqrt{s})$, we deduce that 
\begin{equation}
\label{eq:28:1}
\begin{split}
	&(a + \vert y \vert)^{-\gamma_2}\int_0^\tau s^{-1-\gamma_1}
	\biggl\vert 
	\int_{x'}^{x'-\sqrt{s}y}v_{t+s}(z)\ud Y_{t+s}(z) \biggr\vert \ud s \\
	&\hspace{15pt}\leq C\kappa
	 \lambda^{\frac{\alpha-\beta}8} 
	\Theta
	E(t,a)e^{\vartheta(1+T)|y|}\bigl(1+ \vert y \vert^3 \bigr)
   		\int_0^{\tau}\frac{e^{-(\lambda +\vartheta (a+\vert y \vert))s}}{s^{\gamma_1}(a+ \vert y\vert)^{\gamma_2}} {\mathscr D}'\bigl(t,s,a+\vert y \vert \bigr)\ud s,
\end{split}
\end{equation}
where 
\begin{equation}
\label{defg}
\begin{split}
{\mathscr D}'(t,s,\rho)
= s^{\frac{\alpha}{2}-1} \rho^\chi+s^{\alpha-1}\rho^{2\chi}+s^{\alpha+\frac{\beta}{2}-1} \rho^{2\chi+\frac{\beta}{2}}
+s^{\frac{\alpha}{2}+
{{\beta'}}-1} \rho^{\chi}\left(\rho^{{\beta'}}+(T-t-s)^{-\frac{{{\beta'}}}{2}}\right).
\end{split}
\end{equation}
We thus have to bound integrals of the form
	$\rho^{b'-\gamma_2}\int_0^{\tau}e^{-(\lambda +\vartheta \rho)s}s^{a'-\gamma_1-1}\ud s$
with $a'\geq\alpha/2$ ($\geq \gamma_{2}$), $0<b'\leq a'$ and $\rho \geq 1$. Bounding $s^{\gamma_2-\gamma_1}$ by $\tau^{\gamma_2-\gamma_1}$ and noticing that
\begin{equation}
\label{eq:1:2:14:1}
\begin{split}
	\frac{\rho^{b'-\gamma_2}}{(\lambda+\vartheta \rho)^{a'-\gamma_2}}&\leq 
	\frac{\rho^{b'-\gamma_2}}{(\lambda+\rho)^{a'-\gamma_2}}
	\\
	&\leq 
	\rho^{b'-a'} {\mathbf 1}_{\{\rho \geq \lambda\}}
	+ \lambda^{\gamma_{2}-a'} {\mathbf 1}_{\{1 \leq \rho < \lambda,b' < \gamma_{2}\}}
	+
	 \lambda^{b'-a'}{\mathbf 1}_{\{\rho < \lambda,b' \geq  \gamma_{2}\}}
	 \leq \lambda^{(b'\vee\gamma_2)-a'},
\end{split}
\end{equation}
we get the following upper bound for the integral (performing a change of variable to pass from the first to the second line and recalling that $\gamma_{2} \leq \beta/2$ to derive the last inequality):
\begin{equation}
\label{eq:28:2}
\begin{split}
	&\rho^{b'-\gamma_2}\int_0^{\tau}e^{-(\lambda +\vartheta \rho)s}s^{a'-\gamma_1-1}\ud s
	\leq \tau^{\gamma_{2}-\gamma_{1}}
	\rho^{b'-\gamma_2}\int_0^{\tau}e^{-(\lambda +\vartheta \rho)s}s^{a'-\gamma_2-1}\ud s
	\\
	&\hspace{50pt}\leq \frac{\tau^{\gamma_2-\gamma_1}\rho^{b'-\gamma_2}}{(\lambda+\vartheta \rho)^{a'-\gamma_2}}\int_0^{(\lambda +\vartheta \rho)\tau}e^{-s}s^{a'-\gamma_2-1}\ud s
	\leq\tau^{\gamma_2-\gamma_1}\lambda^{(b'\vee \frac{\beta}2)-a'}\Gamma(a'-\gamma_2).
\end{split}
\end{equation}
Because of the term in $(T-t-s)$ in the definition of ${\mathscr D}'$, we also have to control 
\begin{equation}
\label{eq:28:3:bobby}
\begin{split}
	&\rho^{\chi-\gamma_2}\int_0^{\tau}\frac{e^{-(\lambda +\vartheta \rho)s}}{s^{1-\frac{\alpha}{2}-\beta'+\gamma_1}(T-t-s)^{\frac{\beta'}{2}}} \ud s \leq \tau^{\gamma_2-\gamma_1}\rho^{\chi-\gamma_2}\int_0^{\tau}\frac{e^{-(\lambda +\vartheta \rho)s}}{s^{1-\frac{\alpha}{2}- \beta'+\gamma_2}(T-t-s)^{\frac{\beta'}{2}}} \ud s\\
	&\hspace{15pt} = \tau^{\gamma_2-\gamma_1}\frac{\rho^{\chi-\gamma_2}}{(\lambda+\vartheta \rho)^{\frac{\alpha}{2}-\gamma_2}}\frac{\tau^{\frac{\beta'}{2}}}{(T-t)^{\frac{\beta'}{2}}(\lambda +\vartheta \rho)^{\frac{\beta'}{2}}}\int_0^{1}\frac{(\tau(\lambda+\vartheta \rho))^{\frac{\alpha}{2}+\frac{\beta'}{2}-\gamma_2}e^{-\tau(\lambda +\vartheta \rho)s}}{s^{1-\frac{\alpha}{2}-
	\beta'+\gamma_2}[1-\tau s/(T-t)]^{\frac{\beta'}{2}}}\ud s.
\end{split}
\end{equation}
In order to bound the integral in the second line, we use the inequality
$x^{a'}e^{-xs} \leq (a')^{a'}e^{-a'}/s^{a'}$,
which holds for $s\in(0,1]$ and $a',x\geq0$. Using also the bounds 
$\tau\leq T-t$ and $\lambda +\vartheta \rho\geq1$ together with \eqref{eq:1:2:14:1}, we get (for a possibly new value of the constant $C$):
\begin{equation}
\label{eq:28:3}
\begin{split}
&\rho^{\chi-\gamma_2}\int_0^{\tau}\frac{e^{-(\lambda +\vartheta \rho)s}}{s^{1-\frac{\alpha}{2}-\beta'+\gamma_1}(T-t-s)^{\frac{\beta'}{2}}} \ud s
\\
&\hspace{15pt} \leq C \tau^{\gamma_2-\gamma_1}\lambda^{(\chi \vee \gamma_{2})-\frac{\alpha}{2}}\int_0^{1}\frac{\ud s}{s^{1-\frac{\beta'}{2}}(1-s)^{\frac{
\beta'}{2}}}\leq C\tau^{\gamma_2-\gamma_1}\lambda^{\frac{\beta-\alpha}{2}}.
\end{split}
\end{equation}
A careful inspection of \eqref{defg} shows that we can apply \eqref{eq:28:2} and \eqref{eq:28:3} 
with $a' \geq \alpha/2$ and $b'-a' \leq \chi - \alpha/2$ in order to bound \eqref{eq:28:1}
($a'$ is the part different from $-1$ in the exponent of $s$ and 
$b'$ is the exponent of $\rho$ in \eqref{defg}). We obtain
\begin{equation}
\label{eq:28:bob}
\begin{split}
	&(a+|y|)^{-\gamma_2}\int_0^\tau s^{-1-\gamma_1}
	\biggl\vert 
	\int_{x}^{x-\sqrt{s}y}v_{t+s}(z)\ud Y_{t+s}(z) \biggr\vert \ud s
	\\
	&\hspace{15pt}\leq C\kappa  \lambda^{\frac{\alpha-\beta}8} \Theta E(t,a)e^{\vartheta(1+T)|y|}\bigl( 1+ \vert y\vert^3 \bigr)
   		\tau^{\gamma_2-\gamma_1}\lambda^{\frac{\beta-\alpha}{2}}.
\end{split}
\end{equation}
As $a^{-\gamma_2}\leq (1+|y|)^{\gamma_2}(a+|y|)^{-\gamma_2}$, we get the first bound of the lemma by integrating \eqref{eq:28:bob} against $\left|\partial_x^kp_{1}(y)\right|$. 

We now turn to the proof of the second inequality in the statement. We make use of the first inequality in Lemma \ref{majint}. Replacing $v_{t+s}(z)$ by $v_{t+s}(z) - v_{t+s}(x)$
in \eqref{eq:28:1}, we get the same inequality but with a simpler form of 
${\mathcal D}'(t,s,a+\vert y \vert)$, namely the first term in the right-hand side in 
\eqref{defg} doesn't appear. This says that we can now apply \eqref{eq:28:2} with 
$a' \geq \alpha \wedge (\alpha/2+\beta') \geq \beta'$ and $b'-a' \leq \chi - \alpha/2$.
The value of $a'$ being larger than $\beta'$, this permits to apply \eqref{eq:1:2:14:1}
with $\gamma_{2}$ replaced by $\beta'$. 
Then, we can replace $\gamma_{1}$ and $\gamma_{2}$ by $2 \gamma_{1}$ and 
$\beta'$ in \eqref{eq:28:2} (with $\gamma_{1} \leq \beta'/2$). 
With the prescribed values of $a'$ and $b'$, the resulting bound in \eqref{eq:28:2}
is $C \tau^{\beta' - 2 \gamma_{1}} \lambda^{(b' \vee \beta') -a' }$. 
Following \eqref{eq:28:bob}, we see that the contribution of 
\eqref{eq:28:2} in the second inequality of the statement is $
 \lambda^{(\alpha-\beta)/8}
\Psi
\lambda^{(\beta-\alpha)/2} \tau^{\beta'-2 \gamma_{1}} a^{\beta'}
\leq \Psi
 \lambda^{(\beta-\alpha)/8}  \tau^{\beta'-2 \gamma_{1}} a^{\beta'}$, which fits the first part of the inequality. To recover the second part of the inequality, we must discuss the contribution of 
\eqref{eq:28:3:bobby}. 
Going back to
\eqref{defg}, we have to analyze (pay attention that, in comparison with 
\eqref{eq:28:3:bobby}, $\gamma_{2}$ is set to $0$):
\begin{equation}
\label{eq:28:4}
\begin{split}
&(T-t)^{\frac{\beta'}{2}}
\rho^\chi
\int_0^{\tau}\frac{e^{-(\lambda +\vartheta \rho)s}}{s^{1-\frac{\alpha}{2}-
\beta'+2\gamma_1} (T-t-s)^{\frac{\beta'}{2}}}\ud s
\\
&\hspace{5pt} \leq \tau^{\beta'-2\gamma_1}
\rho^{\chi}
\int_0^{\tau}\frac{e^{-(\lambda +\vartheta \rho)s}}{s^{1-\frac{\alpha}{2}} (1-
s/(T-t))^{\frac{\beta'}{2}}}\ud s
 \leq \tau^{\beta'-2\gamma_1}
\tau^{\alpha/2}
\rho^\chi \int_0^1\frac{e^{-\tau(\lambda +\vartheta \rho)s}}{s^{1-\frac{\alpha}{2}}(1-s)^{\frac{\beta'}{2}}}\ud s
\\
&\hspace{5pt} = \tau^{\beta'-2\gamma_1}\tau^{\frac{\alpha/2-\chi}{2}}\frac{\rho^\chi}{(\lambda +\vartheta \rho)^{\frac{\alpha/2+\chi}{2}}} \int_0^1\frac{
(\tau(\lambda +\vartheta \rho))^{\frac{\alpha/2+\chi}{2}} e^{-\tau(\lambda +\vartheta \rho)s}}{s^{1-\frac{\alpha}{2}}(1-s)^{\frac{\beta'}{2}}}\ud s
 \leq C\lambda^{\frac{\chi-\alpha/2}{2}}\tau^{\beta'-2\gamma_1},
\end{split}
\end{equation}
the first inequality being valid for $2\gamma_{1} \leq \beta'$ only and the
last inequality following from
\eqref{eq:1:2:14:1}. 
Noting that $\chi < \beta/2$, 
this gives the second part of the second inequality of the statement. 
\end{proof}

Here is now the key result to prove Theorem \ref{mildsolution}. 
   \begin{theorem}\label{appcontra}
Keep the notations and assumptions introduced at the beginning of Subsection 
\ref{subse:mild}. 
 For $(v,\partial_{W} v) \in {\mathcal B}^{\beta,\vartheta}([0,T) \times \R,W)$,
define the function $\M(v,\partial_{W}v) : [0,T) \times \R \rightarrow \R$ together with its $W$-derivative
by letting, for any $t\in[0,T)$ and $x\in\R$,
  \begin{equation*}
\begin{split}
 &\bigl[ \M(v, \partial_{W} v) \bigr]_{t}(x)= 
 \int_{t}^T\int_{\R}\partial^2_xp_{s-t}(x-y)\int_x^yv_s(z)\ud Y_s(z)\ud y\ud s.
\\
&\partial_{W} \bigl[ \M(v,\partial_{W} v) \bigr]_{t}(x) = \bigl( 0, v_{t}(x) \bigr) 
\quad 
(i.e. \ \partial_{Y} \M(v,\partial_{W} v)_{t}(x) = 0, \ 
\partial_{Z} \M(v,\partial_{W} v)_{t}(x) = v_{t}(x)).
\end{split}
\end{equation*}
(With an abuse of notation, we will just write $(\M v)_{t}(x)$ for $[\M(v,\partial_{W} v)]_{t}(x)$.) Then $\M$ defines a bounded operator from ${\mathcal B}^{\beta,\vartheta}([0,T) \times \R,W)$ into itself. 
Moreover, there exists a positive constant $C=C(\alpha,\beta,\chi)$ such that for every 
$v\in {\mathcal B}^{\beta,\vartheta}([0,T) \times \R,W)$,
$$\nt[\M v] \leq 
 \bigl( \tfrac12 + C \kappa \exp(C T \vartheta^2) \lambda^{-\epsilon} \bigr) \nt,
 \qquad \textrm{with} \ \epsilon :=  (\alpha-\beta)/8.
 $$
\end{theorem}
\begin{proof} 
As in the proof of Lemma \ref{lem:28:10:1b}, we just denote $\nt$ and $E_{T}^{\vartheta,\lambda}(t,a)$ by $\Theta$ and $E(t,a)$.
By an obvious change of variable, we get for any $a \geq 1$, $x \in [-a,a]$ and $t\in[0,T)$,
\begin{align}
\label{eq:representation:Mv}
	(\M v)_{t}(x)= 
\int_{\R}\partial^2_xp_{1}(y) \int_{0}^{T-t}s^{-1}\int_x^{x-\sqrt{s}y}v_{t+s}(z)\ud Y_{t+s}(z)\ud s\ud y.
\end{align}
Then the first inequality of Lemma \ref{lem:28:10:1b} with $\gamma_1=\gamma_2=0$, $\tau=T-t$ and $k=2$ leads to
\begin{align}
\label{eq:28:10:1}
\bigl( E(t,a) \bigr)^{-1} \left|(\M v)_{t}(x)\right|&\leq C\kappa e^{C T \vartheta^2}
 \lambda^{3\frac{\beta-\alpha}8}
\Theta,
\end{align}
where $C=C(\alpha,\beta,\chi)$.

We now study the time variations of $\M v$. For $0 \leq t\leq s \leq T$ and $x\in\R$, 
we deduce from the identity $\frac{1}{2}\partial_{x}^2p=\partial_tp$:
  \begin{equation*}
  \begin{split}
\bigl|(\M v)_{s}(x)- (\M v)_{t}(x)\bigr|&\leq \frac{1}{2}\left|\int_{s}^T\int_t^s\int_\R \partial_x^{4}p_{r-u}(x-y)\int_x^yv_r(z)\ud Y_r(z)\ud y\ud u\ud r\right|\\
&\hspace{15pt}+\left|\int_{t}^s\int_\R\partial^2_xp_{r-t}(x-y)\int_x^yv_r (z)\ud Y_r (z)\ud y\ud r \right|\\
&:=\frac{1}{2}\mathcal{T}_1+\mathcal{T}_2.
\end{split}
\end{equation*}
By the changes of variable
$(r,u) \mapsto (s+r-u,s-u)$ and 
then $y \mapsto x- \sqrt{r} s$, we get:
\begin{align*}
	\mathcal{T}_1&=\left|\int_\R\partial_x^{4}p_{1}(y)\int_{0}^{s-t}\int_u^{T-s+u}\frac{1}{r^2}\int_x^{x-\sqrt{r}y}v_{s+r-u}(z)\ud Y_{s+r-u}(z)\ud r \ud u\ud y\right|\\
	&\leq \int_\R\left|\partial_x^{4}p_{1}(y)\right|\int_{0}^{s-t}u^{\frac{\beta}2-1}\int_0^{T-t}\frac{1}{r^{1+\frac{\beta}2}}\biggl|\int_x^{x-\sqrt{r}y}v_{s+r-u}(z)\ud Y_{s+r-u}(z)\biggr|\ud r \ud u\ud y.
\end{align*}
Applying Lemma \ref{lem:28:10:1b} with $\tau=T-t$, $\gamma_1=\gamma_2=\beta/2$ and $k=4$, we obtain
\begin{align*}
a^{-\frac{\beta}2}\mathcal{T}_1&\leq Ce^{C T \vartheta^2}\kappa\Theta E(t,a)
 \lambda^{3\frac{\beta-\alpha}8} \int_{0}^{s-t}u^{\frac{\beta}2-1}\ud u
	\leq Ce^{C T \vartheta^2}\kappa\Theta E(t,a)
	 \lambda^{3\frac{\beta-\alpha}8}
	(s-t)^{\frac{\beta}2},
\end{align*}
where $C=C(\alpha,\beta,\chi)$. 
In order to handle ${\mathcal T}_{2}$, we can directly use Lemma \ref{lem:28:10:1b} with $\tau=s-t$, $\gamma_1=0$, $\gamma_{2}=\beta/2$ and $k=2$. 
We then obtain the same bound as for $\mathcal{T}_1$, so that 
  \begin{equation}\label{eq:28:10:12}
  \begin{split}
a^{-\frac{\beta}2} \bigl( E(t,a) \bigr)^{-1}
\bigl|(\M v)_{s}(x)- (\M v)_{t}(x)\bigr|
&\leq Ce^{C T \vartheta^2}\kappa\Theta 
\lambda^{3\frac{\beta-\alpha}{8}}
(s-t)^{\frac{\beta}2}.
\end{split}
\end{equation}

We now investigate the space variations. Fix $-a \leq x < x'\leq a$. If $|x'-x|^2\leq T-t$, 
the space increment between $x$ and $x'$ reads:
\begin{align}
\bigl|(\M v)_{t}(x')-(\M v)_{t}(x)\bigr|
&=\left|\int_{t}^T\int_{\R}\left(\partial^2_xp_{s-t}(x'-y)-\partial^2_xp_{s-t}(x-y)\right)\int_x^y
	v_s(z)\ud Y_s(z)\ud y\ud s\right| \nonumber
\\
&\leq {\mathcal I}_{1}^{x,x'}(x) + {\mathcal I}_{1}^{x,x'}(x') + {\mathcal I}_{2}^{x,x'}, \label{eq:3:2:14:1}
\end{align}
with (using the fact that the mapping $\R \ni z \mapsto \partial^2_{x} p_{s}(z)$ is centered)
\begin{equation*}
\begin{split}
&{\mathcal I}_{1}^{x,x'}(\xi) := \biggl\vert
\int_{0}^{|x'-x|^2}\int_{\R}\partial^2_xp_{s}(\xi-y)\int_{\xi}^y
	v_{t+s}(z)\ud Y_{t+s}(z)\ud y\ud s\biggr\vert,
\\
&{\mathcal I}_{2}^{x,x'} := 
\biggl|\int_{|x'-x|^2}^{T-t}\int_{\R}\int_x^{x'}\partial^3_xp_{s}(u-y)\int_x^y
	v_{t+s}(z)\ud Y_{t+s}(z)\ud u\ud y\ud s\biggr|.
\end{split}
\end{equation*}
By Lemma \ref{lem:28:10:1b} with $\tau=|x'-x|^2$, $\gamma_1=0$, 
$\gamma_2=\beta/2$ and $k=2$, 
we get 
\begin{align}
\label{eq:2:2:14:1}
a^{-\frac{\beta}2} \bigl( E(t,a) \bigr)^{-1}
\bigl(\mathcal{I}_1^{x,x'}(x)+\mathcal{I}_1^{x,x'}(x')\bigr)
\leq Ce^{C T \vartheta^2}\kappa\Theta 
\lambda^{3\frac{\beta-\alpha}{8}}
|x'-x|^{\beta}.
\end{align}
The term ${\mathcal I}_{2}^{x,x'}$ can be bounded in the following way:
\begin{equation}
\label{eq:3:2:14:2}
\begin{split}
	\mathcal{I}_2^{x,x'}&\leq\int_{\R}\left|\partial^3_xp_{1}(y)\right|\int_x^{x'}\int_{|x'-x|^2}^{T-t}s^{- \frac32}\left|\int_u^{u-\sqrt{s}y}
	v_{t+s}(z)\ud Y_{t+s}(z)\right|\ud s\ud u\ud y
	\\
	&\leq |x'-x|^{\beta-1}\int_{\R}\left|\partial^3_xp_{1}(y)\right|\int_x^{x'}\int_{|x'-x|^2}^{T-t}s^{-1-\frac{\beta}2}\left|\int_u^{u-\sqrt{s}y}
	v_{t+s}(z)\ud Y_{t+s}(z)\right|\ud s\ud u\ud y.
\end{split}
\end{equation}
Using now  Lemma \ref{lem:28:10:1b} with $\tau=T-t$, $\gamma_1=\gamma_2=\beta/2$ and $k=3$ we obtain:
\begin{align*}
a^{-\frac{\beta}2}\bigl( E(t,a) \bigr)^{-1}\mathcal{I}_2^{x,x'}\leq Ce^{C T \vartheta^2}\kappa\Theta \lambda^{3\frac{\beta-\alpha}{8}}|x'-x|^{\beta}.
\end{align*}
We end up with the following bound for the space increment:
\begin{align}\label{eq:28:10:13}
	a^{-\frac{\beta}2} \bigl(E(t,a) \bigr)^{-1}
	\bigl|(\M v)_{t}(x')-(\M v)_{t}(x)\bigr|\leq Ce^{C T \vartheta^2}\kappa\Theta
	\lambda^{3\frac{\beta-\alpha}{8}}|x'-x|^{\beta}.
\end{align}
Recall that 
\eqref{eq:28:10:13} holds true when $\vert x'-x\vert^2 \leq T-t$.  When $|x'-x|^2> T-t$, 
the argument is obvious as the space increment is smaller than ${\mathcal I}_{1}^{x,x'}(x) + {\mathcal I}_1^{x,x}(x')$, so that \eqref{eq:28:10:13} holds as well. 

We study the remainder term in a similar way. 
Recalling the definition \eqref{eq:31:10:7}, we then make use of the definition of $Z^T$, see \eqref{eq:Z}:
\begin{equation}
\label{eq:25:12:14:1}
\begin{split}
\cR^{(\M v)_{t}}(x,x')	&= (\M v)_{t}(x') - (\M v)_{t}(x) - v_{t}(x) \bigl( Z^T_{t}(x') - Z^T_{t}(x) \bigr) 
	\\
	&= \int_t^T\int_\R\left(\partial_x^2p_{s-t}(x'-y)-\partial_x^2p_{s-t}(x-y)\right)\int_x^y(v_s(z)-v_t(x))\ud Y_s(z)\ud y\ud s.
	\\
	&= \cR_{t}(x,x') + \cR^{\prime}_{t}(x,x'), 
\end{split}
\end{equation}
where
\begin{equation*}
\begin{split}
&\cR_{t}(x,x') := {\mathcal J}_{1}^{x,x'}(x') - {\mathcal J}_{1}^{x,x'}(x)
+ {\mathcal J}_{2}^{x,x'}, 
\quad \cR_{t}^\prime(x,x') := {\mathcal I}_{1}^{x,x',\prime}(x') - {\mathcal I}_{1}^{x,x',\prime}(x)
+ {\mathcal I}_{2}^{x,x',\prime}, 
\end{split}
\end{equation*}
with
\begin{equation*}
\begin{split}
&{\mathcal J}_{1}^{x,x'}(\xi)  := \int_{0}^{\vert x-x' \vert^2 \wedge (T-t)}
\int_{\R} \partial^2_{x} p_{s}(\xi - y) \int_{\xi}^y 
\bigl( v_{t+s}(\xi) - v_{t}(x) \bigr) \ud Y_{t+s}(z) \ud y \ud s,
\\
&{\mathcal J}_{2}^{x,x'} := \int_{\vert x-x' \vert^2 \wedge (T-t)}^{T-t}
\int_{\R} \int_{x}^{x'} \partial^3_{x} p_{s}(u-y) \int_{u}^y 
\bigl( v_{t+s}(u) - v_{t}(x) \bigr) \ud Y_{t+s}(z) 
\ud u \ud y \ud s,
\end{split}
\end{equation*}
and
\begin{equation*}
\begin{split}
&{\mathcal I}_{1}^{x,x',\prime}(\xi)  := \int_{0}^{\vert x-x' \vert^2 \wedge (T-t)}
\int_{\R} \partial^2_{x} p_{s}(\xi - y) \int_{\xi}^y 
\bigl( v_{t+s}(z) - v_{t+s}(\xi) \bigr) \ud Y_{t+s}(z) \ud y \ud s,
\\
&{\mathcal I}_{2}^{x,x',\prime} := \int_{\vert x-x' \vert^2 \wedge (T-t)}^{T-t}
\int_{\R} \int_{x}^{x'} \partial^3_{x} p_{s}(u-y) \int_{u}^y 
\bigl( v_{t+s}(z) - v_{t+s}(u) \bigr) \ud Y_{t+s}(z) 
\ud u \ud y \ud s.
\end{split}
\end{equation*}
%
%
We start with $\cR^{\prime}$. The strategy is similar to the one used to prove 
\eqref{eq:28:10:13} except that we now apply the second inequality
in  Lemma \ref{lem:28:10:1b} and not the first one. 
In order to handle ${\mathcal I}_{1}^{x,x',\prime}(\xi)$,
with $\xi=x$ or $x'$, we apply the second inequality in 
Lemma \ref{lem:28:10:1b} (with $k=2$, $\tau = \vert x-x' \vert^2 \wedge (T-t)$ and $\gamma_{1}=0$)
in the spirit of \eqref{eq:2:2:14:1}. 
Similarly, 
we can play the same game 
as in \eqref{eq:3:2:14:2}
to tackle ${\mathcal I}_{2}^{x,x',\prime}$, writing $s^{-3/2} = 
s^{-1-\beta'} s^{-1/2+\beta'} \leq \vert x' - x \vert^{2\beta'-1}
s^{-1-\beta'}$ and applying the second inequality in  
Lemma \ref{lem:28:10:1b} (with $k=3$, $\tau = T-t$ and $2 \gamma_{1} = \beta'$). We get
\begin{equation}
\label{eq:28:10:14}
\begin{split}
 \bigl((T-t)^{\frac{\beta'}{2}}\wedge a^{-\beta'}\bigr)
 \bigl( E(t,a ) \bigr)^{-1}\vert \cR_{t}^\prime(x,x') \vert
 &\leq Ce^{C T \vartheta^2}\kappa\Theta 
 \lambda^{\frac{\beta-\alpha}{8}}|x'-x|^{2\beta'}.
\end{split}
\end{equation}
It thus remains to discuss 
${\mathcal J}_{1}^{x,x'}$
and  
${\mathcal J}_{2}^{x,x'}$. We start with the following general bound that holds
true for any $\xi \in [x,x']$ and $s \in [0,T-t]$.   
Since $\beta \leq 2\beta' \leq 2\beta$, we indeed have 
\begin{equation*}
\vert v_{t+s}(\xi)-v_{t}(x) \vert \leq C 
\bigl(\| v_{t+s}  \|_{\infty}^{[-a,a]} + \| v_{t}  \|_{\infty}^{[-a,a]}\bigr)^{2-2\frac{\beta'}\beta}
\vert v_{t+s}(\xi) - v_{t}(x) \vert^{2\frac{\beta'}\beta-1},
\end{equation*}
so that (using the rate of growth of 
$ \ldbrack v \rdbrack_{\beta/2,\beta}^{[t,T) \times [-a,a]}$
in $a$)
\begin{equation}
\label{eq:Roland}
\vert v_{t+s}(\xi)-v_{t}(x) \vert \leq C a^{\beta' - \frac{\beta}2} E(t,a) \Theta
\bigl( \vert \xi -x \vert^{2 \beta' - \beta} + s^{\beta'-\frac{\beta}{2}} \bigr).
\end{equation}
We now handle ${\mathcal J}_{1}^{x,x'}$. 
Following \eqref{eq:2:2:14:1} (but noticing that the integrand is here constant in $z$), we deduce from \eqref{eq:Roland} with 
$s \leq \vert x'-x \vert^2$, 
\begin{equation*}
\begin{split}
\vert {\mathcal J}_{1}^{x,x'}(x)
\vert +
\vert {\mathcal J}_{1}^{x,x'}(x')
\vert  
&\leq 
C \kappa a^{\beta' - \frac{\beta}{2}+\chi} E(t,a) \Theta 
\vert x'-x \vert^{2\beta'-\beta}
\int_{0}^{\vert x'-x \vert^2 \wedge (T-t)}
s^{-1+\frac{\alpha}{2}} ds 
\\
&\leq 
C \kappa a^{\beta'} E(t,a) \Theta 
\vert x'-x \vert^{2\beta'- \beta}
\int_{0}^{\vert x'-x \vert^2}
s^{-1+\frac{\beta}{2}} ds 
\\
&\leq C \kappa a^{\beta'}
E(t,a) \Theta
\vert x'-x \vert^{2\beta'}.
\end{split}
\end{equation*}
Note that there is no decay
in $\lambda$ because 
$\vert v_{t+s}(x')-v_{t}(x) \vert$ is bounded by means of 
$E(t,a)$ and not of $E(t+s,a)$. 
Similarly, using \eqref{eq:Roland} with $\xi=u$ and $\vert u-x\vert \leq s^{1/2}$,
\begin{equation*}
\begin{split}
\vert {\mathcal J}_{2}^{x,x'} \vert &\leq 
C \kappa a^{\beta'} E(t,a) \Theta \vert x'-x \vert
\int_{\vert x-x' \vert^2 \wedge (T-t)}^{T-t} s^{-\frac{3}{2} + \beta' + \frac{\alpha-\beta}{2}}
ds
\\
&\leq C \kappa a^{\beta'}
E(t,a) \Theta \vert x'-x \vert
\int_{\vert x-x' \vert^2 \wedge (T-t)}^{T-t} s^{-\frac{3}{2} + \beta'}
ds \leq C \kappa a^{\beta'}
E(t,a) \Theta
\vert x'-x \vert^{2\beta'},
\end{split}
\end{equation*}
from which we deduce that 
\begin{equation*}
 \lambda^{\frac{\beta-\alpha}{8}}
a^{-\beta'}
 \bigl( E(t,a ) \bigr)^{-1}\vert \cR_{t}(x,x') \vert
 \leq Ce^{C T \vartheta^2}\kappa\Theta 
 \lambda^{\frac{\beta-\alpha}{8}}|x'-x|^{2\beta'}.
\end{equation*}
Together with \eqref{eq:28:10:14}, we get
\begin{equation}
\label{eq:27:12:14:1}
 \lambda^{\frac{\beta-\alpha}{8}}	
 \bigl((T-t)^{\frac{\beta'}{2}}\wedge a^{-\beta'}\bigr)
  \bigl( E(t,a ) \bigr)^{-1} \bigl\| 
    \cR^{(\M v)_{t}} \bigr\|_{2\beta'}^{[-a,a]}
    \leq Ce^{C T \vartheta^2}\kappa\Theta 
 \lambda^{\frac{\beta-\alpha}{8}}.
\end{equation}
	

%
%

Finally, as the $W$-derivative of $(\M v)_{t}$ is defined as $\partial_{W} (\M v)_{t}=(0,v_{t})$, we have  
\begin{equation}
\label{eq:28:10:15}
\frac{1}{2} \bigl( E(t,a) \bigr)^{-1} \|\partial_{W} (\M v)_{t}\|_{\beta/2,\beta}^{[t,T) \times [-a,a]} 
\leq \frac12 \Theta.
\end{equation}

From \eqref{eq:28:10:1}, \eqref{eq:28:10:12}, \eqref{eq:28:10:13}, 
\eqref{eq:27:12:14:1}
 and \eqref{eq:28:10:15}, this completes the proof. 
\end{proof}

\subsection{Proof of Theorem \ref{mildsolution}}

\textit{First step.}
As in the previous subsection, we omit the superscript $T$ in $Z^T$, $W^T$ and $\W^T$.
We also notice that Theorem \ref{appcontra} remains true when $T \leq T_{0}$, for some $T_{0} \geq 1$, provided that the constant $C$ in the statement is allowed to depend upon $T_{0}$. 

Now, for $f$ and $u^T$ as in \eqref{eq:29:12:14:1},
%
%
we let for $(t,x) \in [0,T) \times \R$:
\begin{equation}
\label{eq:25:12:14:5}
\phi_{t}(x) :=
P_{T-t} u^T(x) - \int_t^T P_{s-t}f_s(x)\ud s, \quad 
\psi_{t}(x) = \partial_{x} \phi_{t}(x), \quad (t,x) \in [0,T] \times \R.
\end{equation}
By standard regularization properties of the heat kernel, 
$\psi$ is $(\beta/2,\beta)$-H\"older continuous on any  
$[0,T] \times [-a,a]$, $a \geq 1$, the H\"older norm being less than $C \exp(\vartheta a)$. 
Moreover, 
\begin{equation}
\label{eq:29:12:14:10}
\sup_{0 \leq t <T} \sup_{a\geq1} \left\{ (T-t)^{\beta' - \frac{\beta}2} e^{-\vartheta a}  
\bigl\| \psi_{t} \bigr\|_{2 \beta'}^{[-a,a]} \right\} < \infty, 
\end{equation}
For 
$v \in {\mathcal B}^{\beta,\vartheta}([0,T) \times \R,W)$,
we then let
\begin{equation}
  \label{eq:29:10:20}
  \begin{split}
\bigl(  \Mc v \bigr)_{t}(x)&:= \psi_{t}(x)
+\bigl( \M v\bigr)_{t}(x).
  \end{split}
  \end{equation} 
The point is to check that $\Mc v$ can be lifted up into an element of 
${\mathcal B}^{\beta,\vartheta}([0,T) \times \R,W)$.
By Theorem \ref{appcontra}, the last part of the right-hand side is in ${\mathcal B}^{\beta,\vartheta}([0,T) \times \R,W)$. Its derivative with respect to $W$ is $\partial_{W}[\M v]$, as defined in the statement of Theorem 
 \ref{appcontra}. 
By \eqref{eq:29:12:14:10}, 
 for any $t \in [0,T)$,
$\psi_{t}$ is $2 \beta'$-H\"older continuous (in $x$) 
and belongs to ${\mathcal B}^{\beta,\vartheta}([0,T) \times \R,W)$ with a zero derivative with respect to $W$. Moreover, from \eqref{eq:29:12:14:10},
$\Mc v \in {\mathcal B}^{\beta,\vartheta}([0,T) \times \R,W)$, with 
$[\partial_{W} (\Mc v)]_{t}(x) = [\partial_{W} (\M v)]_{t}(x)=(0,v_{t}(x))$ for $t \in [0,T)$.

\textit{Second step.}  
We construct a solution 
to \eqref{eq:mildpde} by a contraction argument when $T \leq 1$ 
(the same argument applies when $T \geq 1$). We choose $\lambda$ large enough such that 
$C \kappa \exp(C T \vartheta^2) \lambda^{-\epsilon} \leq 1/4$ (with the same $C$ as in Theorem \ref{appcontra})
 and we note that $({\mathcal B}^{\beta,\vartheta}([0,T) \times \R,W),\Theta_T^{\vartheta,\lambda})$ is a Banach space. Since
 $\Mc u-\Mc v=\M(u-v)$ 
 for any 
  $u,v\in {\mathcal B}^{\beta,\vartheta}([0,T) \times \R,W)$
  (the equality holding true for the lifted versions), we deduce from 
  Theorem \ref{appcontra} and Picard fixed point Theorem that  $\Mc$ admits a unique fixed point $\bar{v}$ in 
  ${\mathcal B}^{\beta,\vartheta}([0,T) \times \R,W)$. Letting
  \begin{equation}
  \label{eq:30:10:1}
  \begin{split}
  \bar{u}_{t}(x) &=
\phi_t(x)
+\int_t^T\int_{\R} \partial_{x} p_{s-t}(x-y) \int_{x}^y \bar{v}_s (z)\ud Y_s(z) \ud y \ud s,
\end{split}
  \end{equation}
with $\phi$ as in \eqref{eq:25:12:14:5},  we obtain
$\partial_{x} \bar{u} = \bar{v}$ so that 
$\bar{u}$ is
 a mild solution, as defined in \eqref{eq:mildpde}. It must be unique as the $x$-derivative of 
  any other mild solution (when lifted up) is a fixed point of $\Mc$.  Differentiation under the integral symbol 
  in 
    \eqref{eq:30:10:1}
 and in the mild formulation 
  \eqref{eq:mildpde}
  can be justified by 
  Lebesgue's Theorem, using 
 bounds in the spirit of  
 Lemma \ref{lem:28:10:1b}. 
\vspace{5pt}
  
  \textit{Third step.}
We finally prove \eqref{eq:31:10:1} \and \eqref{eq:31:10:2}. We 
first estimate $\bar{v}$. With our choice of $\lambda$ and by Theorem \ref{appcontra}, we have
$\Theta^{\vartheta,\lambda}_{T}(\bar v) \leq \Theta^{\vartheta,\lambda}_{T}(\Mc 0) + 
(3/4) \Theta^{\vartheta,\lambda}_{T}(\bar v)$,
where $0$ is the null function,
so that 
\begin{equation}
\label{eq:3:11:1}
\Theta^{\vartheta,\lambda}_{T}(\bar v) \leq 4\Theta^{\vartheta,\lambda}_{T}(\Mc 0).
\end{equation} 
As $\Mc0 = \psi \in {\mathcal B}^{\beta,\vartheta}([0,T) \times \R,W)$,
the right-hand side is bounded by some
 $C$ (which would depend on $T_{0}$ if $T$ was less than 
$T_{0}$ for some $T_{0} \geq 1$). Since $\partial_{x} \bar u = \bar v$, this gives the exponential bound for $\bar v$
and for the $(\beta/2,\beta)$-H\"older constant of $\bar v$ in time and space.  

In order to get the same estimate for $\bar u$, we go back to 
  \eqref{eq:30:10:1}.
The function $\phi$ can be estimated by standard properties of the heat kernel: 
it is at most of exponential growth
and it is locally $(1+\beta)/2$-H\"older continuous in time, the
H\"older constant growing at most exponentially fast in the space variable. The second term can be handled by repeating the analysis of $\M v$ in the proof of Theorem \ref{appcontra}: Following \eqref{eq:28:10:1} and \eqref{eq:28:10:12}, it is at most of exponential growth and it is locally $(1+\beta)/2$-H\"older continuous in time, the H\"older constant growing at most exponentially fast in the space variable (in comparison with \eqref{eq:28:10:12}, the additional $1/2$ comes from the fact there is one derivative less in the heat kernel).

\subsection{Proof of Proposition \ref{mildsolution:approx}}
As above, we omit the superscript $T$ in $Z^{n,T}$, $W^{n,T}$ and $\W^{n,T}$. 
Stability of solutions under mollification of the input follows from a classical compactness argument. 
Given a sequence $(W^n,\W^n)_{n \geq 1}$ as in the statement, we can solve \eqref{eq:mildpde} for any $n \geq 1$: The solution 
is denoted by $u^n$ and its gradient by $v^n :=\partial_{x} u^n$. By (2) in Proposition \ref{mildsolution:approx}
and by \eqref{eq:3:11:1}, it is well-checked that 
\begin{equation}
\label{eq:31:10:10}
\sup_{n \geq 1} \Theta_{T}^{\vartheta,\lambda}(v ^n ) < \infty,
\end{equation}
where $[\partial_{W^n}( v^n)]_{t}=(0,v^n_{t})$. 
As a consequence,  
the sequence $(v^n)_{n \geq 1}$ is uniformly continuous on compact subsets of $[0,T] \times \R$. In the same way, 
the sequence $(u^n)_{n \geq 1}$ is also uniformly continuous on compact subsets. Moreover, $u^n$ and $v^n$ are at most of exponential growth (in $x$), uniformly in $n \geq 1$. By Arzel\`a-Ascoli Theorem, we can extract subsequences (still indexed by $n$) that converge uniformly on compact subsets
of $[0,T] \times \R$. Limits of $(u^n)_{n \geq 1}$
and $(v^n)_{n \geq 1}$
are respectively denoted by $\hat{u}$ and $\hat{v}$. 
In order to complete the proof, we must prove that $(\hat{u},\hat{v})$ is a mild solution of \eqref{eq:mildpde}. 

Writing \eqref{eq:31:10:7} for each of the $(v^n)_{n \geq 1}$, exploiting 
\eqref{eq:31:10:10} to
control the 
remainders 
$(\cR^{v^n_{t}})_{n \geq 1}$
uniformly in $n \geq 1$ and then letting $n$ tend to $\infty$, 
we deduce that the pair 
$(\hat{v},(0,\hat{v}))$ belongs to ${\mathcal B}^{\beta,\vartheta}([0,T),\R)$, the remainder at any time $t \in [0,T)$ 
being denoted by 
$\hat{\cR}^{t}$. 
By \eqref{eq:31:10:7}, 
$\lim_{n}\| \hat{\cR}^{t}-\cR^{v^n_{t}} \|_{\infty}^{[-a,a]}=0$ 
for any $a \geq 1$. By 
\eqref{eq:31:10:10}, it holds as well in $\beta''$-H\"older norm, 
for any $\beta'' \in (1/3,\beta')$, that is
$\lim_{n} \| \hat{\cR}^{t}
- \cR^{v^n_{t}}  \|_{2\beta''}^{[-a,a]}=0$. 

Replacing $\beta'$ by $\beta''$ in \eqref{eq:thm:1}, this suffices to pass to the limit in the rough integrals appearing in the mild formulation \eqref{eq:mildpde} of the PDE satisfied by each of the 
$(v^n)_{n \geq 1}$'s.
To pass to the limit in the whole formulation, we can invoke
Lebesgue's Theorem, using 
 bounds in the spirit of  
 Lemma \ref{lem:28:10:1b}. Thus the pair $(\hat{v},(0,\hat{v}))$ satisfies 
$\hat{v} = \Mc \hat{v}$ in ${\mathcal B}^{\beta,\vartheta}([0,T) \times \R,W)$, which is enough to conclude by uniqueness of the solution.

\section{Stochastic Calculus for the Solution}
\label{sec:young:sto}
In Theorem \ref{thm:localmart}, we proved existence and uniqueness of a solution to the martingale problem associated with 
\eqref{eq:18:1:1}, but we said nothing about the dynamics of the solution. In this section, we answer to this question and give a sense to the formulation \eqref{eq:our:stoc:calculus}.

\subsection{Recovering the Brownian part} Equation \eqref{eq:our:stoc:calculus} suggests that the dynamics of 
the solution to \eqref{eq:18:1:1} indeed involves some Brownian part. The point we discuss here is thus twofold: (i) We recover in a quite \textit{canonical} way the Brownian part in the dynamics of the solution; (ii) we discuss the structure of the remainder. 

\begin{theorem}
\label{thm:expansion}
Under the assumption of Theorem \ref{thm:localmart}, for any given initial condition $x_{0}$, we can find a probability measure (still denoted by ${\mathbb P}$) on the {\rm enlarged} canonical space ${\mathcal C}([0,T_{0}],\R^2)$ (endowed with the canonical filtration 
$({\mathcal F}_{t})_{0 \leq t \leq T_{0}}$) such that, under ${\mathbb P}$, the canonical process, denoted by 
$(X_{t},B_{t})_{0 \leq t \leq T_{0}}$,
satisfies the followings:

$(i)$ The law of $(X_{t})_{0 \leq t \leq T_{0}}$ under ${\mathbb P}$ is a solution to the martingale problem with 
$x_{0}$ as initial condition at time $0$ and the law of $(B_{t})_{0 \leq t \leq T_{0}}$ under $\P$ 
is a Brownian motion.

$(ii)$ For any $q \geq 1$ and any $\beta<\alpha$, there is a constant $C = C(\alpha,\beta,\chi,\kappa_{\alpha,\chi}(W,\W),q,T_{0})$ such that, for  any $0 \leq t \leq t +h \leq T_{0}$,
\begin{equation}
\label{eq:1:10:1}
{\mathbb E} \bigl[ \bigl\vert X_{t+h} - X_{t} - (B_{t+h} - B_{t}) \bigr\vert^q  \bigr]^{\frac1q} \leq C h^{(1+\beta)/2}.
\end{equation}

$(iii)$ For any $0 \leq t \leq t+ h \leq T_{0}$,
\begin{equation}
\label{eq:1:10:2}
{\mathbb E} \bigl[ X_{t+h} - X_{t} \vert {\mathcal F}_{t} \bigr] =  {\mathfrak b}(t,X_{t},h) : = u_{t}^{t+h}(X_{t}) - X_{t}, 
\end{equation}
where the mapping $u^{t+h} : [0,t+h] \times \R \ni (s,x) \mapsto u^{t+h}(s,x)$ is the mild solution of
${\mathcal P}(Y,0,t+h)$ with $u_{t+h}^{t+h}(x)=x$ as terminal condition. 
\end{theorem}

\begin{proof} The point is to come back to the proof of the solvability of the martingale problem in 
Subsection \ref{subse:proof:mp:existence}. For free and with the same notations, we have the tightness of the family $(X^n_{t},B_{t})_{0 \leq t \leq T_{0}}$, which is sufficient to extract a converging subsequence. The (weak) limit is the pair $(X_{t},B_{t})_{0 \leq t \leq T_{0}}$ in $(i)$. (Pay attention that  we do not claim that the `$B$' at the limit is the same as the `$B$' in the regularized problems but, for convenience, we use the same letter.) We then repeat the proof of \eqref{eq:31:10:50} which writes:
\begin{equation*}
\begin{split}
X^n_{t+h} - X^n_{t} &= \int_{t}^{t+h} \partial_{x} u_{s}^n(X^n_{s})\ud B_{s} + u^n_{t}(X_{t}^n) - u^n_{t+h}(X^n_{t})
\\
&= B_{t+h} - B_{t} + \int_{t}^{t+h} \bigl[\partial_{x} u_{s}^n(X^n_{s})-1\bigr]\ud B_{s}
+ \bigl[ u^n_{t}(X_{t}^n) - u^n_{t+h}(X^n_{t}) \bigr].
\end{split}
\end{equation*}
Repeating the analysis of the the third step in Subsection \ref{subse:proof:mp:existence}, we know that the third term in 
the right hand side satisfies the bound \eqref{eq:1:10:1}. The point is thus to prove that the second term also satisfies this bound. Recalling that $u^n_{t+h}(x)=x$, we notice that $\partial_{x} u_{s}^n(X^n_{s})-1=\partial_{x} u_{s}^n(X^n_{s}) - 
\partial_{x} u_{t+h}^n (X^n_{s})$. The bound then follows from the fact that $\partial_{x} u^n$ is locally $\beta/2$-H\"older continuous in time, the H\"older constant being at most of exponential growth, as ensured by Theorem \ref{mildsolution}. Letting 
$n$ tend to $\infty$, this completes the proof of $(ii)$. 

The last assertion $(iii)$ is easily checked for with $X$ replaced by $X^n$ and $u^{t+h}$ replaced by $u^n$ (and for sure with 
${\mathcal F}_{t}$ replaced by the $\sigma$-field generated by $(X^n_{s},B_{s})_{0 \leq s \leq t}$). It is quite standard to pass to the limit in $n$.  
\end{proof}

\subsection{Expansion of the drift}
The next proposition gives a more explicit insight into the shape of the function ${\mathfrak b}$ in \eqref{eq:1:10:2}:

\begin{proposition}
\label{prop:drift}
Given $T_{0}>0$, there exist a constant $C$ and an exponent $\varepsilon >0$ such that 
\begin{equation*}
\begin{split}
{\mathfrak b}(t,x,h) &= b(t,x,h) +  O\bigl( h^{1+\varepsilon} \exp(2 \vert x \vert) \bigr),
\\
b(t,x,h) &= \int_{t}^{t+h} \int_{\R} \partial_{x} p_{s-t}(x-y) \bigl( Y_{s}(y) - Y_{s}(x) \bigr)  \ud y  \ud s
\\
&\hspace{15pt}+    
\int_{t}^{t+h}  \int_{\R} \partial_{x} p_{s-t}(x-y) \int_{x}^y  Z_{s}^{t+h}(z)  \ud Y_{s}(z) \ud y  \ud s,
\end{split}
\end{equation*}
$O(\cdot)$ standing for the Landau notation (the underlying constant in the Landau notation being uniform in 
$0 \leq t \leq t+h \leq T_{0}$). 
\end{proposition} 

\begin{remark}
\label{rem:drift}
The first term in the definition of $b(t,x,h)$ reads as a mollification (in $x$) of the gradient (in $x$) of $(Y_{t}(x))_{t \leq s \leq t+h, x \in \R}$ by means of the transition density of $(B_{t})_{t \geq 0}$ (which is the martingale process driving $X$). It is (locally in $x$) of order 
$h^{1/2+\alpha/2}$. The second term reads as a correction in the mollification of $(Y_{s}(x))_{t \leq s \leq t+h,x \in \R}$. It keeps track of the rough path structure of $(Y_{s}(x))_{t \leq s \leq t+h, x \in \R}$. The proof right below shows that 
it is of order $h^{1/2+\alpha}$,  thus proving that it can be `hidden' in the remainder $O(h^{1+\epsilon})$ 
when $\alpha > 1/2$. This requirement $\alpha >1/2$ fits the standard threshold in rough paths above which Young's theory applies. 
\end{remark}

\begin{proof}
From \eqref{eq:mildpde}, we know that $u_{t}^{t+h}(x)$ expands as
\begin{equation*}
u_{t}^{t+h}(x) = x +  \int_{t}^{t+h} \int_{\R} \partial_{x} p_{s-t}(x-y) \int_{x}^y \bar v_{s}^{t+h}(z) \ud Y_{s}(z) 
\ud y  \ud s,
\end{equation*}
where $\bar v_{s}^{t+h}(y) = \partial_{x} u_{s}^{t+h}(y)$. 
Here, the function $\phi$ in 
\eqref{eq:mildpde} is equal to $\phi_{t}(x)=x$ for any $t \in [0,t+h]$
and $x \in \R$, and 
thus $\partial_{x} \phi \equiv 1$. By Theorem \ref{appcontra},
$\bar v^{t+h} \in {\mathcal B}^{\beta,\vartheta}([0,t+h) \times \R,W^{t+h})$
and solves the equation
$\bar v = 1+ \M \bar v$. In particular, 
$\partial_{Y} \bar v_{t}(x) = 0$ and 
$\partial_{Z^{t+h}} \bar v_{t}(x) = \bar v_{t}(x)$. 
Therefore, we can write 
\begin{equation*}
\bar v_{s}^{t+h}(z) =
\bar v_{s}^{t+h}(x) + \bar v_{s}^{t+h}(x)
\bigl( Z_{s}^{t+h}(z) - Z_{s}^{t+h}(x) 
\bigr) + \cR^{\bar v_{s}}(x,z) ,
\end{equation*}
which we can plug into the expression for $u_{t}^{t+h}(x)$ by means of Theorem \ref{defintrp}:
\begin{equation}
\label{eq:4:1:14:1}
\begin{split}
u_{t}^{t+h}(x) - x &=    
\int_{t}^{t+h} \bar v_{s}^{t+h}(x) \int_{\R} \partial_{x} p_{s-t}(x-y) 
\bigl( Y_{s}(y) - Y_s(x) \bigr) \ud y  \ud s
\\
&\hspace{5pt} +   
\int_{t}^{t+h}  \bar v_{s}^{t+h}(x) \int_{\R} \partial_{x} p_{s-t}(x-y) \int_{x}^y  \bigl( Z_{s}^{t+h}(z) - Z_{s}^{t+h}(x) 
\bigr) \ud Y_{s}(z) \ud y  \ud s
\\
&\hspace{5pt} +  \int_{t}^{t+h} \int_{\R} \partial_{x} p_{s-t}(x-y) 
{\mathscr U}_{s}^{t+h}(x,y) \ud y \ud s,
\end{split}
\end{equation}
where ${\mathscr U}_{s}^{t+h}(x,y)$ {is} a remainder term that derives from the approximation of the rough integral
of $\bar v_{s}^{t+h}$ with respect to $Y_{s}$. By Theorem 
 \ref{defintrp}, there exist a constant $C$ and an exponent $\varepsilon >0$ such that 
 \begin{equation}
 \label{eq:4:1:14:2}
 \begin{split}
&\biggl\vert
 \int_{t}^{t+h} \int_{\R} \partial_{x} p_{s-t}(x-y) 
{\mathscr U}_{s}^{t+h}(x,y) \ud y \ud s \biggr\vert 
\\
&\hspace{5pt}
\leq C \exp( 2\vert x \vert) \int_{t}^{t+h} (s-t)^{-\frac12} \int_{\R}  p_{s-t}(y)
\exp ( \vert y \vert) \vert y \vert^{1+\varepsilon} \ud y \ud s 
\leq C  \exp ( 2\vert x \vert) h^{1+\varepsilon}.  
\end{split}
\end{equation} 
Above, the exponential factor permits to handle the polynomial growth of 
{${\boldsymbol W}^{t+h}$, with
$W^{t+h}=(Y,Z^{t+h})$,} and the exponential growth 
of $\bar v^{t+h}$ (see the definition of $\Theta_{T}^{\vartheta,\lambda}(v)$ in the statement of Theorem
\ref{appcontra}), the exponent in the exponential factor being arbitrarily chosen as $1$ (which leaves `some space' to handle additional polynomial growth and which 
 is possible since the terminal condition 
$u_{t+h}^{t+h}$ is
of polynomial growth). 

We now investigate the second term in the right hand side of \eqref{eq:4:1:14:1}. 
We recall that, by assumption, there exists a constant $C$, independent of $h$, such that  
\begin{equation}
\label{eq:4:1:14:3}
\biggl\vert \int_{x}^y \bigl( Z_{s}^{t+h}(z) - Z_{s}^{t+h}(x) \bigr) \ud Y_{s}(z)
\biggr\vert
\leq \bigl\vert \W^{t+h}_{s}(x,y) \bigr\vert
 \leq C ( 1+ \vert x \vert \vee \vert y \vert)^{2 \chi} \vert x-y \vert^{2\alpha}.
\end{equation}
We also recall from Theorem \ref{mildsolution} that $\bar v$ is $(\alpha- \epsilon)/2$-H\"older continuous
in time, locally in space (the rate of growth of the H\"older constant being at most exponential
and
Theorem
\ref{appcontra} allowing to choose $1$ as exponent in the exponential), so that 
$\vert \bar v^{t+h}_{s}(y) - 1 \vert \leq C h^{(\alpha-\epsilon)/2} \exp(\vert y \vert)$, 
for $s\in [t,t+h]$ and
for a possibly new value of the constant $C$. 
Therefore, 
\begin{equation*}
\begin{split}
&\int_{t}^{t+h} \bar v_{s}^{t+h}(x) \int_{\R} \partial_{x} p_{s-t}(x-y) \int_{x}^y  \bigl( Z_{s}^{t+h}(z) - Z_{s}^{t+h}(x) 
\bigr) \ud Y_{s}(z) \ud y  \ud s
\\
&= \int_{t}^{t+h} \int_{\R} \partial_{x} p_{s-t}(x-y) \int_{x}^y  \bigl( Z_{s}^{t+h}(z) - Z_{s}^{t+h}(x) 
\bigr) \ud Y_{s}(z) \ud y  \ud s
\\
& \hspace{15pt} + \int_{t}^{t+h} \bigl(  \bar v_{s}^{t+h}(x) - 1 \bigr) \int_{\R} \partial_{x} p_{s-t}(x-y) \int_{x}^y  \bigl( Z_{s}^{t+h}(z) - Z_{s}^{t+h}(x) 
\bigr) \ud Y_{s}(z) \ud y  \ud s,
\end{split}
\end{equation*}
the last term being less than 
\begin{equation}
\label{eq:4:1:14:5}
C \exp(2\vert x \vert) h^{(\alpha-\epsilon)/2} \int_{t}^{t+h} (s-t)^{-1/2+\alpha} \ud r
\leq C \exp(2\vert x \vert) h^{1/2+3 \alpha/2 - \epsilon} \leq C \exp(2\vert x \vert) h^{1+\epsilon}, 
\end{equation}
the last inequality holding true since $\alpha$ is strictly larger than $1/3$ and $\epsilon$ can be chosen arbitrarily small.  
Therefore, from \eqref{eq:4:1:14:1}, \eqref{eq:4:1:14:2} and \eqref{eq:4:1:14:3}, we deduce that 
\begin{equation}
\label{eq:4:1:14:6}
\begin{split}
&u_{t}^{t+h}(x) - x =    
\int_{t}^{t+h} \bar v_{s}^{t+h}(x) \int_{\R} \partial_{x} p_{s-t}(x-y) \bigl( Y_{s}(y) - Y_s(x) \bigr) \ud y  \ud s
\\
&\hspace{5pt} +  
\int_{t}^{t+h}  \int_{\R} \partial_{x} p_{s-t}(x-y) \int_{x}^y  \bigl( Z_{s}^{t+h}(z) - Z_{s}^{t+h}(x) 
\bigr) \ud Y_{s}(z) \ud y  \ud s + O \bigl( \exp(2 \vert x \vert) h^{1+\epsilon} \bigr). 
\end{split}
\end{equation}
Using \eqref{eq:4:1:14:3} once more and following the proof of \eqref{eq:4:1:14:5}, we also have
\begin{equation*}
u_{t}^{t+h}(x) - x =    
\int_{t}^{t+h} \bar v_{s}^{t+h}(x) \int_{\R} \partial_{x} p_{s-t}(x-y) 
\bigl( Y_{s}(y) - Y_s(x) \bigr) \ud y  \ud s
 +  O \bigl( \exp(2 \vert x \vert) h^{1/2 + \alpha} \bigr). 
\end{equation*}

It then remains to look at the first term in the right-hand side of \eqref{eq:4:1:14:1}. The point is to expand 
$v_{t}^{t+h}(x)$ on the same model as $u_{t}^{t+h}(x)$ right above. Basically, the same expansion holds but, because of the 
derivative in the definition of $v_{t}^{t+h}(x) = \partial_{x} u_{t}^{t+h}(x)$, we loose $1/2$ in the power of $h$ in the Landau notation. Therefore, for $t \leq s \leq t+h$,
the above expansion turns into 
\begin{equation*}
\bar v_{s}^{t+h}(x) - 1 =    
\int_{s}^{t+h} \bar v_{r}^{t+h}(x) \int_{\R} {\partial_{x}^2} p_{r-t}(x-y)
\bigl( Y_{{r}}(y) - Y_{{r}}(x) \bigr) \ud y  \ud r
 +  O \bigl( \exp(2\vert x \vert) h^{\alpha} \bigr). 
\end{equation*}
Using once again the fact that $v^{t+h}$ is $(\alpha-\epsilon)/2$-H\"older continuous in time (locally in space, the H\"older constant being at most of exponential growth), we obtain
\begin{equation*}
\begin{split}
\bar v_{s}^{t+h}(x) - 1 &=    
\int_{s}^{t+h} \int_{\R} \partial^2_{x} p_{r-t}(x-y) 
\bigl( Y_{{r}}(y) - Y_{{r}}(x) \bigr) \ud y  \ud r
\\
&\hspace{15pt} + \int_s^{t+h} \bigl( \bar v_{r}^{t+h}(x) - 1 \bigr) \int_{\R} \partial^2_{x} p_{r-t}(x-y) \bigl( Y_{{r}}(y) - Y_{{r}}(x) \bigr) \ud y  \ud r
 +  O \bigl( \exp(2 \vert x \vert) h^{\alpha} \bigr) 
\\
&= Z_{s}^{t+h}(x)  +   O \biggl( 
\exp(2 \vert x \vert) \biggl[
h^{\alpha}
+  h^{(\alpha-\epsilon)/2} 
\int_{s}^{t+h}  (\rho-t)^{-1+ \alpha/2} \ud \rho \biggr] \biggr). 
\end{split}
\end{equation*}
The last term can be bounded by $O(\exp(2\vert x \vert) h^{\alpha - \epsilon/2})$. 
Now, by \eqref{eq:4:1:14:6},
\begin{equation}
\begin{split}
u_{t}^{t+h}(x) - x &=    
\int_{t}^{t+h} \bigl( 1+ Z_{s}^{t+h}(x) \bigr) \int_{\R} \partial_{x} p_{s-t}(x-y) \bigl( Y_{s}(y) - Y_s(x) \bigr) \ud y  \ud s
\\
&\hspace{5pt} +  
\int_{t}^{t+h}  \int_{\R} \partial_{x} p_{s-t}(x-y) \int_{x}^y  \bigl( Z_{s}^{t+h}(z) - Z_{s}^{t+h}(x) 
\bigr) \ud Y_{s}(z) \ud y  \ud r 
\\
&\hspace{5pt}+ O  \biggl( \exp(2 \vert x \vert) \biggl[ 
h^{\alpha - \epsilon/2} 
\int_{t}^{t+h} (s-t)^{-1/2+\alpha/2} ds +    h^{1+\epsilon} \biggr] \biggr). 
\end{split}
\end{equation}
It thus remains to bound 
\begin{equation*}
\int_{t}^{t+h} Z_{s}^{t+h}(x) \int_{\R} \partial_{x} p_{s-t}(x-y)  \bigl( Y_{s}(y) - Y_{s}(x) \bigr) \ud y  \ud s. 
\end{equation*}
By \eqref{eq:Z}, it is plain to see that $Z_{s}^{t+h}(x) = O( \exp(2 \vert x \vert) 
h^{\alpha/2})$. Then,
the above term must at most of order $O(\exp(2 \vert x \vert) h^{1/2+\alpha})$,
from which the proof of the proposition is easily completed. 

In order to complete the proof of Remark \ref{rem:drift}, it remains to show the announced bound for
\begin{equation*}
\int_{t}^{t+h}  \int_{\R} \partial_{x} p_{s-t}(x-y) \int_{x}^y  Z_{s}^{t+h}(z)  \ud Y_{s}(z) \ud y  \ud s. 
\end{equation*}
We already have a bound when $Z_{s}^{t+h}(z)$ is replaced by $Z_{s}^{t+h}(x)$. By \eqref{eq:4:1:14:3}, we also have a bound when $Z_{r}^{t+h}(z)$ is replaced
by $Z_{r}^{t+h}(z) - Z_{r}^{t+h}(x)$.  
\end{proof}

\subsection{Purpose}
The goal is now to prove that Theorem \ref{thm:expansion} and Proposition \ref{prop:drift} are sufficient to define a differential calculus  for which the infinitesimal variation $dX_{t}$ reads
\begin{equation}
\label{eq:1:10:3}
\ud X_{t} = \ud B_{t} +  b(t,X_{t},\ud t), \quad t \in [0,T),
\end{equation}
or, in a macroscopic way,
$X_{t} = X_{0} + B_{t} + \int_{0}^t b(s,X_{s},\ud s)$, which gives a sense to \eqref{eq:18:1:1}. In that framework, Proposition 
\ref{prop:drift} and Remark \ref{rem:drift} give some insight into the shape of the drift.

As explained below, we are able to define a stochastic calculus in such a way that the process $(\int_{0}^t b(s,X_{s},\ud s))_{0 \leq t \leq T}$ has a
H\"older continuous version, with $(1+\alpha)/2-\epsilon$ as H\"older exponent, for $\epsilon>0$ as small as desired, thus making 
$(X_{t})_{0 \leq t \leq T}$ a Dirichlet process. 
{More generally, we manage to give a sense 
to the integrals
$\int_{0}^T \psi_{t} \ud X_{t}$ and $\int_{0}^T \psi_{t} b(t,X_{t},\ud t)$
for a large class of integrands
$(\psi_{t})_{0 \leq t \leq T}$, thus making meaningful the identity
\begin{equation*}
\int_{0}^T \psi_{t} \ud X_{t} = \int_{0}^T \psi_{t} \ud B_{t} + \int_{0}^T \psi_{t} b(t,X_{t},\ud t).
\end{equation*}
The above integrals will be constructed} with respect to processes $(\psi_{t})_{0 \leq t \leq T}$ that are progressi- -vely-measurable and $(1-\alpha)/2+\epsilon$ H\"older continuous in $L^p$ for some $p>2$ and some $\epsilon>0$.  The construction of the integral consists of a mixture of Young's and It\^o's integrals. Precisely, the progressive-measurability of $(\psi_{t})_{0 \leq t \leq T}$ permits to `get rid of' the martingale increments in $X$ that are different from the Brownian ones and thus to focus on the function $b$ only in order to define the non-Brownian part of the dynamics. Then, the H\"older property of $(\psi_{t})_{0 \leq t \leq T}$ permits to integrate with respect to $(b(t,X_{t},\ud t))_{0 \leq t \leq T}$ in a Young sense. For that reason, the resulting integral is called a stochastic Young integral. 
It is worth mentioning that it permits to consider within the same framework integrals defined with respect to the martingale part of $X$ and integrals defined with respect to the zero quadratic variation part of $X$. 
{
Following the terminology used in \cite{cat:gub:12}, in which the authors address a related problem 
(see Remark \ref{sub:cat:gub} below for a precise comparison), 
the Young integral with respect to $(b(t,X_{t},\ud t))_{0 \leq t \leq T}$
may be called `nonlinear'.} 

The construction we provide below is given in a larger set-up. In the whole section, we thus use the following notation: $(\Omega,({\mathcal F}_{t})_{t \geq 0},\P)$ denotes a filtered probability space satisfying the usual conditions; moreover, for any $0 \leq s \leq t$, ${\mathcal S}(s,t)$ denotes the set 
$\{s' \in [0,s], t' \in [0,t], s' \leq t'\}$. The application to \eqref{eq:1:10:1} is discussed in Subsection \ref{subse:5:4}. 

\subsection{$L^p$ Construction of the Integral}

\subsubsection{Materials}
\label{subsubse:7:1:1}
We are given a real $T>0$ and a continuous progressively-measurable process $(A(s,t))_{0 \leq s \leq t \leq T}$ in the sense that, for any $0 \leq s \leq t$, the mapping $\Omega \times {\mathcal S}(s,t) \ni (\omega,s',t') \mapsto A(s',t')$ is measurable for the product $\sigma$-field
${\mathcal F}_{t} \otimes {\mathcal B}({\mathcal S}(s,t))$
 and the mapping ${\mathcal S}(T,T) \ni (s,t) \mapsto A(s,t)$ is continuous. We assume that there exist a constant $\Gamma \geq 0$, three exponents $\varepsilon_{0} \in (0,1/2]$, $\varepsilon_{1},\varepsilon_{1}' >0$ and a real $q \geq 1$ 
such that,
for any $0 \leq t \leq t+h \leq t+h' \leq T$, 
\begin{equation}
\label{eq:27:3:1}
\begin{split}
&{\mathbb E} \bigl[ 
\bigl\vert {\mathbb E} \bigl[ A(t,t+h) \vert {\mathcal F}_{t} \bigr] \bigr\vert^q \bigr]^{\frac1q} \leq \Gamma h^{\frac{1}2+\varepsilon_{0}},
\\
&{\mathbb E} \bigl[ \vert A(t,t+h) \vert^q \bigr]^{\frac1q} \leq \Gamma h^{\frac{1}2},
\\
&{\mathbb E} \bigl[ \bigl\vert {\mathbb E} \bigl[ A(t,t+h) + A(t+h,t+h') - A(t,t+h') \vert {\mathcal F}_{t} \bigr]
\bigr\vert^q \bigr]^{\frac1q} \leq \Gamma (h')^{1+\varepsilon_{1}},
\\
&{\mathbb E} \bigl[ \vert A(t,t+h) + A(t+h,t+h') - A(t,t+h') \vert^q  \bigr]^{\frac1q}
 \leq \Gamma (h')^{\frac{1}{2} (1+\varepsilon_{1}')}.
\end{split}
\end{equation} 
In the framework of \eqref{eq:1:10:3}, 
we have in mind to choose
$A(t,t+h)=X_{t+h}-X_{t}$ or $A(t,t+h) = B_{t+h} - B_{t}$, in which cases $A$ has an additive structure and $\varepsilon_{1}$ and $\varepsilon_1'$ can be chosen as large as desired, 
or $A(t,t+h)=b(t,X_{t},h)$, in which case $A$ is not additive. The precise 
application to \eqref{eq:1:10:3} is detailed in Subsection 
\ref{subse:5:4}. 
Generally speaking,
we  call $A(t,t+h)$ a \textit{pseudo-increment}. 
Considering pseudo-increments instead of increments (that enjoy, in comparison with, an additive property) allows more flexibility
and permits, as just said, to give a precise meaning to $b(t,X_{t},\ud t)$ in \eqref{eq:1:10:3}.
The strategy is then to split $A(t,t+h)$ into two pieces:
\begin{equation}
\label{eq:1:4:2b}
\begin{split}
R(t,t+h) := {\mathbb E} \bigl[ A(t,t+h) \vert {\mathcal F}_{t} \bigr],
\quad M(t,t+h) := A(t,t+h) - {\mathbb E} \bigl[ A(t,t+h) \vert {\mathcal F}_{t} \bigr],  
\end{split}
\end{equation}
$M(t,t+h)$ being seen as a sort of martingale increment and $R(t,t+h)$ as a sort of drift.

We are also given a continuous progressively-measurable process $(\psi_{t})_{0 \leq t \leq T}$ and we assume that, 
for an exponent $\varepsilon_{2} < \varepsilon_{0}$  and for any $0 \leq t \leq t+h \leq T$,
\begin{equation}
\label{eq:27:3:2:b}
{\mathbb E} \bigl[ \vert \psi_{t} \vert^{q'} \bigr]^{\frac{1}{q'}} \leq \Gamma, \quad {\mathbb E} \bigl[ \psi_{t+h} - \psi_{t} \vert^{q'} \bigr]^{\frac{1}{q'}} 
\leq \Gamma  h^{\frac{1}2- \varepsilon_{2}},
\end{equation}
for some $q' \geq 1$. We then let $p=qq'/(q+q')$ so that 
$1/p=1/q+1/q'$.

\subsubsection{Objective}
The aim of the subsection is to define the stochastic integral
$\int_{0}^T \psi_{t} A(t,t+\ud t)$ as an $L^p(\Omega,\P)$ version of the Young integral. In comparison with the standard version of the Young integral, the $L^p(\Omega,\P)$ construction will benefit from the martingale structure of the pseudo-increments $(M(t,t+h))_{0 \leq t \leq t+h \leq T}$, the integral being defined as the $L^p(\Omega,\P)$ limit of Riemann sums as the step size of the underlying subdivision tends to $0$. Given a subdivision $\Delta = \{0=t_{0} < t_{1} < \dots < t_{N}=T\}$, we thus define the $\Delta$-Riemann sum
\begin{equation}
\label{eq:4:4:4}
S(\Delta) := \sum_{i=0}^{N-1} \psi_{t_{i}} A(t_{i},t_{i+1}). 
\end{equation}
We emphasize that this definition is exactly the same as the one used to define It\^o's integral:
on the step $[t_{i},t_{i+1}]$, the process $\psi$ is approximated by the value at the initial point $t_{i}$. For that reason, we will say that the Riemann sum is \textit{adapted}. In that framework, we claim:
\begin{theorem}
\label{thm:young:L2}
There exists a constant $C=C(q,q',\Gamma,\varepsilon_{0},\varepsilon_{1},\varepsilon_{2})$, such that, given two subdivisions $\Delta \subset \Delta'$, with $\pi(\Delta) \leq 1$,
\begin{equation}
\label{thm:young:L2:1}
\E \bigl[ \vert S(\Delta) - S(\Delta') \vert^p \bigr]^{1/p}
\leq C' \max(T^{1/2},T) \bigl( \pi(\Delta) \bigr)^{\eta},
\end{equation}
where $\pi(\Delta)$ denotes the step size of the subdivision $\Delta$, that is 
$\pi(\Delta) := \max_{1 \leq i \leq N} [ t_{i} - t_{i-1}]$, and with 
$\eta := \min(\varepsilon_{0} - \varepsilon_{2}, \varepsilon_{1},\varepsilon_{1}'/2)$.
\end{theorem}

For general partitions $\Delta$ and $\Delta'$ (without any inclusion requirement),
Theorem \ref{thm:young:L2} applies to the pairs $(\Delta,\Delta \cup \Delta')$ and $(\Delta',\Delta \cup \Delta')$, so that 
\eqref{thm:young:L2:1} holds in that case as well provided $\pi(\Delta)$ in the right-hand side is replaced by $\max(\pi(\Delta),\pi(\Delta'))$. We deduce that $S(\Delta)$ has a limit in $L^p(\Omega,\P)$ as $\pi(\Delta)$ tends to $0$. We call it the stochastic Young integral of $\psi$ with respect to the pseudo-increments of $A$. 

\subsubsection{Proof of Theorem \ref{thm:young:L2}. First Step.}
\label{subse:proof:young:1}
First, we consider the case where the two subdivisions $\Delta$ and $\Delta'$, $\Delta$ being included in $\Delta'$, are not so different one from each other. Precisely, given $\Delta = \{0=t_{0} < t_{1} < \dots < t_{N} = T\}$ and 
$\Delta' = \Delta \cup \{t_{1}' <\dots < t_{L}'\}$ ($L \geq 1$), the 
$(t_{i})_{1 \leq i \leq N}$'s and the 
$(t_{j}')_{1 \leq j \leq L}$'s being pairwise distinct,
we assume that, between two consecutive points in $\Delta$, there is at most one point in $\Delta'$. For any $j \in \{1,\dots,L\}$, we then denote by $s_{j}^-$ and $s_{j}^+$ the largest and smallest points in $\Delta$ such that 
$s_{j}^- < t_{j}' < s_{j}^+$. We have $t_{j}' < s_{j}^+ \leq s_{j+1}^- < t_{j+1}'$ for $1 \leq j \leq L-1$.
 We then claim:
\begin{lemma}
\label{lem:estimate:young:1}
Under the above assumption, the estimate \eqref{thm:young:L2:1} holds with $\pi(\Delta)$ replaced by 
$\rho(\Delta' \setminus \Delta)$, where $\rho(\Delta' \setminus \Delta) := \sup_{1 \leq j \leq L} [ s_{j}^+ - s_{j}^- ]$. 
\end{lemma}
\begin{proof}[Proof of Lemma \ref{lem:estimate:young:1}.]
\textit{(i)} As a first step, we compute the difference $S(\Delta') - S(\Delta)$. We write
\begin{equation*}
S(\Delta') - S(\Delta) = \sum_{j=1}^L \bigl[ S( \Delta^j) - S(\Delta^{j-1}) \bigr],
\end{equation*}
with 
$\Delta^j = \Delta \cup \{t_{1}',\dots,t_{j}'\}$, for $1 \leq j \leq L$, and $\Delta^0 = \Delta$.  Then,
\begin{equation*}
\begin{split}
S(\Delta^j) &= S(\Delta^{j-1}) + \psi_{s_{j}^-} A(s_j^-,t_{j}') + \psi_{t_{j}'} A(t_{j}',s_{j}^+) 
- \psi_{s_{j}^-} A(s_{j}^-,s_{j}^+)
\\
&= S(\Delta^{j-1}) + \bigl( \psi_{t_{j}'} 
- \psi_{s_{j}^-} \bigr) A(t_{j}',s_{j}^+)
+ \psi_{s_{j}^-} \bigl( A(s_{j}^-,t_{j}') + A(t_{j}',s_{j}^+) - A(s_{j}^-,s_{j}^+) \bigr).  
\end{split}
\end{equation*}
Therefore,
\begin{equation}
\label{eq:1:4:2}
\begin{split}
S(\Delta') - S(\Delta) &= \sum_{j=1}^L \bigl( \psi_{t_{j}'}           
- \psi_{s_{j}^-} \bigr) M(t_{j}',s_{j}^+)Ê + \sum_{j=1}^L \bigl( \psi_{t_{j}'}           
- \psi_{s_{j}^-} \bigr) R(t_{j}',s_{j}^+)ÊÊÊÊÊÊÊÊÊÊÊÊÊÊÊÊ   
\\
&\hspace{15pt} + \sum_{j=1}^L 
\psi_{s_{j}^-} \bigl( A(s_{j}^-,t_{j}') + A(t_{j}',s_{j}^+) - A(s_{j}^-,s_{j}^+) \bigr)
\\
&:= \delta_{1} S(\Delta,\Delta',M) + \delta_{1} S(\Delta,\Delta',R) + \delta_{2} S(\Delta,\Delta').  
\end{split}
\end{equation}

\textit{(ii)}
We first investigate $\delta_{1} S(\Delta,\Delta',M)$. The process
$( \sum_{j=1}^{\ell} ( \psi_{t_{j}'}           
- \psi_{s_{j}^-}) M(t_{j}',s_{j}^+) )_{0 \leq \ell \leq L}$
is a discrete stochastic integral and thus a martingale with respect to the filtration $({\mathcal F}_{s_{\ell}^+})_{0 \leq \ell \leq L}$, with the convention that $s_{0}^-=s_{0}^+ = 0$. The sum of the squares of the increments is given by 
$\sum_{j=1}^{L}  ( \psi_{t_{j}'}           
- \psi_{s_{j}^-} )^2 ( M(t_{j}',s_{j}^+) )^2$.
By the second line in \eqref{eq:27:3:1} and by \eqref{eq:27:3:2:b}, we observe from Minkowski's inequality first and then 
from H\"older's inequality (recalling $1/p=1/q+1/q'$) that there exists a constant $C$ such that
\begin{equation*}
\begin{split}
\E \biggl[ \biggl\vert \sum_{j=1}^{L}  \bigl( \psi_{t_{j}'}           
- \psi_{s_{j}^-} \bigr)^2 \bigl( M(t_{j}',s_{j}^+) \bigr)^2 \biggr\vert^{\frac{p}2} \biggr]^{\frac{2}p}
& \leq  \sum_{j=1}^{L}
\E \Bigl[  \bigl( \psi_{t_{j}'}           
- \psi_{s_{j}^-} \bigr)^{p} \E \Bigl[ \bigl( M(t_{j}',s_{j}^+) \bigr)^p \vert {\mathcal F}_{t_{j}'} \Bigr] \Bigr]^{\frac{2}p}
\\
&\leq C \sum_{j=1}^{L}
\bigl(t_{j}' - s_{j}^- \bigr)^{(1 - 2 \varepsilon_{2})} \bigl(s_{j}^+ - t_{j}' \bigr)
\leq C T  \bigl( \rho(\Delta' \setminus \Delta) \bigr)^{\eta_{1}},
\end{split}
\end{equation*}
with $\eta_{1} := 1- 2 \varepsilon_{2} \geq 2 (\varepsilon_{0}-\varepsilon_{2})$, where we have used 
$s_{j}^- < t_{j}' < s_{j}^+$. 
By discrete Burkholder-Davis-Gundy inequalities, we deduce that
$\E [ \vert \delta_{1} S(\Delta,\Delta',M) \vert^p ]^{1/p} \leq C T^{1/2} 
( \rho(\Delta' \setminus \Delta) )^{ \eta_{1}/2}.$
\vspace{5pt}

\textit{(iii)} We now turn to $\delta_{1} S(\Delta,\Delta',R)$. In the same way, 
by the first line in \eqref{eq:27:3:1} and by \eqref{eq:27:3:2:b}, 
\begin{equation*}
\begin{split}
{\mathbb E} \bigl[ \bigl\vert \delta_{1} S(\Delta,\Delta',R) \bigr\vert^p \bigr]^{\frac{1}{p}}
&\leq \sum_{j=1}^{L} 
{\mathbb E} \bigl[ \bigl\vert \vert \psi_{t_{j}'}           
- \psi_{s_{j}^-} \vert^p \vert R(t_{j}',s_{j}^+) \vert^p \bigr]^{\frac{1}p}
\\
&\leq C  \sum_{j=1}^{L}
 \bigl(t_{j}' - s_{j}^- \bigr)^{1/2 -  \varepsilon_{2}} \bigl(s_{j}^+ - t_{j}' \bigr)^{1/2+\varepsilon_{0}}
 \leq C T  \bigl( \rho(\Delta' \setminus \Delta) \bigr)^{\eta_{2}},
\end{split}
\end{equation*}
with $\eta_{2} := \varepsilon_{0} - \varepsilon_{2}$. 
Therefore,
$\E [ \vert \delta_{1} S(\Delta,\Delta',R) \vert^p]^{1/p} \leq C T
\bigl( \rho(\Delta' \setminus \Delta) \bigr)^{ \eta_{2}}$.
\vspace{5pt}

\textit{(iv)} We finally investigate $\delta_{2} S(\Delta,\Delta')$. We split it into two pieces:
\begin{equation}
\label{eq:4:4:5}
\begin{split}
\delta_{2} S(\Delta,\Delta') 
&= \sum_{j=1}^L 
\psi_{s_{j}^-} R'(s_{j}^-,t_{j}',s_{j}^+)
+ \sum_{j=1}^L \psi_{s_{j}^-} M'(s_{j}^-,t_{j}',s_{j}^+),
\\
&: = \delta_{2} S(\Delta,\Delta',R') + \delta_{2} S(\Delta,\Delta',M'), 
\end{split}
\end{equation}
with
\begin{equation*}
\begin{split}
&R'(s_{j}^-,t_{j}',s_{j}^+) := \E \bigl[ A(s_{j}^-,t_{j}') +  A(t_{j}',s_{j}^+) 
- A(s_{j}^-,s_{j}^+) \big\vert {\mathcal F}_{s_{j}^-} \bigr],
\\
&M'(s_{j}^-,t_{j}',s_{j}^+) := A(s_{j}^-,t_{j}') +  A(t_{j}',s_{j}^+) 
- A(s_{j}^-,s_{j}^+)  -  R'(s_{j}^-,t_{j}',s_{j}^+).  
\end{split}
\end{equation*}
By the third line in \eqref{eq:27:3:1} and by \eqref{eq:27:3:2:b}, we have, 
with $\eta_{3} := \varepsilon_{1}$,
$\E[ \vert \delta_{2} S(\Delta,\Delta',R') \vert^p ]^{1/p} \leq C T 
( \rho(\Delta' \setminus \Delta))^{\eta_{3}}$.

We finally tackle $\delta_{2} S(\Delta,\Delta',M')$. We notice that it generates a discrete time martingale with respect to the filtration 
$({\mathcal F}_{s_{\ell}^+})_{0 \leq \ell \leq L}$. As in the second step, we compute the $L^{p/2}(\Omega,\P)$ norm of the sum of the squares of the increments. By the last line in \eqref{eq:27:3:1}, it is given by 
\begin{equation*}
\begin{split}
\E \biggl[ \biggl\vert \sum_{j=1}^L  \psi_{s_{j}^-}^2 \bigl( M'(s_{j}^-,t_{j}',s_{j}^+) \bigr)^2  
\biggr\vert^{\frac{p}{2}}
\biggr]^{\frac{2}p} &\leq \sum_{j=1}^L \E \Bigl[  \psi_{s_{j}^-}^p \E \Bigl[ \bigl( M'(s_{j}^-,t_{j}',s_{j}^+) \bigr)^p \vert {\mathcal F}_{s_{j-1}^+}
\Bigr] \Bigr]^{\frac{2}p} 
\leq C   T  \bigl( \rho(\Delta' \setminus \Delta) \bigr)^{\eta_{4}}, 
\end{split}
\end{equation*}
with $\eta_{4} := \varepsilon_{1}'$. By discrete Burkholder-Davis-Gundy inequality, 
$\E[ \vert \delta_{2} S(\Delta,\Delta',M') \vert^p]^{1/p} \leq C T^{1/2}
 ( \rho(\Delta' \setminus \Delta))^{ \eta_{4}/2}$. 
Putting $(i)$, $(ii)$, $(iii)$ and $(iv)$ together, this completes the proof. 
\end{proof}
\subsubsection{Proof of Theorem \ref{thm:young:L2}. Second Step.}
\label{subsubse:proof:2ndstep}
We now consider the general case when 
$\Delta \subset \Delta'$ ($\Delta' \not = \Delta$) without any further assumption on the difference $\Delta' \setminus \Delta$. 

As above, we denote the points in $\Delta$ by $t_{1},\dots,t_{N}$. The points in the difference $\Delta' \setminus \Delta$ are denoted in the following way. For $i=1,\dots,N$, we denote by $t_{1,i}',\dots,t_{L_{i},i}'$ the points in the intersection
$(\Delta' \setminus \Delta) \cap (t_{i-1},t_{i})$, where $L_{i}$ denotes the number of points in 
$(\Delta' \setminus \Delta) \cap (t_{i-1},t_{i})$. Each $L_{i}$ may be written as $L_{i} = 2 \ell_{i} + \varepsilon_{i}$ where $\ell_{i} \in {\mathbb N}$ and 
$\varepsilon_{i} \in \{0,1\}$. We then define 
$\Delta_{1}'$ as the subdivision made of the points that are in $\Delta$ together with the points
\begin{equation*}
\bigl\{ \{ t_{2 \ell,i}', \ \ell = 1, \dots, \ell_{i} \} \cup 
\{t_{2 \ell_{i}+1} \ {\rm if} \ \varepsilon_{i}=1\} \bigr\} \quad {\rm whenever} \ \ell_{i} \geq 1,
\quad {\rm for} \ i = 1, \dots,N. 
\end{equation*}
This says that, to construct $\Delta_{1}'$, we delete, for any $i=1,\dots, N$, 
the point $t_{1,i}'$ if $L_{i}=1$ and
the points that are in 
$(\Delta' \setminus \Delta) \cap (t_{i-1},t_{i})$ and that have an odd index
$2 \ell-1$ with $1 \leq \ell \leq \ell_{i}$ if $L_{i}>1$ (so that the last point is kept even if labelled by an odd integer when 
$\ell_{i}\geq 1$).  
By construction, $\Delta_{1}'$ and $\Delta'$ satisfy the assumption of Subsection \ref{subse:proof:young:1}, so that 
\begin{equation*}
\bigl\| S(\Delta_{1}') - S(\Delta') \bigr\|_{L^p(\Omega,P)}
 \leq C \max(T^{1/2},T) \bigl[ \rho(\Delta' \setminus \Delta_{1}') \bigr]^{\eta}. 
\end{equation*}

It holds $\Delta_{1}' \supset \Delta$. If $\Delta_{1}' \not = \Delta$, we then build a new subdivision $\Delta_{2}'$ as the subdivision associated with $\Delta_{1}'$ in the same manner as 
$\Delta_{1}'$ is associated with $\Delta'$. We then obtain 
\begin{equation}
\label{eq:1:4:12}
\bigl\| S(\Delta_{2}') - S(\Delta_{1}') \bigr\|_{L^p(\Omega,\P)}
 \leq C \max(T^{1/2},T) \bigl[ \rho(\Delta_{1}' \setminus \Delta_{2}') \bigr]^{\eta}. 
\end{equation}
We then carry on the construction up until we reach $\Delta_{M}' = \Delta$ for some integer $M \geq 1$. We notice that such an $M$ does exist: by construction each $\Delta_{j}'$ contains $\Delta$ and $\sharp[\Delta_{j}'] < \sharp[\Delta_{j-1}']$ (with the convention $\Delta_{0}' = \Delta'$).

We now make an additional assumption: We assume that $\Delta'$ is a dyadic subdivision, that is $\Delta' = \{ 2^{-P} k T, 0 \leq k \leq 2^P\}$ for some 
$P \geq 1$. This says that $\Delta$ is also made of dyadic points of order $P$. We denote by $Q$ the unique integer such that 
\begin{equation*}
\max(L_{i},1 \leq i \leq N)=2^Q + r \quad \textrm{with} \ 0 \leq r \leq 2^Q- 1, 
\end{equation*}
and by $i_{Q}$ some index such that $L_{i_{Q}} = 2^Q + r$. At the first step, the $2^Q$ first points in 
$(\Delta' \setminus \Delta) \cap (t_{i_{Q}-1},t_{i_{Q}})$ are reduced into $2^{Q-1}$ points. At the second step, they are reduced into $2^{Q-2}$ points and so on... Therefore, it takes 
$Q$
steps to reduce the $2^Q$ first points in 
$(\Delta' \setminus \Delta) \cap (t_{i_{Q}-1},t_{i_{Q}})$ into a single one. Meanwhile, it takes at most $Q$ steps to reduce the $r$ remaining points in $(\Delta' \setminus \Delta) \cap (t_{i_{Q}-1},t_{i_{Q}})$ into a single one (without any interferences between the two reductions). We deduce that, after the $Q$th step, there are at most two operations to perform to reduce 
$\Delta_{Q}'$ into $\Delta$. This says that $M$ is either $Q+1$ or $Q+2$ and that, at each step $j \in \{1,\dots,Q\}$ of the induction, we are doubling the step size 
$\rho(\Delta_{j-1}' \setminus \Delta_{j}')$, that is 
\begin{equation*}
\rho(\Delta_{j-1}' \setminus \Delta_{j}') = 2^{j-1} \rho(\Delta' \setminus \Delta_{1}' ), \quad j=1,\dots,Q, 
\end{equation*}
so that 
\begin{equation*}
\rho(\Delta' \setminus \Delta_{1}') \leq 2^{-(Q-1)} \pi(\Delta), \quad \textrm{and} \quad 
\rho(\Delta_{j-1}' \setminus \Delta_{j}')
\leq 2^{j-Q} \pi(\Delta), \quad j=1,\dots,Q.
\end{equation*}
Therefore, 
$\rho(\Delta_{j-1}' \setminus \Delta_{j}')
\leq 2^{j-M+2} \pi(\Delta)$, $j=1,\dots,M$.
By extending \eqref{eq:1:4:12} to each of the steps of the induction, we get  (up to a new value of $C$)
\begin{equation}
\label{eq:31:3:1}
\bigl\| S(\Delta') - S(\Delta) \bigr\|_{L^p(\Omega,\P)} 
\leq C \max(T^{1/2},T)  \bigl[ \pi(\Delta) \bigr]^{\eta} \sum_{j = 0}^M 2^{\eta(j-M)} 
 \leq C \max(T^{1/2},T) 
\bigl[ \pi(\Delta) \bigr]^{\eta}.
\end{equation}

When $\Delta$ and $\Delta'$ contain non-dyadic points (so that they are different from $\{0,T\}$), we can argue as follows. 
We can find a dyadic subdivision, denoted by $D_{2}$, such that, in any open interval delimited by two consecutive points in $D_{2}$, 
there is at most one element of $\Delta$. Then, we remove points from $D_{2}$ to obtain a minimal subdivision $D_{1}$, made of dyadic points, such that, in any open interval delimited by two consecutive points in $D_{1}$, there is exactly one element of $\Delta$.
In such way, in any open interval 
delimited by two consecutive points in $\Delta$, there is at most one point in $D_{1}$. Therefore, we can apply Lemma \ref{lem:estimate:young:1}
to $(D_{1},D_{1} \cup \Delta)$ and $(\Delta,D_{1} \cup \Delta)$. We get 
\begin{equation*}
\bigl\| S(D_{1}) - S(\Delta) \bigr\|_{L^p(\Omega,\P)} 
 \leq C \max(T^{1/2},T) 
\bigl[ \max\bigl(\pi(D_{1}),\pi(\Delta) \bigr) \bigr]^{\eta}
 \leq C' \max(T^{1/2},T) 
\bigl[ \pi(\Delta)  \bigr]^{\eta},
\end{equation*}
since $\pi(D_{1}) \leq 2 \pi(\Delta)$. By the same argument, we can find a dyadic subdivision $D_{1}'$ for which the above inequality applies with $(D_{1},\Delta)$ replaced by $(D_{1}',\Delta')$. Then, we can find a dyadic subdivision $D$ such that both 
$D_{1} \subset D$ and $D_{1}' \subset D$. Applying \eqref{eq:31:3:1} to $(D_{1},D)$ and to $(D_{1}',D)$, we can bound the difference between $S(D_{1}')$ and $S(D_{1})$. The result follows.

\subsection{Further Properties of the Integral}
\label{subse:5:3}
\subsubsection{Extension of the Integral}
Given the decomposition \eqref{eq:1:4:2b}, it is worth noting that both the integrals
$\int_{0}^T \psi_{t} M(t,t+dt)$ and $\int_{0}^T \psi_{t} R(t,t+dt)$
are also defined as $L^p$ limits of the associated adapted Riemann sums. The main point is to check that Lemma \ref{lem:estimate:young:1} applies to $S_M$ and $S_{R}$, where, with the same notation as in \eqref{eq:4:4:4}, 
$S_{M}(\Delta) = \sum_{i=0}^{N-1} \psi_{t_{i}} M(t_{i},t_{i+1})$ and $
S_{R}(\Delta) = \sum_{i=0}^{N-1} \psi_{t_{i}} R(t_{i},t_{i+1})$. 
A careful inspection of the proof of Lemma \ref{lem:estimate:young:1} shows that the non-trivial point is to control the quantities
$\delta_{2} S(\Delta,\Delta',M)$ and 
$\delta_{2} S(\Delta,\Delta',R)$, obtained by replacing $A$ by $M$ and $R$ respectively in the definition of 
$\delta_{2} S(\Delta,\Delta')$ in \eqref{eq:1:4:2}. 
Actually, since we already have a control of the sum of the two terms (as it coincides with $\delta_{2} S(\Delta,\Delta')$ in the proof of Lemma \ref{lem:estimate:young:1}), it is sufficient to control one of them only. 
Clearly, 
\begin{equation*}
\begin{split}
\bigl\| \delta_{2} S(\Delta,\Delta',R) \bigr\|_{L^p(\Omega,\P)} 
&\leq \Bigl\| \sum_{j=1}^{L} \psi_{s_{j}^-} \bigl( R(s_{j}^-,t_{j}') + \E \bigl( R(t_{j}',s_{j}^+) \vert {\mathcal F}_{s_{j}^-} \bigr) - R(s_{j}^-,s_{j}^+) \bigr)
\Bigr\|_{L^p(\Omega,\P)}
\\
&\hspace{15pt} + \Bigl\| \sum_{j=1}^{L} \psi_{s_{j}^-} \bigl( R(t_{j}',s_{j}^+) - \E \bigl( R(t_{j}',s_{j}^+) \vert {\mathcal F}_{s_{j}^-} \bigr) \bigr)
\Bigr\|_{L^p(\Omega,\P)}.
\end{split}
\end{equation*}
We emphasize that the first term above is nothing but $\delta_{2} S(\Delta,\Delta',R')$ in \eqref{eq:4:4:5}, for which we 
already have a bound. Therefore, the only remaining point is to control the second term above. Again, we notice that it has a martingale structure, which can be estimated by Burkholder-Davis-Gundy inequality. By the first line in \eqref{eq:27:3:1} and by \eqref{eq:27:3:2:b}, 
\begin{equation*}
\begin{split}
& \E \biggl[ \Bigl\vert \sum_{j=1}^{L} \psi_{s_{j}^-}^2 \Bigl( R(t_{j}',s_{j}^+) - \E \bigl( R(t_{j}',s_{j}^+) \vert {\mathcal F}_{s_{j}^-} \bigr) \Bigr)^2 \Bigr\vert^{\frac{p}2} \biggr]^{\frac{2}p}
\\
&\leq C  \sum_{j=1}^L \E \Bigl[ \psi_{s_{j}^-}^p \bigl( R(t_{j}',s_{j}^+) \bigr)^p \Bigr]^{\frac{2}p}
\leq C' \sum_{j=1}^L \bigl(s_{j}^+ - s_{j}^- \bigr)^{1+ 2 \varepsilon_{0}}
\leq C'' T  \bigl( \rho(\Delta' \setminus \Delta)\bigr)^{2 \varepsilon_{0}},
\end{split}
\end{equation*}
which is enough to conclude that Theorem \ref{thm:young:L2} is also valid when replacing $A$ by $R$ or $M$
in \textsection
\ref{subsubse:proof:2ndstep}. 
Therefore, we are allowed to split the integral of $\psi$ as
$\int_{0}^T \psi_{t}A(t,t+\ud t) = \int_{0}^T \psi_{t} M(t,t+\ud t) + \int_{0}^T \psi_{t} R(t,t+\ud t)$. The reader must pay attention to the fact that neither $M$ nor $R$ must satisfy \eqref{eq:27:3:1} even if $A$ does. The extension of the integral to the case when they are driven by $M$ or $R$ is thus a consequence of the proof of Theorem \ref{thm:young:L2} itself.

\subsubsection{Continuity in Time}
 It is plain to see that the integral is additive in the sense that, for any $0 \leq S \leq S+S' \leq T$, 
 \begin{equation*}
 \int_{0}^{S+S'} \psi_{t}A(t,t+\ud t) = \int_{0}^S \psi_{t} A(t,t+\ud t) + \int_{S}^{S+S'} \psi_{t} A(t,t+\ud t). 
 \end{equation*}
 An important question in practice is the regularity property of the process 
 $[0,T) \ni t \mapsto \int_{0}^t \psi_{s} A(s,s+\ud s)$, which is not well-defined for the moment. 
At this stage of the procedure, each of the integrals is 
uniquely defined up to an event of zero probability which depends on $t$. A continuity argument is thus needed in order to give a sense to all the integrals at the same time.  
By Theorem \ref{thm:young:L2}, we know that, for $h \in (0,1)$,
\begin{equation}
\label{eq:4:4:6}
\biggl\| \int_{t}^{t+h} \psi_{s} A(s,s+\ud s) - \psi_{t} A(t,t+h) \biggr\|_{L^p(\Omega,\P)} \leq C h^{\frac12+ \eta}, 
\end{equation}
for $\eta >0$ as in the statement of Theorem \ref{thm:young:L2}, so that, by the two first lines in \eqref{eq:27:3:1},
$
\| \int_{t}^{t+h} \psi_{s} A(s,s+\ud s) \|_{L^p(\Omega,\P)} \leq C h^{1/2}$,
for possibly new values of $C$. By Kolmogorov's continuity criterion, this says that there exists a H\"older continuous version of the process 
$( \int_{0}^t \psi_{s} A(s,s+\ud s))_{0 \leq t \leq T}$,
with $1/2 - 1/p-\epsilon$ as pathwise  H\"older exponent, for any $\epsilon >0$. 

By the same argument, we notice that there exist H\"older continuous versions of the processes 
$( \int_{0}^t \psi_{s} M(s,s+\ud s) )_{0 \leq t \leq T}$ and $( \int_{0}^t \psi_{s} R(s,s+\ud s))_{0 \leq t \leq T}$.
The H\"older exponent of the second one is actually better. Indeed, noticing that \eqref{eq:4:4:6} also holds for $R$ and taking advantage of the first line in \eqref{eq:27:3:1}, we deduce that  
$\| \int_{t}^{t+h} \psi_{s} R(s,s+\ud s) \|_{L^p(\Omega,\P)} \leq C h^{(1+\eta)/2}$,
so that the pathwise H\"older exponent can be chosen as $(1+\eta)/2-1/p-\epsilon$ for any $\epsilon >0$. 

\subsubsection{Dirichlet decomposition}
It is well-checked that the process 
$( \int_{0}^t \psi_{s} M(s,s+\ud s))_{0 \leq t \leq T}$ is a martingale, thus showing that the integral of $\psi$ with respect to the pseudo-increments of $A$ can be split into two terms: a martingale
and a drift. We expect that, in practical cases, the exponent $p$ can be choose as large as desired: In this setting, the martingale part has $(1/2-\epsilon)$-H\"older continuous paths, for $\epsilon >0$ as small as desired, and the drift part has $(1/2 + \eta-\epsilon)$-H\"older continuous paths, also for $\epsilon >0 $ as small as desired, thus proving that the integral 
is a Dirichlet process.   

\subsection{Application to diffusion processes driven by a distributional drift}
\label{subse:5:4}
We now explain how the stochastic Young integral applies to \eqref{eq:18:1:1}. 
First, we can choose $A(t,t+h) = X_{t+h} - X_{t}$, for $0 \leq t \leq t+h \leq T_{0}$. Then the process $A$ is additive. 
In particular, the two last lines in \eqref{eq:27:3:1} are automatically satisfied with $\varepsilon_{1}$ and 
$\varepsilon_{1}'$ as large as needed. By \eqref{eq:1:10:1}, the second line in \eqref{eq:27:3:1} is also satisfied. 
Finally, we notice that 
\begin{equation*}
{\mathbb E} \bigl[ X_{t+h} - X_{t} \vert {\mathcal F}_{t} \bigr] 
= {\mathbb E} \bigl[ X_{t+h} - X_{t} - \bigl( B_{t+h} - B_{t} \bigr) \vert {\mathcal F}_{t} \bigr],
\end{equation*}
so that, by \eqref{eq:1:10:1} again, the first line in \eqref{eq:27:3:1} is satisfied with $\varepsilon_{0} = \beta/2$. 

With our construction, this permits to define $(\int_{0}^t \psi_{s} \ud X_{s})_{0 \leq t \leq T_{0}}$ for any progressively measurable process $(\psi_{t})_{0 \leq t \leq T_{0}}$ satisfying \eqref{eq:27:3:2:b} with $\varepsilon_{2} < \beta/2$. It also permits to define the integrals $(\int_{0}^t \psi_{s} M(s,s+\ud s))_{0 \leq t \leq T_{0}}$ and $(\int_{0}^t \psi_{s} R(s,s+\ud s))_{0 \leq t \leq T_{0}}$, where
\begin{equation*}
M(t,t+h) = X_{t+h} - X_{t} - {\mathbb E} \bigl[ X_{t+h} - X_{t} \vert {\mathcal F}_{t} \bigr], \quad 
R(t,t+h) = {\mathbb E} \bigl[ X_{t+h} - X_{t} \vert {\mathcal F}_{t} \bigr]. 
\end{equation*}
By \eqref{eq:1:10:2}, we have $R(t,t+h) = {\mathfrak b}(t,X_{t},h)$, so that $(\int_{0}^t  \psi_{s} {\mathfrak b}(s,X_{s},\ud s) )_{0 \leq t \leq T_{0}}$ is well-defined.  

Moreover, 
by Proposition \ref{prop:drift} and by boundedness of the exponential moments of 
$(X_{t})_{0 \leq t \leq T_{0}}$ (see the proof of Theorem
\ref{thm:localmart}), we know that $\hat{R}(t,t+h) 
= (b - {\mathfrak b})(t,X_{t},h)$ also satisfies 
\eqref{eq:27:3:1}, from which we deduce that 
$(\int_{0}^t \psi_{s} (b-{\mathfrak b})(s,X_{s},\ud s))_{0\leq t \leq T_{0}}$ 
and so 
$(\int_{0}^t \psi_{s} b(s,X_{s},\ud s))_{0\leq t \leq T_{0}}$ 
are well-defined.
Actually the exponent in the power of $h$ appearing in the difference 
$(b - {\mathfrak b})(t,X_{t},h)$ being strictly greater than 1, the integral process
$(\int_{0}^t \psi_{s} (b-{\mathfrak b})(s,X_{s},\ud s))_{0\leq t \leq T_{0}}$
must be $0$. We deduce that 
$(\int_{0}^t \psi_{s} b(s,X_{s},\ud s) = \int_{0}^t \psi_{s}{\mathfrak b}(s,X_{s},\ud s))_{0 \leq t \leq T_{0}}$. 

We finally discuss the integral $(\int_{0}^t \psi_{s} M(s,s+\ud s) )_{0 \leq t \leq T}$. We let
\begin{equation*}
\begin{split}
\hat{M}(t,t+h) &=
X_{t+h} - X_{t} - \bigl( B_{t+h} - B_{t} \bigr) - {\mathbb E} \bigl[ X_{t+h} - X_{t} \vert {\mathcal F}_{t} \bigr] 
\\
&=X_{t+h} - X_{t} - \bigl( B_{t+h} - B_{t} \bigr) - {\mathbb E} \bigl[ X_{t+h} - X_{t} 
- \bigl( B_{t+h} - B_{t} \bigr)
\vert {\mathcal F}_{t} \bigr]. 
\end{split}
\end{equation*}
By \eqref{eq:1:10:1}, ${\mathbb E}[ \vert \hat{M}(t,t+h) \vert^q \vert ]^{1/q} \leq C_{q}' h^{(1+\beta)/2}$
for some $C_{q}' \geq 0$, which reads as a super-diffusive bound for the pseudo-increments of $\hat{M}$. It is then well-checked that $(\hat{M}(t,t+h))_{0 \leq t \leq t+h \leq T_{0}}$ fulfills all the requirements in \eqref{eq:27:3:1}. Therefore, the integral $(\int_{0}^t \psi_{s} \hat{M}(s,s+\ud s) )_{0 \leq t \leq T_{0}}$ makes sense. 
By Subsection \ref{subse:5:3}, it is a martingale but  
by the super-diffusive bound of the pseudo-increments 
it must be the null process. Put it differently, only the Brownian part really matters in $M$ and we can justify \eqref{eq:1:10:3} thanks to the equality
\begin{equation*}
\int_{0}^t \psi_{s} \ud X_{s} = \int_{0}^t \psi_{s} \ud B_{s} + \int_{0}^t \psi_{s} b(s,X_{s},\ud s).
\end{equation*}
 
\begin{remark}
\label{sub:cat:gub}
In \cite{cat:gub:12}, the authors already introduced a `nonlinear' version of the Young integral. The motivation 
was similar to ours as the underlying objective was to solve singular differential equations driven by  a distributional (but time-homogeneous) velocity field and perturbed by a rough signal.
The construction suggested therein also consists of an approximation by means of Riemann sums, but the convergence is shown pathwise. The proof relies on a suitable control on the default of additivity of the nonlinear integrator, on 
the model of the third line in \eqref{eq:27:3:1}, but expressed in a pathwise (instead of $L^p$) form. 
We refer to \cite[Theorem 2.4]{cat:gub:12} for the main statement:  Therein,
the pseudo-increment reads $G_{t_{i},t_{i+1}}(f_{t_{i}})$ instead of
$A(t_{i},t_{i+1})$ and the condition $\gamma + \rho \nu >1$ corresponds to 
the condition $1+\varepsilon_{1}>1$ in the third line of \eqref{eq:27:3:1}. 
In the specific framework of singular differential equations driven by a distributional drift 
and a Brownian path, the Young integral is used in order to give a meaning to the drift part, 
exactly as we do here. Anyhow, 
the construction 
by Catellier and Gubinelli relies on a path by path 
time averaging principle, which goes back to Davie's work \cite{dav:07}. 
Our construction is different as it relies on a space averaging principle, inspired by Zvonkin's method \cite{zvo:74}. We indeed make use of the statistical 
behavior of the Brownian motion (and its connection with the heat equation) in order 
to define explicitly the \text{effective} drift $b(t,x,\ud t)$. This explains why our approach is of stochastic nature. 
\end{remark}

\section{Construction of the integral of $Z$ w.r.t. $Y$. Examples.}
\label{sec:yz}
We here address the existence of a rough path structure $({\boldsymbol W}_{t}^T = (W_{t}^T,\W_{t}^T))_{0 \leq t \leq T}$ for the pair 
$W_{t}^T=(Y_{t},Z_{t}^T)$, for $T$ running in some interval $[0,T_{0}]$, $T_{0}>0$, 
the process $(Z_{t}^T)_{0 \leq t \leq T}$ being given by \eqref{eq:Z}. 
The process $\W^T$ is intended to encapsulate the iterated integrals of 
$W^T$, namely $\int_{x}^{y} (W_{t}^{i,T}(z) - W_{t}^{i,T}(x)) \ud W^{j,T}(z)$, 
for $i,j \in \{1,2\}$ and $x,y \in \R$. Here $W^{i,T}_{t}$ and $W^{j,T}_{t}$ denote
the coordinates of $W_{t}^T$, namely $W^{1,T}_{t}(x) = Y_{t}(x)$
and $W^{2,T}_{t}(x) = Z_{t}^T(x)$. 

As we are seeking a `geometric' rough structure, the iterated integrals are expected to be the limits of 
iterated integrals computed along smooth approximations of the paths $(Y_{t})_{0 \leq t \leq T}$
and $(Z_{t}^T)_{0 \leq t \leq T}$, see (1) and (2) in Proposition
\ref{mildsolution:approx}. In particular, if it exists, $\W^T$ must 
share some of the properties satisfied by iterated integrals of smooth paths, among which the 
integration by parts. This means that $\W^{1,1,T}_{t}$ and $\W^{2,2,T}_{t}$ must be given by 
\begin{equation}
\label{eq:W:IPP}
\begin{split}
&\W^{1,1,T}_{t}(x,x') 
:= \tfrac12 \bigl( Y_{t}(x') - Y_{t}(x) \bigr)^2, \quad
\W^{2,2,T}_{t}(x,x') 
:= \tfrac12 \bigl( Z_{t}^T(x') - Z_{t}^T(x) \bigr)^2,
\end{split}
\end{equation}
 and that $\W^{1,2,T}_{t}$ and $\W^{2,1,T}_{t}$ must be connected through
\begin{equation}
\label{eq:20:1:14:1}
\begin{split}
\bigl( \W^{1,2,T}_{t} +  
\W^{2,1,T}_{t} \bigr) (x,x') 
&= \bigl( Y_{t}(x') - Y_{t}(x) \bigr) \bigl( Z_{t}^T(x') - Z_{t}^T(x) \bigr).  
\end{split}
\end{equation}
To sum up, the only challenge for constructing $\W^T$ is to define the `cross-integral'
\begin{equation}
\label{eq:cI}
{\mathscr I}_t^T(x,x') := \W_{t}^{2,1,T}(x,x') = 
\int_{x}^{x'} (Z_{t}^T(y) - Z_{t}^T(x)) \ud Y_{t}(y).
\end{equation}

\subsection{Overview of the results}
 \label{subset:general:setting}
We are given $(Y_{t}(x))_{0 \leq t \leq T_{0},x \in \R}$ satisfying for some $\alpha \in (1/3,1)$
and $\chi,\kappa >0$:
\begin{equation}
\label{eq:4:11:01}
\kappa_{\alpha,\chi}((Y_{t})_{0 \leq t \leq T_{0}}) := \sup_{a \geq 1,0 \leq t \leq T_0} 
\bigl( \|Y_{t}\|_{\alpha}^{[-a,a]}/a^{\chi} \bigr) \leq \kappa < \infty.
\end{equation}
Below, we often write 
$\kappa_{\alpha,\chi}(Y)$
for $\kappa_{\alpha,\chi}((Y_{t})_{0 \leq t \leq T_{0}})$. 
As a first remark, we note that, for  
$T \in [0,T_{0}]$, the process $(Z_{t}^T)_{0 \leq t \leq T}$
in \eqref{eq:Z} has the same regularity as $Y$, uniformly in $T$:
\begin{lemma}
\label{lemmaZ}
Given $T \in [0,T_{0}]$, recall the definition of $Z_{t}^T$ in \eqref{eq:Z}.
There exists a constant $C$ only depending on $T_0$, $\alpha$ and $\chi$ such that
 $\kappa_{\alpha,\chi}((Z_{t}^T)_{0 \leq t \leq T})\leq C\kappa$. 
\end{lemma}
\begin{proof}

To prove $\kappa_{\alpha,\chi}((Z_{t}^T)_{0 \leq t \leq T}) \leq C \kappa$, 
we go back 
to \eqref{eq:3:2:14:1}, 
noticing that $({\mathcal M}v)_{t}$ therein is equal to 
$Z_{t}^T$ when $v \equiv 1$ and recalling that the analysis is split into two parts: $\vert x' - x\vert^2 \leq T-t$
and $T-t < \vert x'-x\vert^2$, the first case only being challenging. 
It is then plain to check that,
for $x,x',\xi \in [-a,a]$, with $a \geq 1$,
 ${\mathcal I}_{1}^{x,x'}(\xi)   \leq C 
\kappa a^{\chi}
\int_{0}^{\vert x' -x\vert^2}
s^{-1+\alpha/2} \ud s
\leq C \kappa a^{\chi}
 \vert x'-x \vert^{\alpha}$. Moreover, following 
\eqref{eq:3:2:14:2} with $\beta=1$, we also have 
${\mathcal I}_{2}^{x,x'} \leq C \kappa a^{\chi} \int_{x}^{x'} \int_{\vert x'-x\vert^2}^{T} s^{-(3-\alpha)/2} \ud s \ud u
\leq C \kappa a^{\chi}
 \vert x'-x\vert^{\alpha}$, 
for $x,x' \in [-a,a]$, which completes the proof.
\end{proof}

In order to construct 
${\mathscr I}_{t,T}(x,y)$ 
in \eqref{eq:cI} 
as a geometric integral, we must specify what 
an approximation of $Y$ is. 
We shall say that a sequence $(Y^n)_{n\geq 0}$ is a \textit{smooth approximation} of $Y$ 
on $[0,T_{0}]$ if, for each $t \in [0,T_{0}]$, 
the function $Y^n_{t} : \R \ni x \mapsto Y^n_{t}(x)$ is a smooth function 
such that $\sup_{n\geq 0}\kappa_{\alpha,\chi}((Y^n_{t})_{0 \leq t \leq T_{0}})<\infty$
and, for any $a\geq 1$, $\lim_{n\to\infty} \|Y^n -Y \|_{0,\alpha'}^{[0,T_{0}] \times [-a,a]}=0$
for any $\alpha' \in(0, \alpha)$. Below, we shall often use the following trick, that holds true for any 
$a \geq 1$ and any $\alpha' \in (0,\alpha)$, 
\begin{equation}
\label{eq:holder:trick}
\left.
\begin{array}{l}
\sup_{n\geq 0}
\| Y^n \|_{0,\alpha}^{[0,T_{0}] \times [-a,a]} < \infty
\\
\lim_{n\to\infty} \|Y^n -Y \|_{\infty}^{[0,T_{0}] \times [-a,a]} 
=0
\end{array}
\right\} \Rightarrow 
\lim_{n \rightarrow \infty} \|Y^n -Y\|_{0,\alpha'}^{[0,T_{0}] \times [-a,a]}=0. 
\end{equation}
In particular, a typical example for $Y^n$ is to let
\begin{equation}
\label{eq:smooth:kernel}
Y^n_{t}(x) := n \int_{\R} Y_{t}(x-y) \rho(ny) \ud y,
\end{equation}
where $\rho$ is a smooth density, $\rho$ and its derivatives being at most of polynomial decay, in which case the 
smooth approximation is said to be constructed by spatial convolution. 

Given a smooth approximation $(Y^n)_{n \geq 1}$ of $Y$,
we may define, 
 for any $T \in [0,T_{0}]$, 
the process $Z^{n,T}$ by replacing $Y$ by $Y^n$ in \eqref{eq:Z}, and then, following
\eqref{eq:cI},
we may let
$${\mathscr I}_t^{n,T}(x,x'):=\int_{x}^{x'} (Z_{t}^{n,T}(y) - Z_{t}^{n,T}(x))\partial_x Y_t^n(y)\ud y,$$
which permits to define the structure $({\boldsymbol W}^{n,T}_{t} = (W_{t}^{n,T},\W^{n,T}_{t}))_{0 \leq 
t \leq T}$ accordingly.

The following lemma then provides a general principle for constructing ${\mathscr I}_{t}^{T}(x,x')$:
\begin{lemma}
\label{lem:rough path}
Suppose that, for any $T \in [0,T_0]$, there exists a function 
${\mathscr I}^{T} : [0,T] 
\times \R^2 \rightarrow \R$ and a smooth approximation $(Y^n)_{n \geq 1}$ of $Y$ such that, 
for some $\alpha'\in (1/3,\alpha)$ and $\chi' > \chi$,
\begin{equation}
\label{eq:4:11:0}
\begin{split}
&\sup_{0 \leq T \leq T_{0}} \sup_{t \in [0,T]}
\sup_{n \geq 1}
\sup_{a \geq 1} \bigl( 
\| {\mathscr I}^{n,T}_{t} \|_{2 \alpha'}^{[-a,a]} / a^{2 \chi'} \bigr) < \infty,
\\
&\forall T \in [0,T_{0}], \ \forall a\geq1, \ \lim_{n\to\infty}
\sup_{0\leq t \leq T}\|\mathscr{I}_t^T-\mathscr{I}_t^{n,T}\|_{2\alpha'}^{[-a,a]}=0.
\end{split}
\end{equation}
Assume without any loss of generality that $\chi' > \chi + \alpha - \alpha'$. 
Then, for any $T\in[0,T_0]$, there exists  $\W^T\in \mathcal{C}([0,T]\times\R^2,\R^4)$ such that the 
pair process $({\boldsymbol W}^T_{t}=(W_{t}^T,\W_{t}^T))_{0 \leq t \leq T}$ is a time dependent geometric rough path
with indices $(\alpha',\chi')$ in the sense that
\begin{enumerate}
\item $
\sup_{0 \leq T \leq T_{0}}
\kappa_{\alpha',\chi'} ({\boldsymbol W}^T) < \infty$ and $
\sup_{n \geq 1} 
\sup_{0 \leq T \leq T_{0}}
\kappa_{\alpha',\chi'} ({\boldsymbol W}^{n,T} = 
(W^{n,T},\W^{n,T})) < \infty$;
\item  
for any $T \in [0,T_{0}]$ and any segment ${\mathbb I} \subset {\mathbb R}$,\\
$\|{\boldsymbol W}^T - {\boldsymbol W}^{n,T}\|_{0,\alpha'}^{[0,T] \times {\mathbb I}}= \|(W^T-W^{n,T},\W^T-\W^{n,T})\|_{0,\alpha'}^{[0,T] \times {\mathbb I}}$ tends to $0$ as $n$ tends to $\infty$.
\end{enumerate}
\end{lemma}
\begin{proof}
The cross integral ${\mathscr I}^T$ being given, the definition of $\W^T$ follows from 
\eqref{eq:W:IPP}
and 
\eqref{eq:20:1:14:1}. 
The point is thus to prove the geometric nature of the rough path ${\boldsymbol W}^T$. 

By \eqref{eq:holder:trick}, we 
have, for 
any $a \geq 1$, 
$\lim_{n \rightarrow \infty}
\| Y^n - Y \|_{0,\alpha'}^{[0,T_{0}] \times [-a,a]} = 0$. 
Moreover,
$\| Y^n_{t} \|_{\alpha'}^{[-a,a]}
\leq (2a)^{\alpha - \alpha'}\| Y^n_{t} \|_{\alpha}^{[-a,a]}
\leq C a^{\alpha - \alpha' + \chi}
\kappa_{\alpha,\chi}(Y^n)$, 
proving that $\sup_{n \geq 1} \kappa_{\alpha',\chi'}(Y^n)< \infty$
if $\chi' \geq \alpha - \alpha' + \chi$. 

Applying Lemma \ref{lemmaZ} to $(Y^n,Z^{n,T})$,
we get
$\sup_{n\geq 0}
\sup_{0 \leq T \leq T_{0}}
\kappa_{\alpha',\chi'}((Z^{n,T}_{t})_{0 \leq t \leq T})<\infty$. 
Now, it is quite standard to see 
that, for any $T \in [0,T_{0}]$ and $a \geq 1$,
$\sup_{0 \leq t \leq T} \sup_{x \in [-a,a]}
\vert
Z^{n,T}_{t}(x) 
- Z^T_{t}(x) \vert$ tends to 
$0$ as $n \to \infty$.
By Lemma \ref{lemmaZ} again,  
for $a \geq 1$ and $T \in [0,T_{0}]$,
the functions $([-a,a] \ni x \mapsto Z^{n,T}_{t}(x) \in \R)_{0 \leq t \leq T,n \geq 1}$
are uniformly $\alpha$-H\"older continuous.
By the same trick as in \eqref{eq:holder:trick},
we easily deduce
that 
$\| Z^{n,T} - Z^T \|_{0,\alpha'}^{[0,T] \times 
[-a,a]}$ tends to $0$. 

In order to complete the proof, it suffices to 
handle the iterated integrals, which follows from 
\eqref{eq:4:11:0} and \eqref{eq:cI} (applied to
the pair $(Y^n,Z^{n,T})$ instead of $(Y,Z)$). 
\end{proof}



Here is the first main statement of this section:
\begin{theorem}
\label{thm:rough:structure:1}
Given $\alpha \in (1/3,1]$ and $\chi >0$,
let $Y \in \C([0,T_{0}] \times \R,\R)$
satisfy 
$\kappa_{\alpha,\chi}((Y_{t})_{0 \leq t \leq T_{0}}) < \infty$ 
(see \eqref{eq:4:11:01}
for the notation)
and
\begin{equation}
\label{eq:21:09:2}
\bigl\vert Y_s (x)-Y_t (x) - \bigl( Y_s(y)-Y_t(y) \bigr) \bigr\vert \leq 
\kappa
a^{\chi}
 \vert s- t \vert^{\nu} \vert x -y \vert^{\mu}, \quad  (s,t) \in [0,T_{0}], \ x,y \in \R,
\end{equation}
for some $\kappa \geq 0$ and $\mu,\nu \geq 0$ with $2 \nu + \mu 
\in ( 1- \alpha,1]$.
Then, $Y$ satisfies the assumptions of 
Lemma \ref{lem:rough path}
with respect to 
any $(\alpha',\chi')$ with $\alpha'<\alpha$ and $\chi' > \chi
+
\alpha - \alpha' + (1/2 - \alpha)_{+}$. In particular, for any $T \in [0,T_{0}]$, 
the pair $W^T = (Y,Z^T)$, with $Z^T$ given by \eqref{eq:Z}, 
may be lifted into a geometric rough path 
${\boldsymbol W}^T =(W^T,\W^T)$ satisfying the conclusions of Lemma \ref{lem:rough path}. 

Moreover, when the smooth approximation used in Lemma \ref{lem:rough path} is constructed by spatial convolution, 
${\boldsymbol W}^T$ does not depend upon  the kernel $\rho$ in 
\eqref{eq:smooth:kernel}.
When $\alpha >1/2$, 
${\boldsymbol W}^T$ 
is always well-defined and remains the same whatever the smooth approximation is (even if
not constructed by convolution).
\end{theorem}

Theorem \ref{thm:rough:structure:1} guarantees that ${\boldsymbol W}^T$ exists
for any $T \in [0,T_{0}]$ under 
some condition on the time-space structure of the environment $(Y_{t})_{0 \leq t \leq T_{0}}$. 
When $Y$ is time homogeneous, 
\eqref{eq:21:09:2} is automotically
satisfied, and 
the iterated integral in \eqref{eq:cI} always 
exists and is geometric under the simple assumption 
that $\kappa_{\alpha,\chi}(Y) < \infty$.
In that case, 
the cross integral ${\mathscr I}^T_{t}(x,x')$ in \eqref{eq:cI}
can be expressed explicitly, see
\eqref{eq:rough time indep:2}
in Lemma \ref{lem:rough time indep} below. Moreover,
a careful inspection of the proof shows that the constraint 
$\chi' > \chi
+ \alpha - \alpha' 
+ (1/2 - \alpha)_{+}$ can be relaxed into 
$\chi' > \chi
 + \alpha - \alpha'$. When $Y$ is time dependent, 
the additional condition \eqref{eq:21:09:2} is imposed. 
It is inspired from the construction of the so-called Young integral
between a H\"older continuous function and the increments of another H\"older continuous function, see \cite{you:36}
and Lemma \ref{lem:young}
below. For instance, if $\alpha >1/2$,
\eqref{eq:21:09:2} is always satisfied with $\mu =\alpha$ and $\nu=0$
and the constraint on $\chi'$ reduces to $\chi' > \chi + \alpha - \alpha'$.  
When $\alpha \leq 1/2$, 
a sufficient condition to imply 
\eqref{eq:21:09:2}
is that $Y$ has some $\beta$-H\"older regularity in time:
$\vert Y_s(y)-Y_t(y) \vert \leq \kappa'
\bigl( 1+ \vert y \vert^{\chi} \bigr)
 \vert s- t \vert^{\beta}$ with $\beta > (1-\alpha)/2$. The bound \eqref{eq:21:09:2} is then satisfied with $\mu =0$ and $\nu=\beta \wedge (1/2)$. 
A more specific case is when $Y_{t}(y)$ can be expanded as $Y_{t}(y) = f_{t}Y(y)$, 
with $f$ $\beta$-H\"older continuous, for $\beta > 1/2-\alpha$,
and $Y \in \C(\R,\R)$ with $\sup_{a \geq 1}[ a^{-\chi} \|Y\|_{\alpha}^{[-a,a]}] < \infty$,
in which case 
 \eqref{eq:21:09:2}
holds with $\mu=\alpha$ and $\nu=\beta \wedge 
(1/2-\alpha/2)$. Notice finally that the 
constraint
$2 \nu + \mu \leq 1$ can be 
easily overcome: When $2\nu + \mu >1$, 
the value of $\nu$ can be decreased for free so that $2\nu+\mu = 1$.

 As mentioned in Introduction, existence of the cross-integral has been also proved within the framework of the KPZ equation by means of general results on rough paths theory applied to Gaussian processes, see 
{\cite{fri:vic:10}, 
\cite[Section 3]{hai:11} and
\cite[Section 7]{hai:13}}. 
Theorem 
\ref{thm:main:thm:polymer}
below is a refinement:

\begin{theorem}
\label{thm:main:thm:polymer}
Let $(\Xi,{\mathcal G},{\mathbf P})$
be a probability space  
with a Brownian sheet
 $(\zeta(t,x))_{t \geq 0, x \in \R}$. 
 Let
 $Y^T(t,x) := \int_{t}^{T} \int_{\mathbb R} p_{s-t}(x-y) \ud \zeta(s,y)$,
for $\{0\leq t\leq T,x\in\R\}$.
For a smooth density $\rho$, $\rho$ and its derivatives being 
at most of polynomial decay, define in the same way 
$Y^{\rho,T}(t,x) := \int_{t}^{T} \int_{\mathbb R} p_{s-t}(x-y) \ud \zeta^{\rho}(s,y)$, 
with $\zeta^\rho(t,x) := \int_{0}^t \int_{\R} \rho(x-y) \ud \zeta(s,y)$.

 Then, for any $T_0>0$, we can find an event $\Xi^\star \in {\mathcal G}$, 
 with ${\mathbf P}(\Xi^\star)=1$, such that, for any realization 
 in $\Xi^\star$, for any 
 $Y^{(b)} \in {\mathcal C}([0,T_{0}] \times \R,\R)$, 
 with $\kappa_{\alpha_{b},\chi_{b}}(Y^{(b)}) < \infty$
 for some $\alpha_{b} >1/2$ and $\chi_{b} >0$, 
 for any approximation sequence 
 $(Y^{n,(b)})_{n \geq 1}$ of $Y^{(b)}$, 
  the function 
 \begin{equation*}
 Y(t,x) = Y^{T_0}(t,x) + Y^{(b)}(t,x), \quad (t,x) \in [0,T_0] \times \R,
 \end{equation*}
 satisfies the assumption of Lemma 
 \ref{lem:rough path}
with respect to any $\alpha \in (0,1/2)$ and any $\chi > \chi_{b} + \alpha_{b}-\alpha$, 
and with respect to 
the smooth approximation $(Y^n = Y^{n \rho(n \cdot),T_0} + Y^{n,(b)})_{n \geq 1}$.
\end{theorem}
Theorem 
\ref{thm:main:thm:polymer} is specifically designed to handle 
the KPZ 
equation and to construct, in the next section, the related polymer measure. 
In this perspective, an important point 
is to control the time-dependent rough paths $({\boldsymbol W}^T_{t})_{0 \leq t \leq T}$, 
uniformly in $T \in [0,T_{0}]$, which is one of the reason why we revisit the argument 
given in \cite[Section 7]{hai:13}. Instead of making use of general results on rough paths theory for Gaussian 
processes, we benefit from the fact that $Y^{T_0}$ solves the backward stochastic heat equation
to identify the cross-integral 
${\mathscr I}_{t}^T(x,x')$ in \eqref{eq:cI}
with a stochastic integral. 
Such a construction can be extended to non-Gaussian cases when $Y^{T_0}$ solves a stochastic PDE of a more general form (with possibly random coefficients).

\subsection{Proof of Theorem \ref{thm:rough:structure:1}}
Following the decomposition of $Y$ introduced in 
the statement of Theorem \ref{thm:rough:structure:1}, it makes sense 
to split
$Z^T_{t}(x)$ into  
$Z_t^T(x) = Z_t^{(1),T}(x)+Z_t^{(2),T}(x)$,
with
\begin{equation}
\label{eq:splitting:Z}
\begin{split}
&Z_{t}^{(1),T}(x) := \int_t^T\int_\R\partial^2_xp_{s-t}(x-y)\bigl( Y_t(y) - Y_{t}(x) \bigr)\ud y\ud s,
\\
&Z_{t}^{(2),T}(x) :=\int_t^T\int_\R\partial^2_xp_{s-t}(x-y)
	\bigl( Y_{s}(y) - Y_{s}(x) - (Y_{t}(y) - Y_{t}(x) \bigr) \ud y\ud s.
\end{split}
\end{equation}
Accordingly, we can split, at least formally, the iterated integral 
${\mathscr I}^T_{t}(x,x')$ in \eqref{eq:cI} into
${\mathscr I}^T_{t}(x,x') = {\mathscr I}^{(1),T}_{t}(x,x')
+ {\mathscr I}^{(2),T}_{t}(x,x')$, with 
\begin{equation}
\label{eq:splitting:cI}
{\mathscr I}^{(i),T}_{t}(x,x') := \int_{x}^{x'} \bigl( 
Z_{t}^{(i),T}(y) - Z_{t}^{(i),T}(x) \bigr) \ud Y_{t}(y), \quad i =1,2.
\end{equation}

The analysis of ${\mathscr I}^{(1),T}_{t}$ relies on 
\begin{lemma}
\label{lem:rough time indep}
Given $\alpha,\chi,\kappa >0$, 
there is a constant $C$, such that, 
for any $Y \in \C([0,T_{0}] \times \R,\R)$
with 
$\kappa_{\alpha,\chi}(Y) \leq \kappa$,
the map $x \mapsto 
Y_{t}(x)$ being differentiable for any $t \in [0,T_{0}]$,
it holds that
\begin{equation}
\label{eq:rough time indep:1}
\forall T \in [0,T_{0}], \ \forall t \in [0,T], \ \forall a\geq 1, \ \forall x, x'\in[-a,a],\ 
\bigl\vert {\mathscr I}_t^{(1),T}(x,x')  \bigr\vert \leq C a^{2\chi}
\vert x'-x\vert^{2\alpha}.
\end{equation}
Moreover,
\begin{equation}
\label{eq:rough time indep:2}
\begin{split}
{\mathscr I}_{t}^{(1),T}(x,x')
&= \bigl( Z_{t}^{(1),T}(x') - Z_{t}^{(1),T}(x) \bigr)
\bigl( Y_{t}(x') - Y_{t}(x) \bigr) + \bigl( Y_{t}(x') - Y_{t}(x) \bigr)^2
\\
&\hspace{15pt} - 2 \int_{x}^{x'} \int_{\R}
\partial_{x} p_{T-t}(y-z) Y_{t}(z) \bigl( Y_{t}(y) - Y_{t}(x) \bigr) \ud z \ud y.
\end{split}
\end{equation}
\end{lemma}

In the framework of Lemma \ref{lem:rough path}, 
\eqref{eq:rough time indep:2} remains true when $Y$ is not differentiable in $x$, by passing to the limit 
along a smooth approximation. 
When
$Y$ is time-homogeneous, ${\mathscr I}^{T}_{t}(x,x')$ 
and ${\mathscr I}^{(1),T}_{t}(x,x')$ coincide, and we have an explicit 
formula for the cross integral in \eqref{eq:cI}. 

\begin{proof}
\label{subsubse:time:homo}
Taking benefit of the heat equation satisfied by $p_{s-t}$, we have
\begin{equation*}
\begin{split}
Z_{t}^{(1),T}(x) &= 2 \int_{\R} \int_{t}^T \partial_{s} p_{s-t}(x-y) \bigl( Y_{t}(y) - Y_{t}(x) \bigr) \ud y \ud s
= 2 \int_{\R} p_{T-t}(x-y) Y_{t}(y) \ud y - 2Y_{t}(x).
\end{split}
\end{equation*}
Recalling that, under the assumption of Lemma \ref{lem:rough time indep}, 
$Y$ is smooth in space, we get from 
\eqref{eq:20:1:14:1}:
\begin{equation*}
\begin{split}
{\mathscr I}_{t}^{(1),T}(x,x') &= \bigl( Z_{t}^{(1),T}(x') - Z_{t}^{(1),T}(x) \bigr) \bigl( Y_{t}(x') - Y_{t}(x) \bigr)
- \int_{x}^{x'} \bigl( Y_{t}(y) - Y_{t}(x) \bigr) \partial_{x} Z_{t}^{(1),T}(y) \ud y.
\end{split}
\end{equation*}
Plugging the formula for $Z_{t}^{(1),T}$ into the above relationship, we get
\eqref{eq:rough time indep:2}.

By Lemma \ref{lemmaZ}, 
$\kappa_{\alpha,\chi}(Z^{(1),T})
\leq C \kappa$. 
It is then clear that, for $x,x' \in [-a,a]$, the two first terms in the right hand side of \eqref{eq:rough time indep:2}
satisfy \eqref{eq:rough time indep:1}. In order to prove that the third one satisfies 
it as well, we notice that it may be rewritten under the form
(up to the factor $-2$)
\begin{equation*}
{\mathscr J}_{T-t}(x,x') :=
\int_{x}^{x'} \int_{\R}
\partial_{x} p_{T-t}(y-z) \bigl( Y_{t}(z) - Y_{t}(x) \bigr) \bigl( Y_{t}(y) - Y_{t}(x) \bigr) \ud 
z \ud y.
\end{equation*}
Splitting the increment
$Y_{t}(z) - Y_{t}(x)$ into $Y_{t}(z) - Y_{t}(y)$ plus $Y_{t}(y) - Y_{t}(x)$,  
we deduce from the bound $\vert \partial_{x} p_{T-t} \vert \leq 
c(T-t)^{-1/2}p_{c(T-t)}$
 that
$\vert {\mathscr J}_{T-t}(x,x') \vert \leq C a^{2\chi}
[ 
 (T-t)^{-(1-\alpha)/2}
\vert x'-x \vert^{1+\alpha}
+ (T-t)^{-1/2} \vert x'-x \vert^{1+2\alpha}
]$, so that, for $T-t \geq \vert x'-x\vert^2$, 
$\vert {\mathscr J}_{T-t}(x,x') \vert \leq C a^{2\chi} \vert x'-x \vert^{2\alpha}$.

In order to handle the case $T-t \leq \vert x'-x \vert^2$, we first notice, by antisymmetry, that, for $x<x'$,
${\mathscr J}_{T-t}(x,x') ={\mathscr J}_{T-t}^{(-\infty,x]}(x,x') +{\mathscr J}_{T-t}^{[x',+\infty)}(x,x')$, 
that is ${\mathscr J}_{T-t}^{[x,x']}(x,x')=0$,
where
\begin{equation*}
\begin{split}
{\mathscr J}_{T-t}^{\mathbb I}(x,x') 
&:= \int_{x}^{x'} \int_{{\mathbb I}} \partial_{x} p_{T-t}(y-z)\bigl[ Y_{t}(z) - Y_{t}(x) \bigr] \bigl[ Y_{t}(y) - Y_{t}(x) \bigr] \ud z 
\ud y.
\end{split}
\end{equation*}
We start with ${\mathscr J}_{T-t}^{(-\infty,x]}(x,x')$ (the other one may be handled in the same way). We have
\begin{equation*}
\begin{split}
\bigl\vert {\mathscr J}_{T-t}^{(-\infty,x]}(x,x') 
\bigr\vert
&\leq C a^{\chi} \int_{x}^{x'} \int_{-\infty}^{x} \vert \partial_x p_{T-t}(y-z) \vert (a^\chi+|z|^\chi)\vert x- z \vert^\alpha 
\vert y-x \vert^{\alpha} \ud z \ud y.
\end{split}
\end{equation*}
Bounding $\vert y-x\vert$ by $\vert x'-x \vert$, 
the result follows from the following bound applied with $\gamma=0$ or $\chi$ and $h=T-t$,
\begin{align*}
\int_{- \infty}^{x} \biggl(\int_x^{\infty}  |z|^\gamma\vert x-z \vert^\alpha \bigl\vert \partial_{x} p_{h}(y-z) \bigr\vert \ud y \biggr) \ud z
&=- \int_{- \infty}^{x} |z|^\gamma\vert x-z \vert^\alpha \biggl(\int_x^{\infty}  \partial_{x} p_{h}(y-z) \ud y \biggr) \ud z
\\
&= \int_{- \infty}^{x} |z|^\gamma\vert x-z \vert^\alpha p_{h}(x-z) \ud z
\leq Ca^\gamma h^{\alpha/2}.\hspace{50pt}
\end{align*} 
\end{proof}

In order to handle ${\mathscr I}^{(2),T}_{t}$, we will make use of a famous result by Young \cite{you:36}:
\begin{lemma}
\label{lem:young}
Given two exponents  $\alpha,\alpha' >0$ 
with $\alpha+\alpha'>1$, there exists a universal constant $c>0$ such that, for any $\alpha'$-H\"older function $f$ and any $\alpha$-H\"older function $g$ on the interval $[x,x']$, the Stieltjes integral $\int_{x}^{x'} f(z) \ud g(z)$ is well defined and it holds
\begin{equation}
\label{eq:13:1:14:2}
\biggl\vert \int_{x}^{x'} f(z) \ud g(z)
- f(x) \bigl( g(x') - g(x) \bigr) \biggr\vert \leq c \|f\|_{\alpha'}\|g\|_{\alpha}  \vert x'- x \vert^{\alpha+\alpha'}
\end{equation}
where $\|f\|_{\alpha'}$ (resp. $\|g\|_{\alpha}$) is the H\"older semi-norm of $f$ (resp. $g$).
\end{lemma}

Young's result gives directly the existence of ${\mathscr I}^{(2),T}$:

\begin{lemma}
\label{lem:rough time indep:bis}
Consider $Y \in {\mathcal C}([0,T_{0}] \times \R,\R)$
satisfying both $\kappa_{\alpha,\chi}(Y) \leq \kappa$ and 
 \eqref{eq:21:09:2}. Then, for any $0\leq t\leq T\leq T_0$, the map $\R \ni x \mapsto Z_{t}^{(2),T}$ is 
locally
$2\nu+\mu$-H\"older in space and there exists a constant $C$, independent of $t$ and $T$, such that $\kappa_{2\nu+\mu,\chi}(Z^{(2),T})\leq C\kappa$. 
As a consequence of Young's theory, the integral ${\mathscr I}^{(2),T}$ is well defined and, for $a \geq 1$, $x,x'\in[-a,a]$,
$$
\bigl|{\mathscr I}^{(2),T}_t(x,x')\bigr|\leq Ca^{2\chi} \kappa^2 \left|x'-x\right|^{\alpha+2\nu+\mu }.
$$
\end{lemma}
\begin{proof} Let $a\geq1$ and $x\leq x'\in[-a,a]$. As in the proof of Lemma \ref{lemmaZ}, we have to bound $|Z_{t}^{(2),T}(x')-Z_{t}^{(2),T}(x)|$. We 
split the analysis into two cases: $|x'-x|^2\leq T-t $ and $|x'-x|^2>T-t$, 
only the case 
$|x'-x|^2\leq T-t$ being challenging. To handle it, we go back to 
\eqref{eq:3:2:14:1}, letting $v \equiv 1$ therein
and replacing $Y_{s}(z)$ by $Y_{s}(z) - Y_{t}(z)$. 
By \eqref{eq:21:09:2}, we can repeat the computations of 
Lemma \ref{lemmaZ}, replacing $s^{-\alpha/2}$ by $s^{-\nu-\mu/2}$. 
We deduce that $|Z_{t}^{(2),T}(x')-Z_{t}^{(2),T}(x)|
\leq C \kappa a^\chi|x'-x|^{2\nu+\mu}$. Since the sum of the H\"older exponents of $Z_{t}^{(2),T}$ and $Y_t$ is larger than 1,  the existence of and the bound for ${\mathscr I}^{(2),T}$ are direct consequences of Lemma \ref{lem:young}.
\end{proof}

Given Lemmas \ref{lem:rough time indep}
and \ref{lem:rough time indep:bis}, we now turn to 

\begin{proof}[Proof of Theorem \ref{thm:rough:structure:1}.]
Consider a smooth approximation $(Y^n)_{n \geq 1}$
of $Y$ constructed by spatial convolution, as in \eqref{eq:smooth:kernel}. 
Following
\eqref{eq:splitting:Z} and \eqref{eq:splitting:cI}, we may split 
${\mathscr I}^{n,T}$ accordingly, into 
${\mathscr I}^{n,T} = {\mathscr I}^{n,(1),T} + {\mathscr I}^{n,(2),T}$.
We then notice that each $Y^{n}$ satisfies 
$\kappa_{\alpha,\chi}(Y^n) \leq c\kappa$ and 
satisfies \eqref{eq:21:09:2} with $\kappa$ replaced by
$c\kappa$, for $c$ independent of $n$.  
If $\alpha > 1/2$, we can always choose $\nu=0$ and $\mu=\alpha$
in \eqref{eq:21:09:2}, in which case, by Lemmas 
\ref{lem:rough time indep} 
and \ref{lem:rough time indep:bis}, 
the first line in \eqref{eq:4:11:0} is satisfied with 
$\chi' = \chi$. If $\alpha \leq 1/2$, we must have $2 \nu+\mu > 1-\alpha \geq \alpha$
so that $\alpha + 2\nu+\mu \geq 2\alpha$.  
By Lemmas 
\ref{lem:rough time indep} 
and \ref{lem:rough time indep:bis}, 
the first line in \eqref{eq:4:11:0} is satisfied with 
$2\chi' = 2\chi + (2\nu+\mu -\alpha)$. Since the value of $\nu$ can be arbitrarily 
decreased provided that $2 \nu + \mu > 1-\alpha$ still holds true, we deduce that the 
first line in \eqref{eq:4:11:0} is satisfied for any $\chi' > \chi + (1-\alpha/2)_{+}$. 

It thus remains to check the second line in \eqref{eq:4:11:0}. 
We first notice that we can pass to the limit in 
the formula  
\eqref{eq:rough time indep:2}
for ${\mathscr I}^{n,(1),T}$, replacing therein $Z^{(1),T}$ by $Z^{n,(1),T}$
and $Y$ by $Y^n$. Obviously, the limit is ${\mathscr I}^{(1),T}$ (whatever the choice of the 
smooth approximation is). 
Following the proof of Lemma 
\ref{lem:rough path}, the convergence is uniform on any 
$[0,T] \times [-a,a]$, $a \geq 1$, which means that 
\begin{equation}
\label{eq:unit:cv:proof}
\lim_{n \to \infty} \sup_{0 \leq t \leq T} \sup_{x,x' \in [-a,a]}
 \bigl\vert 
 {\mathscr I}^{(1),T}_{t}(x,x') - {\mathscr I}_{t}^{n,(1),T}(x,x') 
 \bigr\vert = 0.
\end{equation}
Passing to the limit in \eqref{eq:rough time indep:1},
${\mathscr I}^{(1),T}$ satisfies \eqref{eq:rough time indep:1}.
Combining with \eqref{eq:unit:cv:proof}, we deduce, as in 
 \eqref{eq:holder:trick},
 that the second line in 
 \eqref{eq:4:11:0} holds 
with  ${\mathscr I}^{T} - {\mathscr I}^{n,T}$
replaced by
${\mathscr I}^{(1),T} - {\mathscr I}^{n,(1),T}$. 
 
 In order to complete the proof, we must prove the second line in \eqref{eq:4:11:0}, but
with  ${\mathscr I}^{T} - {\mathscr I}^{n,T}$
replaced by
${\mathscr I}^{(2),T} - {\mathscr I}^{n,(2),T}$. We have the decomposition 
\begin{align*}
{\mathscr I}^{(2),T} - {\mathscr I}^{n,(2),T}=&\int_x^{x'}\Bigl(Z_{t}^{(2),T}(y)-Z_{t}^{(2),T}(x)-	\bigl(Z_{t}^{n,(2),T}(y)-Z_{t}^{n,(2),T}(x) \bigr)\Bigr)\ud Y_t(y)\\
&+\int_x^{x'}\bigl(Z_{t}^{n,(2),T}(y)-Z_{t}^{n,(2),T}(x)\bigr)\ud \, (Y_t-Y_t^n)(y).
\end{align*}
We start with the second term in the right-hand side. 
For any $\alpha'<\alpha$, we know that, locally, the 
$\alpha'$-H\"older norm of $Y - Y^n$ in space tends to $0$. 
Repeating the proof 
of Lemma \ref{lem:rough time indep:bis},
we deduce that, locally,
the $2\alpha'$-H\"older semi-norm of 
the last term tends to $0$. 
In order to prove
the same result 
for the first term, it suffices to notice that, 
for any pair $(\nu',\mu')$ with $\nu' \leq \nu$ and $\mu' \leq \mu$, 
one of the two inequalities being strict, 
the difference $Y^n - Y$ satisfies 
\eqref{eq:21:09:2} with a constant $\kappa$ that may depend on $a$ 
but that 
tends to $0$ as $n$ tends to $\infty$. Therefore,
the $(2\nu'+\mu')$-H\"older norm of the integrand in the first term tends to $0$

If $(\tilde{Y}^n)_{n \geq 1}$ is another approximation, also constructed by convolution, 
we can prove in the same way that the difference ${\mathscr I}^{n,(2),T}
-\tilde{\mathscr I}^{n,(2),T}$ tends to $0$, where 
$\tilde{\mathscr I}^{n,(2),T}$ is associated with 
$\tilde{Y}^n$. 
Therefore, 
$({\mathscr I}^{n,(2),T})_{n \geq 1}$
and $(\tilde{\mathscr I}^{n,(2),T})_{n \geq 1}$
have the same limit. 
Things are the same when $\alpha >1/2$ (with $\nu =0$ and
$\mu = \alpha$)  and 
 the construction of
$(\tilde{Y}^n)_{n \geq 1}$ is arbitrary, since, in that case, 
$(\tilde{Y}^n)_{n \geq 1}$ necessarily satisfies \eqref{eq:21:09:2}, uniformly in $n \geq 1$. 
\end{proof}

\subsection{Proof of Theorem \ref{thm:main:thm:polymer}}
The proof is divided in several steps. The first one is to prove a generalization of the well-known Kolmogorov's H\"older continuity criterion.

\begin{theorem}
\label{thm:kolmo}
Let
${\mathcal Q}$ be a countable set and 
 $(R_{L} : [-1,1]^2 \times \Xi \ni (x,y,\xi) \mapsto R_{L}(x,y)(\xi) \in \R)_{L \in {\mathcal Q}}$ be a family of random fields on the space $(\Xi,{\mathcal G},{\mathbf P})$, satisfying, for some $p \geq 1$, some $C,\beta,\gamma,\gamma_{1},\gamma_{2} >0$,
 some random variable $\zeta$, all $a \geq 1$
 and all $x,y,z \in [-a,a]$, $x<y<z$, 
\begin{equation}
\label{eq:15:06:3}
\begin{split}
&\E \bigl[ \sup_{L \in {\mathcal Q}} \vert R_{L}(x,y) \vert^p \bigr] \leq C 
a^{p\gamma}
\vert x-y \vert^{1+\beta},  
\\
&\forall L \in {\mathcal Q},  \ \vert R_{L}(x,y) + R_{L}(y,z) - R_L(x,z) \vert \leq \zeta a^\gamma   \vert x - y \vert^{\gamma_{1}} \vert y - z \vert^{\gamma_{2}}. 
\end{split}
\end{equation}
Then, 
for any 
$L \in {\mathcal Q}$ and 
$x,y \in \R$, we can redefine $R_{L}(x,y)$
on a ${\mathbf P}$ null event, and, for any $\chi >1/p$ and 
$0 < \varsigma < \min (\gamma_{1}+\gamma_{2},\beta/p)$, 
we can find a constant 
$c:=c(\varsigma,\chi,\gamma_{1},\gamma_{2},\beta,p)$
and a 
non-negative random variable $\zeta'$, with $\E[\vert \zeta'\vert^p] < c C$, such that,   
for all $a \geq 1$,
\begin{equation}
\label{eq:nouvelle:preuve:KPZ:4}
\forall L \in {\mathcal Q}, \ x,y \in [-a,a], \quad \vert R_{L}(x,y) \vert \leq c 
\bigl( 
\zeta'
a^{\chi + (1+\beta)/p} + \zeta  a^{\gamma_{1}+\gamma_{2}}\bigr) 
a^{\gamma - \varsigma }
\vert x- y \vert^{\varsigma}
\end{equation}
The result remains true when ${\mathcal Q}$ is a separable metric space
and, for any $x,y \in \R$, the mapping ${\mathcal Q} \ni L \mapsto 
R_{L}(x,y)$ is almost-surely continuous. 
\end{theorem}

\begin{proof}
In the case  
$a =1$, \eqref{eq:nouvelle:preuve:KPZ:4} can be proved by adapting 
the proof of the standard version of 
Kolmogorov's criterion. In order to get
the result for any $a \geq 1$, we can 
fix $a \in {\mathbb N} \setminus \{0\}$ and then apply the result on 
$[-1,1]^2$ to the family $(R_{L} : [-1,1]^2 \times \Xi \ni (x,y,\xi) 
\mapsto R_{L}(a x,a y)(\xi) )_{L \in {\mathcal Q}}$. 
It satisfies \eqref{eq:15:06:3}
with $C a^{p\gamma}$
replaced by 
$C a^{1+\beta+p\gamma}$
in the first line 
and $\zeta a^{\gamma}$
replaced by 
$\zeta a^{\gamma_{1}+\gamma_{2}+\gamma}$
in the second line. Therefore, 
for any $\varsigma \in 
(0,\min (\gamma_{1}+\gamma_{2},\beta/p))$,
we can find a constant $C'$, independent of 
$a$, and a variable $\zeta_{a}'$ (which may depend
on $a$), such that (up to a redefinition of each $R_{L}(a x,a y)$
on a ${\mathbf P}$ null event)
\begin{equation*}
\forall L \in {\mathcal Q}, 
\ \forall x,y \in [-1,1],
\quad \vert R_{L}\bigl( a x, a y \bigr) \vert 
\leq C' a^{\gamma}
\bigl(a^{(1+\beta)/p} \zeta_{a}' + a^{\gamma_{1}+\gamma_{2}}\zeta \bigr) 
\vert x - y \vert^\varsigma,
\end{equation*}
with $\E[\vert \zeta_{a}' \vert^p] \leq C'$. 
Choose $\chi >1/p$ and let 
$\Gamma :=  \sup_{a \in {\mathbb N}\setminus\{0\}}
[
a^{-\chi} \zeta_a']$.
Then, for another constant $C''>0$, $\E[\vert \Gamma \vert^p]
\leq C' \sum_{a \geq 1} a^{-p\chi} \leq C''$. 
We have
\begin{equation*}
\forall L \in {\mathcal Q}, \ \forall x,y \in [-1,1], 
\quad\vert R_{L}(a x,a y) \vert \leq 
C' a^{\gamma - \varsigma} \bigl(
a^{\chi + (1+ \beta)/p} \Gamma
+ a^{\gamma_{1}+\gamma_{2}} \zeta \bigr) 
\vert a x- a y \vert^{\varsigma}.
\end{equation*}

When ${\mathcal Q}$
is a separable metric space, 
we consider a countable dense subset $\hat{\mathcal Q}$. 
For any realization in an event of probability one, 
for any $x,y \in {\mathbb Q}$, the map 
$\mathcal Q \ni L \mapsto R_{L}(x,y)$
is continuous, and, by the first part, 
the maps  $(\R^2 \ni (x,y) \mapsto R_{L}(x,y))_{L \in \hat{\mathcal Q}}$
satisfy \eqref{eq:nouvelle:preuve:KPZ:4} and are thus uniformly continuous on compact sets. 
With probability one, 
we can extend 
$\hat{\mathcal Q} \times {\mathbb Q}^2 \ni (L,x,y) 
\mapsto R_{L}(x,y)$ into a continuous mapping 
on ${\mathcal Q} \times \R^2$, which satisfies  \eqref{eq:nouvelle:preuve:KPZ:4}.
\end{proof}

\subsubsection{Regularity of $Y^{T_{0}}$}
We start with:
\begin{lemma}
\label{lem:preuve:KPZ:1}
There exists $\Xi^\star \in {\mathcal G}$, with 
${\mathbf P}(\Xi^\star)=1$, 
such that, on $\Xi^\star$, for all $\alpha <1/2$, $\chi >0$,
the map 
$[0,T_{0}] \times \R \ni (t,x) \mapsto 
Y_{t}^{T_{0}}(x)$ is continuous and satisfy $\kappa_{\alpha,\chi}(Y^{T_{0}})<\infty$. Moreover,
$\E[(\kappa_{\alpha,\chi}(Y^{T_{0}}))^p]
< \infty$ 
for all $\alpha <1/2$, $\chi >0$ and $p\geq1$.
\end{lemma}
Continuity of $Y^{T_{0}}$ is 
a well-known fact, which follows from 
Kolmogorov's criterion. 
Letting ${\mathcal D}_{T_{0}} = \{ (t,s) \in [0,T_{0}]^2 : t < s\}$,
the almost sure finiteness of 
$\kappa_{\alpha,\chi}(Y)$
 is a consequence of 
the following result:
\begin{lemma}
\label{lem:reg:noyau}
Let  ${\mathcal K} : {\mathcal D}_{T_{0}} \times \R \times \Xi 
\ni ((t,s),y,\xi) \mapsto {\mathcal K}_{t}(s,y)(\xi) \in \R$
be a random function, continuous in $(t,s,y)$ for any $\xi$
and differentiable in $t$ for any $(s,y,\xi)$, 
such that 
${\mathcal K}_{t}(s,y)$ is
measurable  
with respect to the $sigma$-field ${\mathcal G}_{s}^{T_{0}}:
= \sigma(\zeta(u,y) - \zeta(u,x) - \zeta(s,y) + \zeta(s,x), \ s \leq u \leq T_{0}, \ x,y \in \R)$. 
Assume that there exist 
a constant $\epsilon >0$
and
a non negative random variable $\kappa$, 
with $\E[\kappa^q]<\infty$ for any $q \geq 1$, such that, $\P$-almost surely, for any $t \leq s \leq T_{0}$, 
\begin{align}\label{borne:reg:noyau}
	\int_\R\left|{\mathcal K}_{t}(s,y)\right|^2\ud y\leq \kappa \left|s-t\right|^{-1+\epsilon}\quad\text{ and }\quad \int_\R\left|\partial_t{\mathcal K}_{t}(s,y)\right|^2\ud y\leq \kappa \left|s-t\right|^{-3+\epsilon}.
\end{align}
Then, 
letting 
${\mathcal I}_{t}^T$
be the backward stochastic It\^o integral
$\int_{t}^T \int_{\R} {\mathcal K}_{t}(r,u) \ud \zeta(r,u)$, 
the quantity 
$\sup_{0 \leq t \leq T \leq T_{0}}
\vert {\mathcal I}_{t}^T \vert$ is a random variable and, for any 
$p \geq 1$, we can find a	 constant $c_{p}$, only depending upon 
$p$ and $T_{0}$, such that  
${\mathbb E} [ 
 \sup_{0 \leq t \leq T \leq T_{0}}
	\vert {\mathcal I}_t^T
\vert^p ]^{1/p} \leq
 c_{p}\E[\kappa^{p/2}]^{1/p}$.
 \end{lemma}
 
 \begin{proof}[Proof of Lemma \ref{lem:reg:noyau}]
For any $t \leq t' \leq t'+\delta \leq T \leq T_{0}$, 
${\mathcal I}_t^T - {\mathcal I}_{t'}^{T}$ is equal to
\begin{equation*}
\begin{split}
 \int_{t'+\delta}^T \int_{\R}
\bigl[
 {\mathcal K}_{t}(s,y) - {\mathcal K}_{t'}(s,y) \bigr] \ud \zeta(s,y)
+ \int_{t}^{t'+\delta} \int_{\R}
{\mathcal K}_{t}(s,y) \ud \zeta(s,y)
- \int_{t'}^{t'+\delta}
\int_{\R} {\mathcal K}_{t'}(s,y) \ud \zeta(s,y).
\end{split}
\end{equation*}
By square integrability of ${\mathcal K}$ in $(s,y)$, 
${\mathcal I}_{t}^{T}$ is continuous in $T$, 
and we can take the 
supremum over $T \in [t'+\delta,T_{0}]$. 
Writing ${\mathcal K}_{t'}-{\mathcal K}_{t}= \int_{t}^{t'} \partial_{t} {\mathcal K}_{r} \ud r$, 
we get
that, for any $p \geq 1$,
\begin{equation*}
\begin{split}
&\E \bigl[ \sup_{t'+\delta \leq T \leq T_{0}}\bigl\vert 
{\mathcal I}_t^{T} - {\mathcal I}_{t'}^{T}
\bigr\vert^p \bigr]^{1/p}
\leq 
c_{p}
\vert t'-t \vert^{1/2} \E \biggl[
 \biggl( \int_{t'+\delta}^{T_{0}} 
 \int_{t}^{t'} \int_{\R}
\vert 
\partial_{t} {\mathcal K}_{r}(s,y) \vert^2 \ud y \ud r \ud s \biggr)^{p/2}
\biggr]^{1/p}
\\
&\hspace{15pt}+  \E \biggl[
 \biggl( \int_{t}^{t'+\delta} \int_{\R}
\vert 
{\mathcal K}_{t}(s,y) \vert^2 \ud y \ud s \biggr)^{p/2}
\biggr]^{1/p}
 +
\E \biggl[
 \biggl( \int_{t'}^{t'+\delta} \int_{\R}
\vert 
{\mathcal K}_{t'}(s,y) \vert^2 \ud y \ud s \biggr)^{p/2}
\biggr]^{1/p},
\end{split}
\end{equation*}
the first term in the right-hand side being obtained by the Burkh\"older-Davies-Gundy inequality
and the constant $c_{p}$ only depending upon $p$. Using the bounds  \eqref{borne:reg:noyau}, we get
\begin{align*}
	&\E \bigl[ \sup_{t'+\delta \leq T \leq T_{0}}\bigl\vert 
{\mathcal I}_t^T - {\mathcal I}_{t'}^T
\bigr\vert^p \bigr]^{1/p}
\leq \E[\kappa^{p/2}]^{1/p}\left(c_p
\vert t'-t \vert^{1/2} \delta^{-1/2+\varepsilon/2}+\vert t'-t +\delta\vert^{\epsilon/2}+\delta^{\epsilon/2}\right).
\end{align*}

Choosing $\delta =  t'-t$ and 
modifying $c_{p}$ if necessary, we can bound the right-hand side by  
$ c_{p}  \E[\kappa^{p/2}]^{1/p}\vert t'-t \vert^{\epsilon/2}.$ We easily get a similar bound when the supremum is taken 
over $T \in [t,(t'+\delta) \wedge T_{0}]$, with the convention that 
${\mathcal I}_{t'}^T = 0$ when $t'>T$. 
We deduce that
\begin{equation*}
\E \bigl[ \bigl\vert \sup_{t \leq T \leq T_{0}}
\vert {\mathcal I}_t^T \vert  - 
 \sup_{t' \leq T \leq T_{0}}
\vert {\mathcal I}_{t'}^T \vert
\bigr\vert^p \bigr]^{1/p} \leq
 c_{p}  \E[\kappa^{p/2}]^{1/p}\vert t'-t \vert^{\epsilon/2}.
 \end{equation*} 
The result follows from Kolmogorov's criterion.  
\end{proof}

\begin{proof}[Proof of Lemma \ref{lem:preuve:KPZ:1}]
 Lemma \ref{lem:reg:noyau} applies to the proof of Lemma 
 \ref{lem:preuve:KPZ:1}
 with ${\mathcal K}_{t}(s,y) = p_{s-t}(x-y)-p_{s-t}(x'-y)$
 and thus ${\mathcal I}_{t}^T = Y_{t}^T(x)- Y_{t}^T(x')$,
 for $x,x' \in \R$.  
 We have two bounds for
$\int_{\R} \vert {\mathcal K}_{t}(s,y) \vert^2 \ud y$. 
The first one is  
$\int_{\R} \vert {\mathcal K}_{t}(s,y) \vert^2 \ud y \leq C (s-t)^{-1/2}$
and the second one 
is 
$\int_{\R} \vert {\mathcal K}_{t}(s,y) \vert^2 \ud y \leq C (s-t)^{-3/2} \vert x'-x\vert^2$. 
By interpolation, we have,
for any $\alpha \in (0,1)$, 
$\int_{\R} \vert {\mathcal K}_{t}(s,y) \vert^2 \ud y \leq C (s-t)^{-(1/2+\alpha)} \vert x'-x\vert^{2\alpha}$. A similar argument applies to 
$\int_{\R} \vert \partial_{t} {\mathcal K}_{t}(s,y) \vert^2 \ud y$. 
We can bound it by $C(s-t)^{-5/2}$
and by $C(s-t)^{-7/2} \vert x'-x\vert^2$, and thus by
$C (s-t)^{-(5/2+\alpha)} \vert x'-x\vert^{2\alpha}$. 
The bound in the statement of Lemma 
\ref{lem:reg:noyau} holds true with  
$\kappa=C \vert x'-x\vert^{2\alpha}$ and $\epsilon = (1/2-\alpha)$. 
Therefore, for $p \geq 1$ and 
$\alpha \in (0,1/2)$, we can find a constant $C_{p}$ such that 
\begin{equation*}
{\mathbb E} \bigl[ \sup_{0 \leq t \leq T \leq T_{0}} \vert Y_{t}^T(x) - Y_{t}^T(x') 
\vert^p \bigr]^{1/p} \leq C_{p} \vert x'-x\vert^{\alpha}.  
\end{equation*}
The conclusion follows now from Theorem \ref{thm:kolmo}.
We apply it with $R_{t,T}(x,y) = Y_{t}^T(y) - Y_{t}^T(x)$,
$p$ as large as desired, $\beta = p \alpha-1$, $\zeta =0$, $\gamma=0$, $\gamma_{1}=\gamma_{2}=1$, $L=(t,T)$ and $\mathcal{Q}={\mathcal D}_{T_{0}}$.   
We get that, for any 
$\chi>0$, $\alpha \in (0,1/2)$ and $p\geq1$, $\E[(\sup_{T\leq T_0}\kappa_{\alpha,\chi}(Y^{T}))^p]< \infty$ and there is an event $\Xi^{\chi,\alpha,\star}$,
of probability $1$,
on which $\sup_{T\leq T_0}\kappa_{\alpha,\chi}(Y^{T})<\infty$. 
Letting $\Xi^\star= \cap_{\chi,\alpha \in {\mathbb Q},\chi>0,\alpha \in (0,1/2)}
\Xi^{\chi,\alpha,\star}$, this completes the proof. (Note that the result is 
actually stronger than the claim in the statement. 
The reason why we included a supremum over $T$ in the statement of Lemma 
\ref{lem:reg:noyau} will become clear in the last part of the proof of Theorem
\ref{thm:main:thm:polymer}.)
\end{proof}

\subsubsection{Reducing the proof to the case $Y^{(b)} \equiv 0$}
The next step is to show:

\begin{lemma}
\label{lem:KPZ:30}
In order to prove Theorem \ref{thm:main:thm:polymer}, 
we can assume $Y^{(b)} \equiv 0$.
\end{lemma}

\begin{proof}
\textit{First step.}
If $Y^{(b)} \not \equiv 0$, we consider $\Xi^\star$ 
and then $\chi \in (0,\chi_{b}]$ and $\alpha \in (1-\alpha_{b},1/2)$
as in Lemma \ref{lem:preuve:KPZ:1}. 
For a realization in $\Xi^\star$ and for a smooth kernel $\rho$, $\rho$ and its derivatives
being
at most of polynomial decay, we let, for every integer $n \geq 1$,
$Y^{n,T_0}$ be the $n$th 
approximation of $Y^{T_0}$
constructed by convolution, see 
\eqref{eq:smooth:kernel}. 
Clearly, $Y^{n,T_0}_{t}(x)
 = \int_{t}^{T_{0}}
 \int_{\R} p_{s-t}(x-y) \ud \zeta^n(s,y)$,
 where $\zeta^n(s,y) = n \int_{0}^s \int_{\R}
 \rho(n(y-u)) \ud \zeta(r,u)$, proving 
 that $Y^{n,T_0}=Y^{n \rho(n \cdot),T_0}$. 
By Lemma \ref{lem:preuve:KPZ:1}, 
\eqref{eq:holder:trick} holds true (with $(Y,Y^n)$ replaced by 
$(Y^{T_{0}},Y^{n,T_{0}})$). 
 
Given a realization in $\Xi^\star$ and 
the path $Y^{(b)}$, we consider an arbitrary
smooth approximation $(Y^{n,(b)})_{n \geq 1}$
of $Y^{(b)}$, so that
$(Y^n := Y^{n,T_0} + Y^{n,(b)})_{n \geq 1}$ is a smooth approximation of 
$Y$.  

\textit{Second step.} 
Letting
\begin{equation}\label{eq:Z:i}
Z^{i,T}_{t}(x)
:=  \int_{t}^T \int_{\R}
\partial^2_{x} p_{s-t}(x-y) 
\bigl(
Y^{i}_{s}(y) - Y_{s}^{i}(x)  \bigr) \ud y \ud s, \quad i \in \{T_0,(b)\},
\end{equation}
we may split, at least formally, ${\mathscr I}^{T}_{t}(x,x')$ into 
\begin{equation*}
{\mathscr I}^{T}_{t}(x,x') = 
{\mathscr I}^{(T_0,T_0),T}_{t}(x,x') + {\mathscr I}_{t}^{(T_0,(b)),T}(x,x') + {\mathscr I}_{t}^{((b),T_0),T}(x,x')
+ {\mathscr I}_{t}^{((b),(b)),T}(x,x'),
\end{equation*}
where
\begin{equation}
\label{eq:cross:integral:i,j}
{\mathscr I}^{(i,j),T}_{t}(x,x')
:= \int_{x}^{x'} \bigl( 
Z^{i,T}_{t}(y)
- Z^{i,T}_{t}(x) \bigr) \ud Y_{t}^{j}(y), \quad (i,j) \in \{T_0,(b)\}^2. 
\end{equation}
When $(i,j) \not = (T_0,T_0)$, the cross-integrals 
${\mathscr I}^{(i,j),T}_{t}(x,x')$ can be constructed
as Young integrals by means 
of Lemma \ref{lem:young}.
Indeed, when $i=(b)$, the path 
$Z^{(b),T}_{t}$ has the same regularity as $Y^{(b)}$, see
Lemma \ref{lemmaZ}, so that
the sum of the H\"older exponents 
of the two curves involved in the definition of the integral 
is always greater than $\alpha+\alpha_{b}>1$
when at least one of the two indices $i$ or $j$ 
is equal to $(b)$. 
Lemmas \ref{lem:young} 
and \ref{lem:rough time indep:bis}
directly say that 
$\vert {\mathscr I}^{(i,j),T}_{t}(x,x') \vert 
\leq C a^{2(\chi_{b}+\alpha_{b}-\alpha)} \vert x'-x\vert^{2\alpha}$
when $x,x' \in [-a,a]$ with $a \geq 1$,
the constant $C$ 
being random (as it depends upon the realization of
$\kappa_{\alpha,\chi}(Y)$). 
Denoting by $Z^{n,i,T}$ and ${\mathscr I}^{n,(i,j),T}$
the quantities associated with the smooth approximation 
$Y^{n}$, 
${\mathscr I}^{n,(i,j),T}_{t}(x,x')$ satisfies a similar bound, with the same $C$. 
By bilinearity of Young's integral in $(f,g)$, see Lemma \ref{lem:young}, it is clear that, for all $T \in [0,T_{0}]$ and 
$a \geq 1$, 
$\sup_{0\leq t \leq T}\|\mathscr{I}_t^{(i,j),T}-\mathscr{I}_t^{n,(i,j),T}\|_{2\alpha'}^{[-a,a]}
$ tends to $0$ as $n$ tends to $\infty$, when $(i,j) \not = (T_0,T_0)$ and $\alpha'<\alpha$. 
\end{proof}

\subsubsection{Proof of Theorem \ref{thm:main:thm:polymer} when $Y^{(b)} \equiv 0$}
\begin{lemma}
\label{lem:preuve:KPZ:2}
Theorem \ref{thm:main:thm:polymer} is true when $Y^{(b)} \equiv 0$. 
\end{lemma}

\begin{proof}
\textit{First step.}
 The point is to construct ${\mathscr I}^{T}$ which is equal to ${\mathscr I}^{(T_0,T_0),T}$ as $Y^{(b)} \equiv 0$.
With $\rho$ as in the statement, recall 
$Y^{\rho,T}_{t}(x) 
= \int_{t}^{T} \int_{\R} p_{s-t}(x-y) \ud \zeta^\rho(s,y)
= \int_{t}^T P_{s-t} \rho(x-u) \ud \zeta(s,u)$. 
With these notations, the smooth approximation $Y^{n,T_0}$ considered in the first step of the proof of Lemma \ref{lem:KPZ:30} is obtained by 
replacing $\rho$ by $n \rho(n \cdot)$  and $Y^{T_0}$  by 
replacing $\rho$ by the Dirac mass $\delta_{0}$ at $0$. 
With $Y^{\rho,T_{0}}$, we associate a 
cross-integral as in
\eqref{eq:cI}. We let $Z^{\rho,T}_{t}(x) := \int_{t}^T \partial^2_{x} P_{s-t}
 Y_{s}^{\rho,T_{0}}(x) \ud s$ and 
 \begin{equation}
\label{eq:Y:rho:a}
\begin{split}
{\mathscr I}_{t}^{\rho,T}(x,x') &= \int_{x}^{x'} \bigl( Z^{\rho,T}_{t}(y)- Z^{\rho,T}_{t}(x) \bigr) \ud Y_{t}^{\rho,T_0}(y)\\
&=\int_{x}^{x'} \biggl[ \int_t^T\bigl(\partial_x^2P_{s-t}Y^{\rho,T_0}_s(y) -\partial_x^2P_{s-t}Y^{\rho,T_0}_s(x)\bigr) \ud s \biggr] \ud Y_{t}^{\rho,T_0}(y).
\end{split}
\end{equation}
Using
the identity 
$Y_t^{\rho,T_{0}}=P_{s-t}Y_s^{\rho,T_{0}}+Y_t^{\rho,s}$,	
for $0\leq t\leq s \leq T_0$, this leads to ${\mathscr I}_{t}^{\rho,T}(x,x')={\mathscr I}_{t}^{\rho,(1),T}(x,x')+{\mathscr I}_{t}^{\rho,(2),T}(x,x')$, with 
\begin{equation}
\label{eq:mathcal:I}
\begin{split}
{\mathscr I}_{t}^{\rho,(1),T}(x,x') &:=  \int_t^T
\biggl[\int_{x}^{x'}\bigl(\partial_x^2P_{s-t}Y^{\rho,T_0}_s(y) -\partial_x^2P_{s-t}Y^{\rho,T_0}_s(x)\bigr)  \partial_xP_{s-t}Y_{s}^{\rho,T_0}(y)\ud y \biggr]\ud s\\
\quad {\mathscr I}_{t}^{\rho,(2),T}(x,x') &:=  \int_t^T
\biggl[ \int_{x}^{x'}\bigl(\partial_x^2P_{s-t}Y^{\rho,T_0}_s(y) -\partial_x^2P_{s-t}Y^{\rho,T_0}_s(x)\bigr)  \ud Y_{t}^{\rho,s}(y) \biggr] \ud s.
\end{split}
\end{equation}
With these notations, ${\mathscr I}^{n,T}_{t}(x,x')$, the cross integral corresponding to $Y^{n,T_0}$,
is obtained by replacing $\rho$ by $n \rho(n \cdot)$
and  ${\mathscr I}^{T}_{t}(x,x')$
by replacing (at least at a formal level) $\rho$ by $\delta_0$. 

\textit{Second step.} Direct integration yields
\begin{align*}
{\mathscr I}_{t}^{\rho,(1),T}(x,x') = & \int_t^T\frac{1}{2}\left(\left(\partial_xP_{s-t}Y^{\rho,T_0}_s(x')\right)^2 -\left(\partial_xP_{s-t}Y^{\rho,T_0}_s(x)\right)^2\right)\ud s\\
&-\int_t^T\partial_x^2P_{s-t}Y_{s}^{\rho,T_0}(x)\left(P_{s-t}Y_{s}^{\rho,T_0}(x')-P_{s-t}Y_{s}^{\rho,T_0}(x)\right)\ud s.
\end{align*}
Imitating the proof of Lemma \ref{lemmaZ}, it is now easy to see that, for $x,x' \in [-a,a]$, with $a \geq 1$, 
$\vert {\mathscr I}_{t}^{\rho,(1),T}(x,x')  \vert
\leq C \kappa_{\rho}^2 a^{2 \chi} \vert x - x'\vert^{2 \alpha}$, for a 
deterministic constant $C$, independent of $\rho$,
and with $\kappa_{\rho} := \kappa_{(\alpha+1/2)/2,\chi}(Y^{\rho,0})$
(the reason why we use $(\alpha+1/2)/2$ will be explained below).  
 The computations also apply
 when $\rho$ is replaced by $\delta_{0}$.
Since $Y^{\rho,0}$ is constructed by 
convolution of $Y^{T_0}$ with respect to 
$\rho$, we can bound 
$\kappa_{\rho}$ by $c_{\rho} \kappa$, 
where $c_{\rho}$ is a deterministic constant that 
may depend on the decay of $\rho$. 
When $\rho$ is replaced by 
$n \rho(n \cdot)$, 
the constants
$c_{n\rho(n\cdot)}$
can be uniformly bounded in $n$, so that, for any $n \geq 1$,  
$\vert {\mathscr I}_{t}^{n \rho(n\cdot),(1),T}(x,x') \vert \leq 
C_\rho \kappa^2 a^{2 \chi}
\vert x - x'\vert^{2 \alpha}$. 

The bilinearity of the cross-integral shows that  
$\sup_{0 \leq t \leq T \leq T_{0}}
\sup_{x,x' \in [-a,a]}
| {\mathscr I}_{t}^{n \rho(n\cdot),(1),T}(x,x') - 
{\mathscr I}_{t}^{\delta_0,(1),T}(x,x') |$ tends to $0$ as $n$ tends to $\infty$. 
By \eqref{eq:holder:trick}, the convergence holds in H\"older norm.

\textit{Third step.}
We now study 
${\mathscr I}_{t}^{\rho,(2),T}(x,x')$ for $x,x' \in [-a,a]$, $a \geq 1$.
It is equal to 
\begin{equation*}
\begin{split}
&\int_t^T
\biggl[ \int_{x}^{x'}\bigl(\partial_x^2P_{s-t}Y^{\rho,T_0}_s(y) -\partial_x^2P_{s-t}Y^{\rho,T_0}_s(x)\bigr)  \biggl( \int_{t}^s \int_{\R}\partial_{x} P_{r-t} \rho(y-z)  \ud \zeta(r,z) 
\biggr) \ud y
\biggr] \ud s
\\
&= \int_{t}^T \int_{\R} \biggl[ \int_{x}^{x'}
\biggl( \int_{r}^T
\bigl(\partial_x^2P_{s-t}Y^{\rho,T_0}_s(y) -\partial_x^2P_{s-t}Y^{\rho,T_0}_s(x)\bigr) 
\ud s \biggr)
\partial_{x} P_{r-t} \rho(y-z) \ud y\biggr]
\ud \zeta(r,z).
\end{split}
\end{equation*}
By integration by parts,
${\mathscr I}_{t}^{\rho,(2),T}(x,x')  =\int_{t}^T \int_\R
[{\mathcal K}_{t}^{1} - {\mathcal K}_{t}^{2}](r,u) \ud \zeta(r,u)$ with
\begin{equation}
\label{eq:K:1:2}
\begin{split}
&{\mathcal K}_{t}^{1}(r,u)
= P_{r-t}\rho(x'-u)
\int_{r}^T \bigl(\partial_x^2P_{s-t}Y^{\rho,T_0}_s(x') -\partial_x^2P_{s-t}Y^{\rho,T_0}_s(x)\bigr)\ud s
\\
&{\mathcal K}_{t}^{2}(r,u)=\int_{x}^{x'}
P_{r-t}\rho(y-u)
\biggl[ \int_{r}^T \partial_x^3P_{s-t}Y^{\rho,T_0}_s(y)\ud s \biggr] \ud y.
\end{split}
\end{equation}
We start with ${\mathcal K}^2$.
For $x,x' \in [-a,a]$, 
$
\left\vert 
\int_{r}^T\partial_{x}^3 P_{s-t}
Y_{s}^{\rho,T_0}(y) \ud  s 
\right\vert \leq C \kappa_{\rho} a^{\chi} 
\int_{r}^T (s-t)^{-3/2+ \alpha/2} \ud s
\leq C \kappa_{\rho}
a^{\chi}(r-t)^{-1/2+\alpha/2}$, so that  
\begin{equation}
\label{eq:nouvelle:preuve:KPZ:1}
\begin{split}
\int_{\R}
&\vert {\mathcal K}_{t}^2(r,u) \vert^2 \ud u 
\leq C \kappa_{\rho}^2 a^{2\chi}
(r-t)^{-1+\alpha}
\\
&\hspace{5pt}
\times \int_{\R} \biggl[ \int_{\R} \int_{\R}
\biggl( \int_{x}^{x'} \int_{x}^{x'}
p_{r-t}\bigl(y-v-u\bigr)  p_{r-t}\bigl(y'-v'-u \bigr) 
 \ud y \ud y' \biggr)
\rho(v) \rho(v') \ud v  \ud v' \biggr] \ud u.
\end{split}
\end{equation}
By Gaussian convolution, the  integral is equal to $\int_\R\rho(v)\int_{x}^{x'}\int_{x}^{x'}P_{2(r-t)}\rho(y-y'+v)\ud y \ud y'\ud v$. It
is bounded by 
$\vert x'-x\vert$ or by 
$(r-t)^{-1/2} \vert x'-x\vert^2$. 
By interpolation, it is less than 
$ (r-t)^{-\alpha} \vert x'-x\vert^{1+2\alpha}$. 
Replacing $\alpha$ by $(\alpha+1/2)/2$ in 
\eqref{eq:nouvelle:preuve:KPZ:1}
and only in
\eqref{eq:nouvelle:preuve:KPZ:1}
 (which is always possible since $\alpha$ can be chosen as close as $1/2$ as needed), we deduce that
\begin{equation}
\label{eq:nouvelle:preuve:KPZ:10}
\int_{\R}
\vert {\mathcal K}_{t}^2(r,u) \vert^2 \ud u 
\leq 
C \kappa_{\rho}^2 a^{2\chi}
\vert x'-x\vert^{ 1+ 2 \alpha}
(r-t)^{-1+(1/2-\alpha)/2}.
\end{equation}
We now reproduce the same analysis,
with $\partial_{t} {\mathcal K}^2_{t}(r,u)$
instead of 
${\mathcal K}^2_{t}(r,u)$. 
Since $\vert \partial_{t} p_{r-t}(x) \vert \leq c (r-t)^{-1} p_{c(r-t)}(x)$, 
this amounts to replace 
$(r-t)^{-1+(\alpha+1/2)/2}$
by 
$(r-t)^{-3+(\alpha+1/2)/2}$ in the above computation, so that, for $t \leq t'$, 
\begin{equation}
\label{eq:nouvelle:preuve:KPZ:2}
\int_{\R}
\vert \partial_{t} {\mathcal K}^2_{t}(r,u) \vert^2 \ud u\leq 
C \kappa_{\rho}^2 a^{2\chi}
\vert x'-x\vert^{ 1+ 2 \alpha}
(r-t)^{-3+(1/2-\alpha)/2}.
\end{equation}
The analysis of ${\mathcal K}^1$ may be handled in the same way. 
It is actually much easier since 
$\int_{r}^T [ \partial_x^2P_{s-t}Y^{\rho,T_0}_s(x') -\partial_x^2P_{s-t}Y^{\rho,T_0}_s(x)] \ud s$ has the same structure as $Z^{\rho,T}_{t}(x')-
Z^{\rho,T}_{t}(x)$ and can be bounded by $C \kappa_{\rho} a^{\chi} \vert x'-x\vert^{\alpha}$. 
We thus deduce that \eqref{eq:nouvelle:preuve:KPZ:10}
and \eqref{eq:nouvelle:preuve:KPZ:2}
hold true with ${\mathcal K}^2$ replaced by
${\mathcal K}^1-{\mathcal K}^2$. 
By
Lemma
\ref{lem:reg:noyau} with $\epsilon=(1/2-\alpha)/2$, 
\begin{equation}
\label{eq:KPZ:nouvelle:preuve:14}
\E \bigl[ \sup_{0 \leq t \leq T \leq T_{0}}
\vert {\mathscr I}_{t}^{\rho,(2),T}(x,x')\vert^p \bigr]^{1/p}
\leq c_{p} {\mathbb E}[\vert \kappa_{\rho} \vert^{p}]^{1/p} a^{2\chi} \vert x'-x\vert^{1/2+\alpha}, 
\end{equation}
where $c_{p}$ is a constant independent of $\rho$. 
Repeating the analysis,  \eqref{eq:KPZ:nouvelle:preuve:14}
also holds when $\rho$ is replaced by $\delta_{0}$. 

\textit{Fourth step.}
The goal is to prove an analog of
\eqref{eq:KPZ:nouvelle:preuve:14}, but for the difference 
${\mathscr I}_t^{\rho,(2),T}(x,x')
-{\mathscr I}_t^{\delta_{0},(2),T}(x,x')$. 
Letting
$\kappa_{\rho,\delta_{0}} := \kappa_{(1/2+\alpha)/2,\chi}(Y^{\rho,T_0}
- Y^{T_0})$
and $\| \rho \|_{1} := \int_{\R} \vert v \vert \rho(v) \ud v$, we claim 
\begin{equation}
\label{eq:KPZ:nouvelle:preuve:24}
\begin{split}
&\E \bigl[ \sup_{0 \leq t \leq T \leq T_{0}}
\vert {\mathscr I}_t^{\rho,(2),T}(x,x')
-{\mathscr I}_t^{\delta_{0},(2),T}(x,x') \vert^p \bigr]^{1/p}
\\
&\hspace{15pt} \leq c_{p} a^{2\chi}
\Bigl( 
\E[\vert \kappa_{\rho,\delta_{0}} \vert^p]^{1/p}
+ 
\| \rho\|_{1}^{(1/2-\alpha)/2}
\E[\vert  \kappa
 \vert^p]^{1/p} \Bigr)
 \vert x'-x\vert^{1/2+\alpha}.
 \end{split}
 \end{equation}
 The proof is as follows. By bilinearity of the cross-integral, 
${\mathscr I}^{\rho,(2),T} -
{\mathscr I}^{\delta_0,(2),T}$ 
reads as the sum of two terms of the 
same type as 
${\mathscr I}_{t}^{\rho,(2),T}$
but each involving a modification in
the definition \eqref{eq:K:1:2} of 
${\mathcal K}_{t}^1(r,u)$
and ${\mathcal K}_{t}^2(r,u)$. 
The first modification consists in replacing $Y^{\rho,T_0}$
by $Y^{\rho,T_0}-Y^{T_0}$
and the second one in replacing 
$Y^{\rho,T_0}$ by $Y^{T_0}$
 and then $P_{r-t}\rho$ by $P_{r-t}\rho-p_{r-t}$ (or equivalently $\rho$
 by $\rho - \delta_{0}$). 
The first 
modification contributes
for $c_{p} a^{2\chi} \E[ \vert \kappa_{\rho,\delta_{0}} \vert^p]^{1/p} 
\vert x'-x\vert^{1/2+\alpha}$ in \eqref{eq:KPZ:nouvelle:preuve:24}
(compare with \eqref{eq:KPZ:nouvelle:preuve:14}). 
Concerning the second modification, when $\rho$ is replaced by $\rho-\delta_{0}$ in
\eqref{eq:nouvelle:preuve:KPZ:1}, the 
quintuple integral becomes (after convolution in $u$)
\begin{align*}
&\int_{\R} \int_{\R}
 \int_{x}^{x'}
\int_{x}^{x'}
\Bigl(
p_{2(r-t)} \bigl( y - y' - (v-v') \bigr)
- p_{2(r-t)} \bigl( y - y'-v'\bigr)
\Bigr)
 \ud y \ud y' \rho(v) \rho(v') \ud v \ud v'\\
&+\int_{\R}
 \int_{x}^{x'}
\int_{x}^{x'}
\Bigl(
p_{2(r-t)} \bigl( y - y'  \bigr)-p_{2(r-t)} \bigl( y - y' -v \bigr)
\Bigr)
 \ud y \ud y' \rho(v)\ud v.
\end{align*}

Since
$\vert \partial_{x} p_{r-t}(y) \vert 
\leq c (r-t)^{-1/2} p_{c(r-t)}(y)$, 
the integrals on the square 
$[x,x'] \times [x,x']$ are bounded by 
$(r-t)^{-\alpha} \vert x' -x \vert^{1+2\alpha}$
as in the proof of \eqref{eq:nouvelle:preuve:KPZ:10}
and  by 
$C(r-t)^{-\alpha-1/2}
\vert x' -x \vert^{1+2\alpha} \vert v\vert$. 
By interpolation, it is less
than $C(r-t)^{-\alpha-\eta/2}
\vert x' -x \vert^{1+2\alpha} \vert v\vert^\eta$,
for any $\eta \in (0,1)$. 
Choosing $\eta=(1/2-\alpha)/2$, 
the right-hand side in 
\eqref{eq:nouvelle:preuve:KPZ:10}
becomes 
$C \kappa^2 a^{2\chi}
\vert x'-x\vert^{ 1+ 2 \alpha}(r-t)^{-1+(1/2-\alpha)/4} 
\| \rho\|_{1}^{(1/2-\alpha)/2}$
(with $\kappa:=\kappa_{(1/2+\alpha)/2,\chi}(Y^{T_0})$).  
Similarly, 
the right-hand side in \eqref{eq:nouvelle:preuve:KPZ:2}
becomes $C \kappa^2 a^{2\chi}
(r-t)^{-3+(1/2-\alpha)/4}
\vert x'-x\vert^{1+2\alpha} 
\| \rho\|_{1}^{(1/2-\alpha)/2}$.
Playing the same game with ${\mathcal K}^1$, we get \eqref{eq:KPZ:nouvelle:preuve:24}.

\textit{Fifth step.}
We now replace $\rho$ by $n \rho(n \cdot)$.
From the second step, 
we know that we can find a constant $c_{\rho}'$
such that $\kappa_{n \rho (n \cdot)} \leq c_{\rho}' \kappa$
for any $n \geq 1$. 
Similarly,
for any $\alpha'>(1/2+\alpha)/2$, we can find a deterministic constant $c$ such that 
\begin{equation*}
\begin{split}
\kappa_{\rho,\delta_{0}}
&\leq c \bigl[\kappa_{\alpha',\chi}(Y^{\rho,T_0}) + 
\kappa_{\alpha',\chi}(Y^{T_0}) \bigr]^{(1/2+\alpha)/(2\alpha')}
\bigl[ \sup_{a \geq 1}
 \bigl(
\|Y^{\rho,T_0}-Y^{T_0} \|_{\infty}^{[-a,a]} /a^\chi \bigr) \bigr]^{1-(1/2+\alpha)/(2\alpha')}.
\end{split}
\end{equation*}
Now,  
$ \sup_{a \geq 1}
(\|Y^{\rho,T_0}-Y^{T_0} \|_{\infty}^{[-a,a]} /a^\chi) \leq c_{\rho}' \kappa  \| \rho\|_{1}^{(1/2+\alpha)/2}$, 
for a possibly new value of the constant $c_{\rho}'$. It remains true with the same
constant $c_{\rho}'$ when $\rho$ is replaced by $n \rho(n \cdot)$, so that 
$\E[\vert \kappa_{n \rho(n \cdot),\delta_{0}}
\vert^p]^{1/p}
\leq C_{p} \| n \rho(n \cdot) \|_{1}^{\eta}
= C_{p} n^{-\eta} \| \rho \|_{1}^\eta$, for some $\eta >0$. 
Modifying if necessary the value of $\eta$, we deduce from
 \eqref{eq:KPZ:nouvelle:preuve:24} that 
 \begin{equation*}
\begin{split}
\E \bigl[ \sup_{0 \leq t \leq T \leq T_{0}}
\vert {\mathscr I}_t^{n \rho(n\cdot),(2),T}(x,x')
-{\mathscr I}_t^{\delta_{0},(2),T}(x,x') \vert^p \bigr]^{1/p}
 &\leq 
C_{p} a^{2\chi} n^{-\eta}
\vert x'-x\vert^{1/2+\alpha}.
 \end{split}
 \end{equation*}

\textit{Conclusion.}
We let $\Gamma = \sup_{n \geq 1}
[n^{\eta/2}
\sup_{0 \leq t \leq T \leq T_{0}}
\vert {\mathscr I}_t^{n \rho(n\cdot),(2),T}(x,x')
-{\mathscr I}_t^{\delta_{0},(2),T}(x,x') \vert]$. 
We have 
$\E[ \vert \Gamma \vert^p]
\leq \sum_{n \geq 1}
n^{p \eta/2}
\E[
\vert {\mathscr I}_t^{n \rho(n \cdot),(2),T}(x,x')
-{\mathscr I}_t^{\delta_{0},(2),T}(x,x') \vert^p]$, which is less than
$$C_{p}^p a^{2p\chi} \vert x'-x\vert^{p(1/2+\alpha)}
\sum_{n \geq 1}
n^{-\eta p/2}= 
C_{p}' a^{2p\chi} \vert x'-x\vert^{p(1/2+\alpha)}$$
when $\eta p >2$. 
We deduce, for $x,x' \in [-a,a]$,
\begin{equation*}
\begin{split}
\E \bigl[ 
\sup_{n \geq 1} 
\sup_{0 \leq t \leq T \leq T_{0}}
\bigl( n^{\eta/2}
\vert {\mathscr I}_t^{n \rho(n \cdot),(2),T}(x,x')
-{\mathscr I}_t^{\delta_{0},(2),T}(x,x')\vert \bigr)^p \bigr]^{1/p}
 &\leq C_{p} a^{2\chi}
 \vert x'-x\vert^{1/2+\alpha}. 
 \end{split}
 \end{equation*}
We aim at applying 
Theorem
\ref{thm:kolmo}
with $L=(n,t,T)$ and $R_{L}(x,x')
= n^{\eta/2} ( {\mathscr I}_t^{n \rho(n\cdot),(2),T}(x,x')
-{\mathscr I}_t^{\delta_{0},(2),T}(x,x') )$,
the issue being to control $R_{L}(x,x')+R_{L}(x',x'') - R_{L}(x,x'')$. 
From \eqref{eq:mathcal:I},
\begin{align*}
&{\mathscr I}_{t}^{ n \rho(n\cdot),(2),T}(x,x') 
+ {\mathscr I}_{t}^{n \rho(n\cdot),(2),T}(x',x'')
- {\mathscr I}_{t}^{n \rho(n\cdot),(2),T}(x,x'') \\
&\quad= \int_t^T\bigl(\partial_x^2P_{s-t}Y^{n \rho(n \cdot),T_0}_s(x') -\partial_x^2P_{s-t}Y^{n \rho(n \cdot),T_0}_s(x)\bigr)\left(Y_{t}^{n \rho(n \cdot),s}(x'')- Y_{t}^{n \rho(n \cdot),s}(x')\right)\ud s.
\end{align*}
All the terms converge in $L^p$ and the same relationship holds 
with $n \rho(n \cdot)$ replaced by $\delta_{0}$. Making the difference 
between the relationships with 
$n \rho(n\cdot)$ and $\delta_{0}$, we get an explicit expression 
for $R_{L}(x,x')+R_{L}(x',x'') - R_{L}(x,x'')$. 
All the terms involved are explicitly controlled. 
By the same method as in the second step, 
$\vert R_{L}(x,x')+R_{L}(x',x'') - R_{L}(x,x'') \vert \leq \zeta a^{2\chi}
\vert x'-x\vert^{1/2+\alpha}$, for a random variable $\zeta$. 
By Theorem 
\ref{thm:kolmo}, 
for any $\chi'>\chi$ and 
$\alpha' \in (0,(1/2+\alpha)/2)$, 
we 
can find a random variable $\zeta'$ such that
$\vert {\mathscr I}_t^{n \rho(n\cdot),(2),T}(x,x')
-{\mathscr I}_t^{\delta_{0},(2),T}(x,x') \vert
\leq \zeta' n^{-\eta/2} a^{2\chi'} \vert x'-x\vert^{\alpha'}$, 
for all $x,x' \in [-a,a]$, $a \geq 1$. 
%
As $ {\mathscr I}_t^{n,T}= {\mathscr I}_t^{n \rho(n\cdot),(1),T}+ {\mathscr I}_t^{n \rho(n\cdot),(2),T}$ and $ {\mathscr I}_t^{T}= {\mathscr I}_t^{\delta_0,(1),T}+ {\mathscr I}_t^{\delta_0,(2),T}$, this last bound combined with the conclusion of the second step prove that the assumptions of Lemma \ref{lem:rough path} are satisfied. 
\end{proof}

\section{Connection with the KPZ equation}
\label{sec:KPZ}
KPZ equation was introduced by Kardar, Parisi and Zhang in \cite{kar:par:zha:86} in order to model the growth of a random surface 
subjected to three phenomena: a diffusion effect, a lateral growth and a random deposit. It has the formal shape:
\begin{equation}
\label{eq:KPZ}
\partial_{t} h_{t}(x) = \tfrac12 \partial_{x}^2 h_{t}(x) + \tfrac12 \vert \partial_{x} h_{t}(x) \vert^2 + \dot{\zeta}(t,x),
\end{equation}
with $0$ as initial condition,
where $\dot{\zeta}$ is a time-space white noise (that is the time-space derivative of a Brownian sheet, defined on $(\Xi,{\mathcal G},{\mathbf P})$ as discussed in 
Theorem \ref{thm:main:thm:polymer}). Unfortunately, it is ill-posed 
since the gradient does not exist as a true function, but
as a distribution only. 
 
 Two strategies have been developed so far to give a sense to \eqref{eq:KPZ}. The first one goes back to \cite{ber:gia:97}
 and consists in linearizing the equation by means of the so-called Hopf-Cole exponential transformation. The second approach is due to Hairer \cite{hai:13} in the case when $x$ is restricted to the torus
 (in which case $\zeta$ is defined accordingly). Therein, the key point is to solve second-order PDEs driven by a distributional first-order term by means of rough paths theory, which is precisely the strategy we used in Section \ref{sec:pde} to solve \eqref{eq:mildpde}. The two
 interpretations coincide but the resulting solution solves a \textit{renormalized} version of \eqref{eq:KPZ}, which writes (in a formal sense) as \eqref{eq:KPZ} with an additional `$-\infty$' in the right-hand side. 
The normalization must be understood as follows: When mollifying the noise (say $\dot{\zeta}$ into $\dot{\zeta}^n)$, 
Eq. \eqref{eq:KPZ} admits a solution, denoted by $h^n$, but the sequence $(h^n)_{n \geq 1}$ is
not expected to converge. To make it converge to the solution of \eqref{eq:KPZ}, some `counterterm' must be subtracted to the right-hand side of \eqref{eq:KPZ}: This counterterm is a constant $\gamma^{n}$ depending upon $n$, which tends to $\infty$ with $n$, thus explaining the 
additional `$-\infty$'. 


\subsection{Polymer measure on the torus} 
Below, we make use of the framework defined in \cite{hai:13}. This imposes two restrictions. 
The first one is that $\zeta$  
has to be defined on $[0,\infty) 
\times {\mathbb S}^1$, where ${\mathbb S}^1$ is the 1d torus, 
which means that $\dot{\zeta}$ is a cylindrical Wiener process on 
$L^2({\mathbb S}^1)$.  
The second one is that
the Fourier transform 
$\hat{\rho}$
of the kernel $\rho$ used to mollify the noise has to be even, compactly supported, smooth and non-decreasing on $[0,\infty)$, in which case $\rho$ is defined from its Fourier transform. 
In particular, $\rho$ has polynomial decay of any order, but may not be positive.    
The mollified version $\zeta^n$ of $\zeta$ is given by 
$\zeta^n(t,x) = \int_{0}^t \int_{\R} n\rho(n(x-y)) \ud{\zeta}(s,y)$, 
with the convention that 
$\int_{0}^t \int_{\R} \varphi(s,y) \ud{\zeta}(s,y) 
= \sum_{k \in {\mathbb Z}}
\int_{0}^t \int_{{\mathbb S}^1} \varphi(s,y+k) \ud \zeta (s,y)$
if $\int_{0}^t \int_{{\mathbb S}^1} \vert \sum_{k \in {\mathbb Z}}
\varphi(s,y+k) \vert^2 \ud s \ud y < \infty$.

Given $T_{0} >0$ and $n \geq 1$, we introduce the (random) polymer measure:
\begin{equation*}
\frac{d {\mathbb Q}_{{\zeta}^n}}{d {\mathbb P}}
\sim 
\exp \biggl( 
\int_{0}^{T_{0}}
\int_{\R} n \rho \bigl( n(B_{T_{0}-t} - y ) \bigr) \ud \zeta(s,y)
\biggr),
\end{equation*}
where $(B_{t})_{0 \leq t \leq T_{0}}$ is a Brownian motion under ${\mathbb P}$ ($(\Omega,{\mathcal A},\P)$ being distinct of $(\Xi,{\mathcal G},{\mathbf P})$),
the symbol $\sim$ indicating that the right-hand side is normalized in such a way that ${\mathbb Q}_{{\zeta}^n}$ is a probability. 
The polymer measure describes the law of a continuous random walk evolving in 
the periodic random environment $\zeta^n$. 
The factor 
$\int_{0}^{T_{0}}
\int_{\R} n \rho ( n(B_{T_{0}-t} - y )) \ud \zeta(s,y)$
is sometimes written $\int_{0}^{T_{0}}
\dot{\zeta}^n(t,B_{T_{0}-t}) \ud t$ or 
$\int_{0}^{T_{0}}
\dot{\zeta}^n(T_{0}-t,B_{t}) \ud t$.

By 
applying
It\^o-Wentzell formula to $(h^n_{T_{0}-t}(B_{t}))_{0 \leq t \leq T}$, we obtain, 
${\mathbf P} \otimes {\mathbb P}$ a.s.,
\begin{equation*}
h^n_{T_{0}}(0) + \int_{0}^{T_{0}}
\dot{\zeta}^n(T_{0}-t,B_{t}) dt - \gamma^{n}
=  \int_{0}^T \partial_{x} h^n_{T_{0}-s}(B_{s}) \ud B_{s}
- \frac12  
\int_{0}^T \vert \partial_{x} h^n_{T_{0}-s}(B_{s}) \vert^2 \ud s,
\end{equation*}
proving, by Girsanov Theorem, that, ${\mathbf P}$ a.s., the dynamics of $(B_{t})_{0 \leq t \leq T_{0}}$ 
under ${\mathbb Q}_{{\zeta}^n}$ satisfy the SDE \eqref{eq:18:1:1} with $Y_{t}(x) = h^n_{T_{0}-t}(x)$ ($h^n_{T_{0}}(0)$ and $\gamma^n$ are unnoticeable in the definition of the polymer measure as they are hidden in the normalization constant of the right-hand side). 
%

The main challenging question is to define the limit of 
${\mathbb Q}_{\zeta^n}$ rigorously.
The following theorem provides a new result in that direction:
\begin{theorem}
\label{thm:polymer:1}
Consider the solution to the (normalized) KPZ equation 
\eqref{eq:KPZ} with $0$ as initial solution and let 
$Y_{t}(x):=h_{T_{0}-t}(x)$, for $(t,x) \in [0,T_{0}] \times \mathbb{T}$.  
Then, we can find an event $\Xi^\star$, with ${\mathbf P}(\Xi^\star)=1$, 
such that, for any realization in $\Xi^\star$ and any $T \in [0,T_{0}]$, 
the pair $W^T = (Y,Z^T)$, with $Z^T$ given by \eqref{eq:Z}, 
may be lifted into a geometric rough path 
${\boldsymbol W}^T =(W^T,\W^T)$ satisfying the conclusions of Lemma \ref{lem:rough path}, 
with $(Y_{t}^n(x):=h^n_{T_{0}-t}(x))_{n \geq 1}$, for $(t,x) \in [0,T_{0}] \times \mathbb{T}$, 
as approximation sequence. 

Moreover, for any realization $\xi \in \Xi^\star$,
${\mathbb Q}_{{\zeta}^n}$ converges towards the law (on $\Omega$) of the solution $(X_{t})_{0 \leq t \leq T_{0}}$ to 
\eqref{eq:18:1:1}
when driven by the trajectory $Y$ associated with $\xi$.
The limit law is independent of the choice of $\rho$ in the construction of $h^n$
and reads as a rigorous interpretation of the (a priori ill-defined) polymer measure ${\mathbb Q}_{\zeta}
\sim \exp(\int_{0}^{T_{0}} \dot{\zeta}(T_{0}-t,B_{t}) \ud t) \cdot {\mathbb P}$
on the canonical space ${\mathcal C}([0,T_{0}],\mathbb{T})$. 
\end{theorem}

\begin{proof}
It suffices to check the assumption of Lemma 
\ref{lem:rough path}. 
To this end, recall from \cite[Theorem 1.10]{hai:13} that, ${\mathbf P}$ a.s., $h$ expands as 
$Y^{\bullet} + h^b$, 
where $Y^\bullet$ solves the stochastic heat equation 
for some initial condition $Y^\bullet_{0} \in \cap_{\varepsilon >0}{\mathcal C}^{1/2-\varepsilon}({\mathbb S}^1)$  
and $h^b$ is a 
 continuous 
remainder satisfying $h^b_{t} \in \cap_{\varepsilon >0}{\mathcal C}^{1-\varepsilon}({\mathbb S}^1)$
for any $t>0$ (the associated H\"older constant being uniform 
on any closed interval of $(0,T_{0}]$). 
The point is thus to apply Theorem 
\ref{thm:main:thm:polymer} (which easily extends to ${\mathbb S}^1$) with
\begin{equation*}
Y^{T_0}_{t}(x) := Y^{\bullet}_{T_{0}-t}(x) - 
\bigl[P_{T_{0}-t} Y^{\bullet}\bigr]
(x), 
\quad
Y^{(b)}(t,x) := h^b_{T_{0}-t}(x) +\bigl[P_{T_{0}-t} Y^{\bullet}\bigr]
(x).
\end{equation*} 
The fact that $\rho$ may not be positive is not a problem as we can split it
into $\rho = \rho_{+}- \rho_{-}$ and then check that the results of Section 
\ref{sec:yz} still apply with such a decomposition. 
Clearly, $Y^{T_0}$ solves the backward 
stochastic heat equation with zero as terminal condition. 
Moreover, for any $T< T_{0}$ and any $\alpha_{b}<1$, \cite[Theorem 1.10]{hai:13} 
 ensures that, for ${\mathbf P}$ a.e. realization in $\Xi^\star$,  
$\kappa_{\alpha_{b},0}((Y^{(b)}_{t})_{0 \leq t \leq T})$ is finite
(here we can choose $\chi_{b}=0$ as we work on ${\mathbb S}^1$). 
Then, with the same notation as above, we know 
from \cite{hai:13} that, 
almost surely on $\Xi^\star$, 
$\|h^n_{T_{0}-\cdot} - Y^{n \rho(n \cdot),T_0} 
- Y^{(b)}  
\|_{0,\alpha_{b}}^{[0,T] \times {\mathbb S}^1}$
converges to $0$ as $n$ tends to $\infty$.
By Theorems \ref{thm:localmart}
(and its proof) 
and
\ref{thm:main:thm:polymer}, we deduce that, a.s. on $\Xi^\star$,
the solution to the SDE \eqref{eq:18:1:1}
on $[0,T]$, when driven by 
$h^n_{T_{0}-\cdot}$, converges to the solution 
of \eqref{eq:18:1:1} driven by $Y$. 
This completes the proof on any $[0,T] \subset [0,T_{0})$. 

In order to get the convergence on the entire $[0,T_{0}]$, we must 
revisit \cite{hai:13}
in order to control the H\"older norm (in $x$) of 
$Y^{(b)}_{t}$ uniformly in $t \in [0,T_{0}]$.  
The technical issue is that, in \cite{hai:13}, the KPZ equation
is solved by means of a fixed point argument 
that allows for irregular initial conditions. 
As the initial condition may be irregular, 
solutions exhibit a strong blow-up at the boundary, see 
\cite[Proposition 4.3]{hai:13}. In \cite{hai:13}, 
$h$ is split into $h_{t}(x) =u_{t}(x) + h^\star_{t}(x)$, where $h^{\star}_{t}(x) = \sum_{\tau \in \bar{{\mathcal T}}}
Y^{\tau}_{t}(x)$, $\bar{\mathcal T}$ denoting a finite collection of trees containing the root tree $\bullet$.
For $\tau \in \bar{\mathcal T} \setminus \{ \bullet\}$, $Y^\tau$
is continuous and, for any $\varepsilon >0$, $\| Y^\tau_{t} \|_{1-\varepsilon}$ is finite, uniformly in $t \in [0,T_{0}]$. The remainder $u$ is investigated  
through its derivative $v : [0,T_{0}] \times {\mathbb S}^1 \ni (t,x) \mapsto 
v_{t}(x) = \partial_{x} u_{t}(x)$, 
defined as solution of (see \cite[Section 4]{hai:13} for the notations): 
\begin{equation}
\label{eq:point:fixe:KPZ}
v_{t}(x) = P_{t}v_{0}(x) + {\mathcal M} \bigl[ G(v_{\cdot},\cdot) \bigr]_{t}
+ \partial_{x} \int_{0}^t P_{t-s} F(v_{s},s) ds,
\end{equation}
for some functionals ${\mathcal M}$, $G$ and $F$. 
Our goal here is to expand $h_{t}$ as $h_{t} = [u_{t} - P_{t}u_{0}]
+ [h_{t}^\star + P_{t} u_{0}]$ and to investigate the regularity 
of $u_{t} - P_{t} u_{0}$ directly by taking benefit of the fact that $h_{0}=0$. 
Letting $t=0$, we notice that $u_{0}=-h_{0}^\star$ 
so that $h_{t} = [u_{t} - P_{t} u_{0} ] + 
[h_{t}^\star - P_{t} h_{0}^\star]$. 
We also notice that 
$h_{t}^\star - P_{t} h_{0}^\star$ may be written 
$Y_{t}^\bullet - P_{t} Y_{0}^\bullet + 
\sum_{\tau \in \bar{\mathcal T} \setminus \{\bullet\}}
Y_{t}^{\tau} - P_{t} Y_{0}^\tau$. 
Here,
$Y^{\bullet} - P_{t} Y^{\bullet}_{0}$ is our $Y^{T_0}_{T_{0}-\cdot}$ and, for any small $\varepsilon >0$,
$\sum_{\tau \in \bar{\mathcal T} \setminus \{\bullet\}}
Y^{\tau}_{t} - P_{t} Y_{0}^\tau$
has a finite norm in ${\mathcal C}^{1-\varepsilon}({\mathbb S}^1)$,
uniformly in $t \in [0,T_{0}]$, so that  
$h_{t}^\star - P_{t} h_{0}^\star$
has the right decomposition to apply  
Theorems \ref{thm:localmart} 
and
\ref{thm:main:thm:polymer}.
It thus suffices to focus on 
$u_{t} - P_{t} u_{0}$ or, equivalently, on 
$v_{t} - P_{t} v_{0}= \partial_{x} [u_{t} - P_{t} u_{0}]$
in \eqref{eq:point:fixe:KPZ}. 
The main idea is to see 
$\bar{v}_{t}:= v_{t} - P_{t} v_{0}$ as the solution of 
\begin{equation*}
\bar{v}_{t} 
= {\mathcal M} \bigl[ G\bigl( \bar{v}_{\cdot} 
 + P_{\cdot} v_{0},\cdot \bigr) \bigr] +  \partial_{x} \int_{0}^t P_{t-s} F(v_{s},s) ds,
\end{equation*}
with $\bar{v}_{0}=0$. (Note that, in the second term in the right-hand side, the value of $v$ is fixed.) We then make use of the norm $\| \cdot \|_{\star,T}$
defined in \cite[p.597]{hai:13}, but with different parameters 
$\kappa$, $\delta$, $\alpha$, $\beta$ and $\gamma$. 
We choose $\kappa=\varepsilon$ small enough, $\delta = 2 \varepsilon$, 
$\alpha=1/2+2\varepsilon$, $\beta=1/4+ \varepsilon$ and $\gamma=\alpha$, 
which satisfy all the prescriptions 
\cite[Eqs. (76a)-(76g)]{hai:13}.
Following \cite[Eqs.(83a),(83b),(83c),(85)]{hai:13}, we get, for $C,\theta>0$, 
$\| \bar{v} \|_{\star,T} \leq C + C T^{\theta} (\| \bar{v} \|_{\star,T}
+ \| P_{\cdot} v_{0} \|_{\star,T})$, where the derivative of $P_{\cdot} v_{0}$
with respect to the rough path structure 
is $0$. Here $v_{0} = -\partial_{x} h_{0}^\star$
is a distribution in ${\mathcal C}^{-1/2-\varepsilon'}({\mathbb S}^1)$, for $\varepsilon'>0$ as
small as desired. Following \cite[Eq.(82)]{hai:13}, 
$\| P_{\cdot} v_{0}\|_{\star,T} < \infty$. 
We deduce that, for $T$ small enough, 
$\| \bar{v} \|_{\star,T}<\infty$. By \cite[Eq.(73)]{hai:13}, 
we get $\| \bar{v}_{t} \|_{\infty} \leq C t^{-3\varepsilon}$. 
Working at the level of the primitive, we obtain 
$\| u_{t} - P_{t} u_{0} \|_{1-3 \varepsilon}
< \infty$, uniformly in $t \in [0,T]$. 
The fact that $T$ has to be small is not a problem
  since we are interested in the behavior of $h$ near the origin. 
Therefore, 
$h_{t} = Y_{T_{0}-t}^{T_{0}} + 
Y^{(b)}_{t}$, with $Y^{(b)}_{t}
=
[u_{t} - P_{t} u_{0} ] + 
[h_{t}^\star - Y^{\bullet}_{t} - P_{t} (h_{0}^\star - Y_{0}^\bullet)]$,
fits the assumptions in 
Theorems \ref{thm:localmart} 
and
\ref{thm:main:thm:polymer}.
The convergence to $0$ of 
$\|h^n_{T_{0}-\cdot} - Y^{n \rho(n \cdot),T_0} 
- Y^{(b)}  
\|_{0,\alpha_{b}}^{[0,T_{0}] \times {\mathbb S}^1}$
(on the whole $[0,T_{0}] \times {\mathbb S}^1$) is handled in the same way. 
\end{proof}

We end up with:
\begin{theorem}
\label{thm:polymer:2}
For ${\mathbf P}$ almost every realization of the environment $\zeta$, 
under the polymer measure ${\mathbb Q}_{\zeta}$ defined in Theorem 
\ref{thm:polymer:1}, 
the canonical path has dynamics of the form
\begin{equation*}
\ud X_{t} = \ud B_{t} + b(t,X_{t},\ud t), \quad t \in [0,T_{0}],
\end{equation*}
in the sense of \eqref{eq:1:10:3}, where $b(t,X_{t},\ud t)$ is
of order $O(\ud t^{3/4-\varepsilon})$, for $\varepsilon$ as small as desired, 
the constant in the Landau notation being random but uniform in $t \in [0,T_{0}]$. 
Moreover, in the expression of $b$ in 
Proposition \ref{prop:drift}, the second term 
can be computed by replacing $(Y,Z^{t+h})$ by 
$(Y^{T_0},Z^{T_0,t+h})$, where $Y^{T_0}$ is the solution of the stochastic heat equation as in 
Theorem \ref{thm:main:thm:polymer}
and $Z^{0,t+h}$ is computed accordingly as in \eqref{eq:Z:i}. 
\end{theorem}

\begin{proof}
The proof is a consequence of Proposition 
\ref{prop:drift}. The reason why the second term in 
the decomposition of $b$ can be simplified follows from the proof of 
Theorem \ref{thm:polymer:1}. Indeed, we know that
$Y$ may be split into $Y^{T_0}+Y^{(b)}$, with $\kappa_{\alpha_{b},0}(Y^{(b)})
< \infty$ for $\alpha_{b}$ close to $1$.
The game is then the same as in 
\eqref{eq:cross:integral:i,j}: for $(i,j) \not = (0,0)$, 
the cross-integrals ${\mathscr I}^{i,j,t+h}$ 
in 
\eqref{eq:cross:integral:i,j}
give a contribution of order $O(h^{3/2-\varepsilon})$
in the computation of $b$, which can be forgotten 
at the macroscopic level. 
\end{proof}

\section*{Appendix}
\begin{lemma}
\label{lem:regularisation}
Given a sequence of smooth paths
$(Y^n)_{n \geq 1}$ such that, for some $T_{0}>0$ and any $T \in [0,T_{0}]$, the 
sequence  
$(W^{n,T}=(Y^n,Z^{n,T}))_{n \geq 1}$
satisfies
 the assumption of Proposition \ref{mildsolution:approx},
 with  $\kappa =\sup_{0 \leq T \leq T_{0}} \sup_{n \geq 1}
\kappa_{\alpha,\chi}((W_{t}^{n,T},\W_{t}^{n,T})_{0 \leq t <T}) < \infty$,
then,
  we can assume that, 
 for any $n \geq 1$, $Y^n$ has bounded derivatives on the whole space. 
\end{lemma}

\begin{proof}
For $N \in {\mathbb N} \setminus \{0\}$, we consider  
a smooth function $\varphi^N : \R \rightarrow [0,1]$, 
symmetric, equal to 
$1$ on $[0,N]$ and to $0$ on $[2N,+\infty)$, non-increasing 
on $[N,2N]$, satisfying $\| \ud^p \varphi^N/\ud x^p  \|_{\infty} \leq c_{p}/N^p$
for some $c_{p} \geq 1$, independent of $N$, for any integer $p \geq 1$.  
Then, we let $Y^{n,N}_{t}(x)=Y^n_{t}(0)+\int_{0}^x \varphi^N(y) \partial_{x} Y^n_{t}(y) 
\ud y$ and, for a given $T>0$, we define 
$Z^{n,N,T}$, $W^{n,N,T}$ and $\W^{n,N,T}$ accordingly. 
 
For a given $n$,  
$(Y^{n,N})_{N \geq 1}$ (resp. $\partial_{x}^p Y^{n,N}$ for an integer $p \geq 1$)
 converges towards 
$Y^n$ (resp. $\partial_{x}^p Y^n$) as $N$ tends to $\infty$, uniformly in $x$ in compact sets and in $t \in [0,T)$. 
Using the representations of $Z_{t}^{n,N,T}$ and $Z^{n,T}_{t}$, see 
\eqref{eq:Z}, the same holds true
for the sequence $(Z^{n,N,T})_{N \geq 1}$
(resp. $({W}^{n,N,T})_{N \geq 1}$)
with $Z^{n,T}$ (resp. $W^{n,T}$) as limit path. Hence, 
$({\W}_{t}^{n,N,T})_{N \geq 1}$ converges towards $\W_{t}^{n,T}$ 
in norm $\| \cdot \|_{\alpha}$,
uniformly in $t \in [0,T)$. 
Using the same notation as  in Proposition \ref{mildsolution:approx},  
$(\|(W^{n,N,T}-W^{n,T},\W^{n,N,T} - \W^{n,T})\|_{0,\alpha}^{[0,T) \times {\mathbb I}})_{N \geq 1}$ tends to 
$0$ as $N$ tends to $\infty$. Therefore, we can find a sequence 
$(N_{n})_{n \geq 1}$ such that
$\|(W^{n,N_{n},T}-W^{n,T},\W^{n,N_{n},T} - \W^{n,T})\|_{0,\alpha}^{[0,T) \times {\mathbb I}}$, and thus 
$\|(W^{n,N_{n},T}-W^{T},\W^{n,N_{n},T} - \W^{T})\|_{0,\alpha}^{[0,T) \times {\mathbb I}}$,
tend to $0$ as $n$ tends to $\infty$, which fits (1) 
in Proposition \ref{mildsolution:approx}.

%

We now discuss (2) in Proposition \ref{mildsolution:approx}. 
We start with the H\"older estimate of $Y^{n,N}_{t}$. For $0\leq x \leq y \leq a$,
with $a \geq 1$, 
the second mean-value theorem yields  
${Y}^{n,N}_{t}(y) - 
{Y}^{n,N}_{t}(x) = \varphi_{N}(x) [Y^{n}_{t}(y')-Y_{t}^{n}(x)]$, for $y' \in [x,y]$. 
We deduce that $\vert Y^{n,N}_{t}(y) - Y^{n,N}_{t}(x) \vert 
\leq \kappa a^{\chi} \vert y - x \vert^{\alpha}$.
The same holds true when $-a \leq y \leq x \leq 0$. Changing $\kappa$ into $2\kappa$, 
we get the same result for any $x,y \in [-a,a]$. 
By
Lemma \ref{lemmaZ}, the bound 
$\vert 
{Z}^{n,N}_{t}(x) -
{Z}^{n,N}_{t}(y) 
\vert \leq \kappa a^{\chi}
\vert x-y \vert^{\alpha}$ follows. 

We finally discuss the regularity of the second-order integrals. 
As discussed in Section 
\ref{sec:yz}, it suffices to focus on the cross-integral
$\int_{x}^y 
[
{Z}^{n,N}_{t}(z) - 
{Z}^{n,N}_{t}(x)] \ud 
{Y}^{n,N}_{t}(z)$. 

By \eqref{eq:Z}, $\partial_{t} Z_{t}^{n,N,T}(x) + (1/2) 
\partial^2_{x} Z_{t}^{n,N,T}(x) = - \partial^2_{x} [Y_{t}^{n,N}](x)
= - \partial_{x} [ \varphi^N \partial_{x} Y^{n}_{t}](x)$. 
Similarly,  $\partial_{t} Z_{t}^{n,T}(x) + (1/2) 
\partial^2_{x} Z_{t}^{n,T}(x) = - \partial^2_{x} [Y_{t}^{n}](x)$.
Therefore,
\begin{equation*}
\begin{split}
&\partial_{t} \bigl[ Z_{t}^{n,N,T} - \varphi^N Z_{t}^{n,T}
\bigr] + \tfrac12 
\partial^2_{x} \bigl[ Z_{t}^{n,N,T}
-\varphi^N Z_{t}^{n,T} \bigr] 
= - \varphi_{N}' \partial_{x} \bigl[Y^{n}_{t} + Z^{n,T}_{t} \bigr]
- \tfrac12 \varphi_{N}'' Z^{n,T}_{t},
\end{split}
\end{equation*}
with $Z_{T}^{n,N,T} - Z_{T}^{n,T}=0$. Therefore, 
integrating against $p_{s-t}$ and then integrating by parts,
\begin{equation}
\label{eq:mollification}
\begin{split}
Z_{t}^{n,N,T}(x) &- \varphi^N(x) Z_{t}^{n,T}(x)
= \int_{t}^T \int_{\R} \partial_{x} p_{s-t}(x-y) 
\varphi_{N}'(y) \bigl[Y^{n}_{t} + Z^{n,T}_{t} \bigr](y) \ud y
\ud s 
\\
&- \int_{t}^T \int_{\R} p_{s-t}(x-y) 
\varphi_{N}''(y)  \Bigl( \bigl[Y^{n}_{t} + Z^{n,T}_{t} \bigr](y)
+\frac12  Z^{n,T}_{t}(y) \Bigr)
\ud y \ud s.
\end{split}
\end{equation}
The aim is to differentiate both sides of the equality in order to estimate the 
derivative of the left-hand side. In order to bound the derivative of the right-hand side, we discuss the H\"older constant of the integrands right above. We have $\vert \varphi_{N}'(y) Y^{n}_{t}(y) - 
\varphi_{N}'(x) Y^{n}_{t}(x) \vert \leq c_{2} \vert Y^n_{t}(x) \vert
\vert y - x \vert/N^2 + (c_{1}\kappa/N) a^{\chi} \vert y -x \vert^{\alpha}$, 
for $x,y \in [-a,a]$, $a \geq 1$. Modifying $\kappa$ if necessary 
$\vert Y^n_{t}(x) \vert \leq \kappa a^{1+\chi}$. Therefore, we can find a constant 
$C \geq 0$ such that $$\vert \varphi_{N}'(y) Y^{n}_{t}(y) - 
\varphi_{N}'(x) Y^{n}_{t}(x) \vert \leq C a^{1+\chi}
\vert y - x \vert/N^2 + C a^{\chi} \vert y -x \vert^{\alpha}/N.$$
Since 
$\varphi_{N}' =0$ outside $[-2N,2N]$, we can always 
assume that $x,y \in [-2N,2N]$ 
(by projecting $x$ and $y$ onto $[-2N,2N])$
and thus that $a \leq 2N$. Then, the left-hand side is less than 
$C a^{\chi} \vert y -x \vert^{\alpha}/N$. 
Using a similar argument for all the other terms of the same type in the right-hand side
of \eqref{eq:mollification}, 
we deduce that the left-hand side in \eqref{eq:mollification} is differentiable and 
that $\vert \partial_{x} [
Z_{t}^{n,N,T} - \varphi^N Z_{t}^{n,T}](x) \vert \leq C a^{\chi}/N$, when $x \in [-a,a]$, $a \geq 1$. By integration by parts,
\begin{equation*}
\begin{split}
&\biggl\vert \int_{x}^y 
\Bigl( \bigl[
Z_{t}^{n,N,T} - \varphi^N Z_{t}^{n,T}\bigr](z)
- \bigl[
Z_{t}^{n,N,T} - \varphi^N Z_{t}^{n,T}\bigr](x)
\Bigr)
 \ud Y^{n,N}_{t}(z) 
\biggr\vert
\\
&= \biggl\vert \int_{x}^y 
\bigl[ÊY^{n,N}(y) - Y^{n,N}(z) \bigr]  
\partial_{x} [
Z_{t}^{n,N,T} - \varphi^N Z_{t}^{n,T}](z) \ud z 
\biggr\vert \leq C \frac{a^{2\chi}\vert x-y \vert^{1+\alpha}}{N}.
\end{split}
\end{equation*}
Since $\partial_{x}Y^{n,N}_{t}(z)=0$ when $\vert z \vert \geq 2N$, 
we can always assume that $x,y \in [-2N,2N]$ and $a \leq 2N$. We deduce that the term in the first line is less than $C a^{2\chi} \vert x-y \vert^{2\alpha}/N^{\alpha}$.
To end up the analysis, it thus suffices to prove that 
\begin{equation*}
\begin{split}
&\biggl\vert \int_{x}^y 
 \bigl(
 \varphi^N(z) Z_{t}^{n,T}(z)
 -  \varphi^N(x) Z_{t}^{n,T}(x) \bigr)
 \ud Y^{n,N}_{t}(z) 
\biggr\vert
\leq C a^{2\chi}\vert x-y \vert^{2\alpha}.
\end{split}
\end{equation*}
Since $\partial_{x}
Y^{n,N}_{t}(z) = \varphi^N(z) \partial_{x} Y^n_{t}(z)$, we can use again the second mean-value theorem to handle 
$\int_{x}^y 
 \varphi^N(z)[ Z_{t}^{n,T}(z) - Z_{t}^{n,T}(x)]
 \ud Y^{n,N}_{t}(z)
 = \int_{x}^y 
 (\varphi^N(z))^2[ Z_{t}^{n,T}(z) - Z_{t}^{n,T}(x)]
 \ud Y^{n}_{t}(z)$. Therefore, it suffices to focus on 
$Z_{t}^{n,T}(x) \int_{x}^y 
[ \varphi^N(z) -  \varphi^N(x)] 
 \ud Y^{n,N}_{t}(z)$. By integration by parts,
\begin{equation*}
\begin{split}
\biggl\vert Z_{t}^{n,T}(x) \int_{x}^y 
[ \varphi^N(z) -  \varphi^N(x)] 
 \ud Y^{n,N}_{t}(z)
 \biggr\vert &=
 \biggl\vert Z_{t}^{n,T}(x) \int_{x}^y 
\bigl[Y^{n,N}_{t}(y)-Y^{n,N}_{t}(z)
\bigr]
( \varphi^N)'(z) \ud z \biggr\vert,
\end{split}
\end{equation*} 
which is less than 
$C a^{2\chi} \vert y -x \vert^{1+\alpha}/N$ (following Lemma \ref{lemmaZ},
$Z^{n,T}_{t}$ satisfies $\vert Z^{n,T}_{t}(x) \vert \leq C a^{\chi}$ --better than the elementary but rough bound $\vert Z^{n,T}_{t}(x) \vert \leq C a^{1+\chi}$--).
Limiting the analysis to the case $a \leq 2N$, we conclude as above.
\end{proof}

\end{document}